\documentclass[a4paper, 12pt]{article}

\usepackage{amsfonts}
\usepackage{amsmath}
\usepackage{amssymb}
\usepackage{amsthm}
\usepackage{mathtools}
\usepackage{enumerate}

\newtheorem*{theorem*}{Theorem}
\newtheorem*{lemma*}{Lemma}
\newtheorem{theorem}{Theorem}
\newtheorem{lemma}[theorem]{Lemma}
\newtheorem{proposition}[theorem]{Proposition}
\newtheorem{corollary}[theorem]{Corollary}
\newenvironment{definition}[1][Definition]{\begin{trivlist}
\item[\hskip \labelsep {\bfseries #1}]}{\end{trivlist}}

\theoremstyle{remark}
\newtheorem*{remark}{Remark}

\mathtoolsset{showonlyrefs}

\def\U{\mathcal{U}}

\def\D{\mathcal{D}}
\def\V{\mathcal{V}}
\def\I{\mathcal{I}}

\title{Arnold Diffusion in Multi-Dimensional Convex Billiards}
\author{
  Andrew Clarke\thanks{UPC, Barcelona, Spain. Email: andrew.michael.clarke@upc.edu}\\
  \and
  Dmitry Turaev\thanks{Imperial College, London, UK. Email: d.turaev@imperial.ac.uk}\\
}
\date{
\today
}

\begin{document}
\maketitle
\begin{abstract} \noindent
Consider billiard dynamics in a strictly convex domain, and consider a trajectory that begins with the velocity vector making a small positive angle with the boundary. Lazutkin proved that in two dimensions, it is impossible for this angle to tend to zero along trajectories. We prove that such trajectories can exist in higher dimensions. Namely, using the geometric techniques of Arnold diffusion, we show that in three or more dimensions, assuming the geodesic flow on the boundary of the domain has a hyperbolic periodic orbit and a transverse homoclinic, the existence of trajectories asymptotically approaching the billiard boundary is a generic phenomenon in the real-analytic topology.
\end{abstract}

\section{Introduction} \label{sec_introduction}
Let $\Gamma$ be a real-analytic, closed, and strictly convex hypersurface of $\mathbb{R}^d$ where $d \geq 3$. The billiard map takes a point $x \in \Gamma$ and an inward pointing velocity vector $v$ of norm $1$ and follows the straight line in the direction of $v$ until the next point of intersection with $\Gamma$. At this point, say $\bar{x}$, the velocity vector transforms according to the optical law of reflection: the angle of reflection $\bar{\alpha}$ equals the angle of incidence $\bar{\beta}$ (see Figure \ref{figure_billiardmap}).

\begin{figure}[!ht]
\includegraphics[width = \textwidth,keepaspectratio]{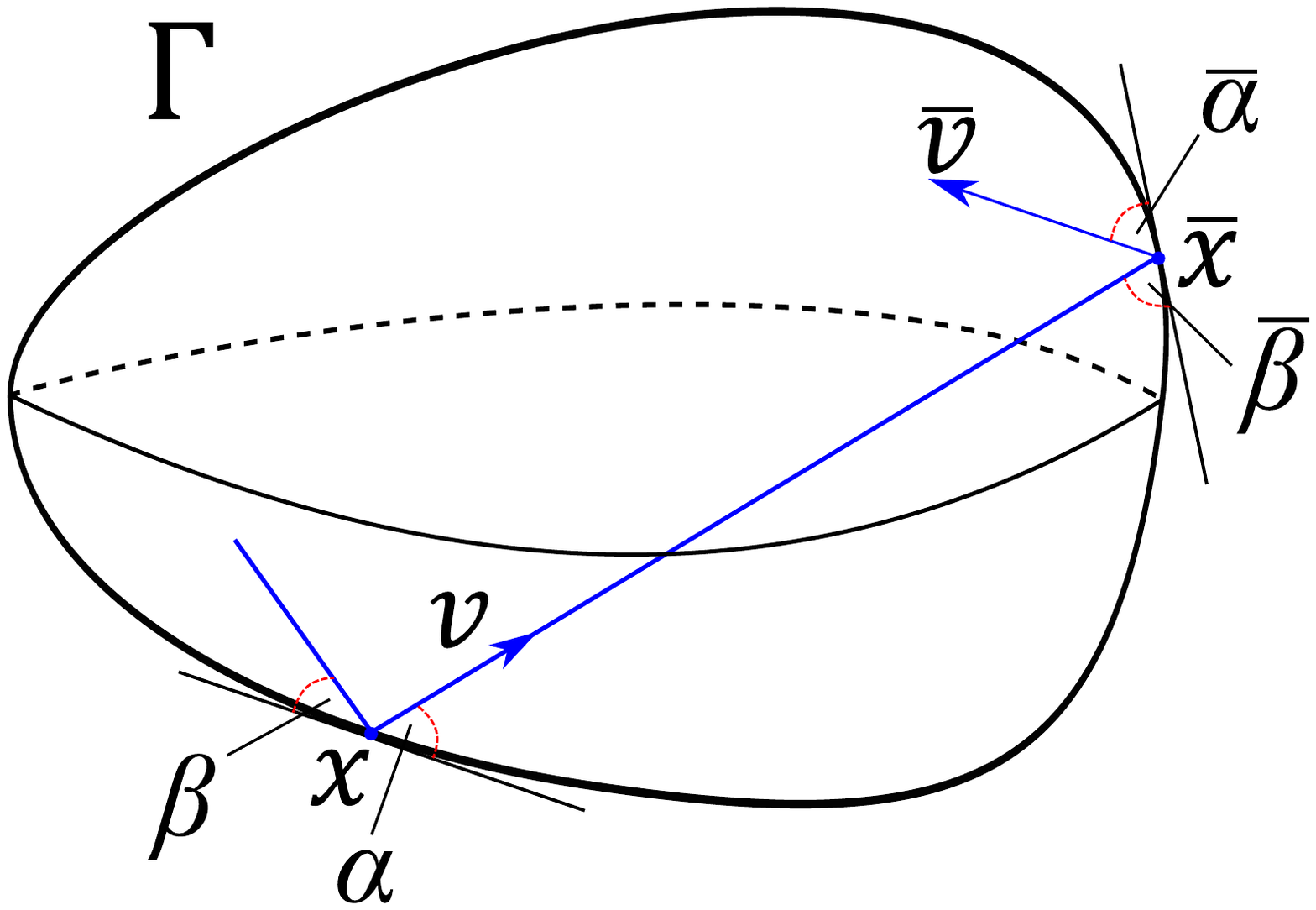}
\caption{The billiard map follows the straight line in the direction of $v$ from a point $x \in \Gamma$ until it next meets $\Gamma$ at $\bar{x}$. At $\bar{x}$, the new velocity vector $v$ is determined according to the optical law of reflection: $\bar{\alpha} = \bar{\beta}$.}
\label{figure_billiardmap}
\end{figure}

It is known as a general principle that the dynamics of the billiard map in the limit as the angle of reflection goes to 0 is determined by the geodesic flow on $\Gamma$. We formulate and prove this statement rigorously (see Section \ref{sec_nearboundary}). Recall the geodesic flow takes a point $x \in \Gamma$ and a tangent vector $u \in T_x \Gamma$ and follows the uniquely defined geodesic $x (t)$ with $x(0) = x$ and $\dot{x}(0) = u$. We make the following assumptions on the geodesic flow on $\Gamma$:
\begin{enumerate}[{[}{A}1{]}]
\item
There is a hyperbolic closed geodesic $\gamma$;
\item
There is a transverse homoclinic geodesic $\xi$ to $\gamma$.
\end{enumerate}
Condition [A2] means that the stable and unstable manifolds of the hyperbolic closed geodesic $\gamma$ intersect transversely, and the geodesic $\xi$ represents an orbit of the geodesic flow lying in their intersection. It is well-known that these conditions are $C^2$-open in terms of the Hamiltonian function, and in our case the Hamiltonian depends on the curvature of the hypersurface (i.e. the second derivative). Therefore the set of hypersurfaces satisfying these assumptions is $C^4$-open. In particular the set $\V$ of hypersurfaces of $\mathbb{R}^d$ for which these assumptions hold is open in the space of all real-analytic, closed, and strictly convex hypersurfaces of $\mathbb{R}^d$, equipped with the real-analytic topology (see Section \ref{sec_results} for definitions). Moreover, the fact that $\V$ is nonempty follows from results of \cite{clarke2019generic}. Our main theorem is as follows:

\begin{theorem*}
For a generic billiard domain in $\V$, there are trajectories of the billiard map such that $\alpha_n \to 0$ as $n \to \infty$ where $\{ \alpha_n \}_{n \in \mathbb{N}}$ denotes the corresponding sequence of angles of reflection.
\end{theorem*}

The word ``generic'' means that the phenomenon occurs on a residual set of hypersurfaces $\Gamma$ belonging to $\V$ in the real-analytic topology. This topology is very restrictive as bump functions are not real-analytic, and so they cannot be used explicitly when making perturbations. To overcome this difficulty, we use a method introduced by Broer and Tangerman \cite{broer1986differentiable}: determine open conditions (in a weaker topology, e.g. $C^4$) to be satisfied by a family of perturbations to obtain the desired effect for arbitrarily small values of the parameter; show that these conditions are satisfied by some locally-supported family of perturbations; and approximate the family of perturbed systems sufficiently well by a real-analytic family. Since the conditions are open in a weaker topology, they are satisfied by the real-analytic family. Therefore we obtain the desired effect for arbitrarily small values of the parameter of the real-analytic family. As a consequence of this approach, our theorem is also true in the $C^k$ topology for $k = 4,5, \ldots, \infty$.

The assumptions [A1,2] are commonly believed to be satisfied in general. In fact, as it is shown in \cite{clarke2019generic} they hold for any $C^{\omega}$-generic convex surface in $\mathbb{R}^3$. They are also satisfied for $C^{\omega}$-generic convex hypersurfaces carrying an elliptic closed geodesic in $\mathbb{R}^d$ for $d \geq 3$ \cite{clarke2019generic}. The remaining case is when $d > 3$ and every closed geodesic is hyperbolic. It is currently unknown whether conditions [A1,2] are $C^{\omega}$-generic in this case, but there have been positive results in a similar setting but with a much weaker topology (for geodesic flows on manifolds with a $C^2$-generic metric)\cite{contreras2010geodesic}. The assumptions [A1,2] have been made before in a similar setting; in the paper \cite{delshams2006orbits}, it was shown that there are orbits of unbounded energy for certain Hamiltonian systems that are small perturbations of geodesic flows satisfying [A1,2]. Moreover there is a discussion of the plenitude of Riemannian metrics satisfying [A1,2] in section 2 of \cite{delshams2006orbits}. 

The diffusive trajectories described by our main result cannot occur in dimension 2 (assuming sufficient smoothness and strict convexity), as was shown by Lazutkin. Indeed, consider a curve in the billiard domain with the following property: if one segment of a billiard trajectory is tangent to the curve, then every segment of the trajectory is (see Figure \ref{figure_2dbilliardmap}). Such curves are called caustics, and correspond to closed invariant curves in the phase space of the billiard map. Lazutkin showed the existence of caustics near the billiard boundary in dimension 2 \cite{lazutkin1973existence}. The corresponding 1-dimensional invariant circles divide the 2-dimensional phase space into invariant regions, keeping the angle of reflection bounded away from zero. See also \cite{douady1982applications}.  

\begin{figure}[!ht]
\includegraphics[width = \textwidth,keepaspectratio]{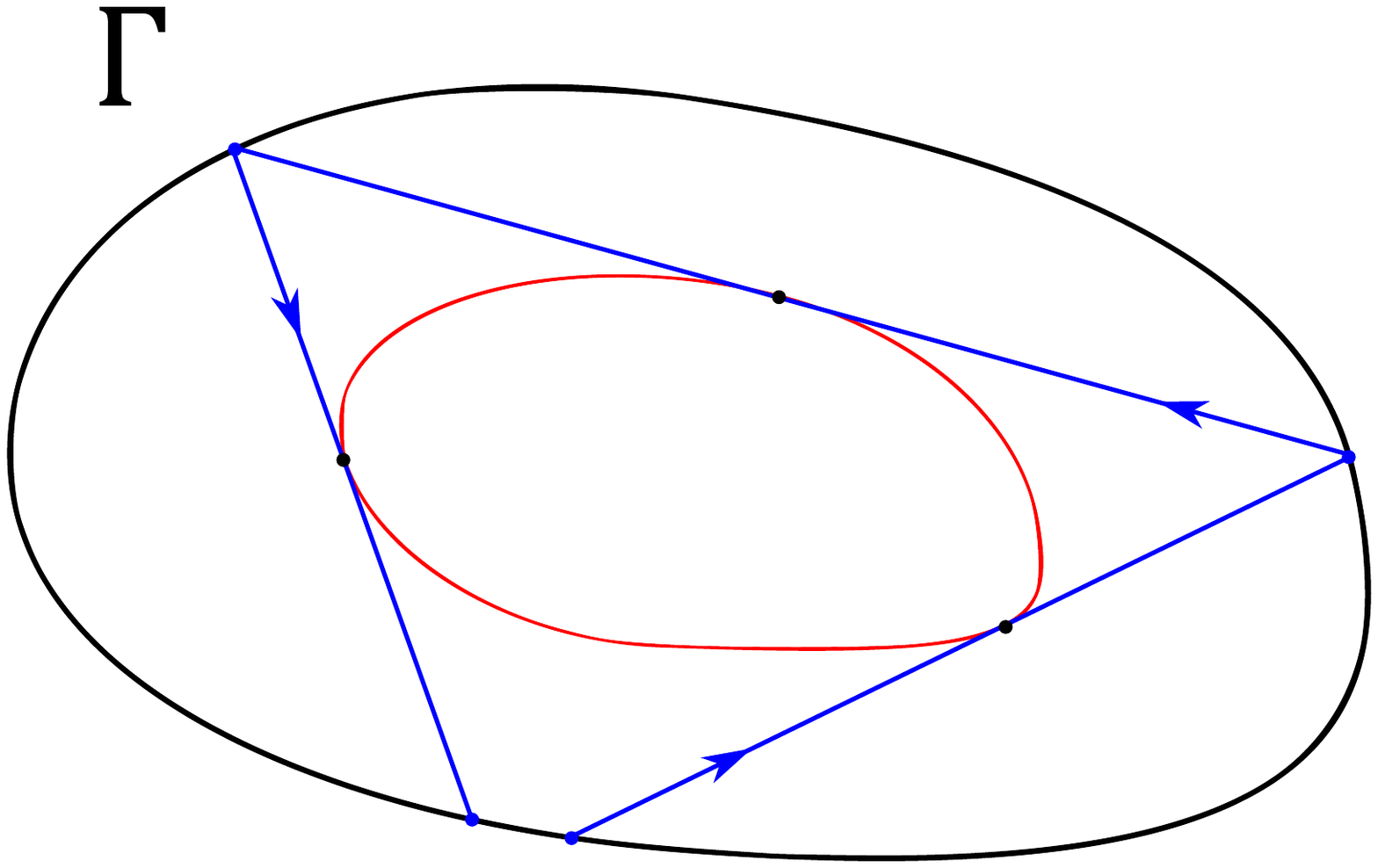}
\caption{The closed curve in the interior of $\Gamma$ represents a caustic, and the three lines are three segments of a trajectory of the billiard map tangent to the caustic.}
\label{figure_2dbilliardmap}
\end{figure}

Both smoothness and convexity are essential for Lazutkin's result. Halpern gave an example of a billiard boundary which is strictly convex, but only $C^2$-smooth, for which there are billiard trajectories that drift towards the boundary in finite time \cite{halpern1977strange}. In this case we do not have Lazutkin's caustics, and KAM theory does not apply. It was shown by Mather that if a billiard table in two dimensions is $C^k$-smooth for $k \geq 2$ but has a flat point (i.e. a point at which the curvature is 0), then there are billiard trajectories which come arbitrarily close to being tangent with the boundary \cite{mather1982glancing}. As a consequence of this and Birkhoff's theorem about invariant circles of twist maps being graphs of Lipschitz functions, he showed the absence of caustics in this case.

In dimension 3 and higher, Theorems of Berger and Gruber imply that only ellipsoids have caustics \cite{berger1995seules, gruber1995only}. For every other billiard domain, this result removes an obstacle to the existence of trajectories asymptotically approaching the boundary. Gruber called such trajectories \emph{asymptotically terminating on the boundary} \cite{gruber1990convex}, and proved that the set of such trajectories must have measure 0. The question of existence, however, was left open.

Even though caustics do not exist, invariant tori with a different structure may be present in higher-dimensional billiards (see \cite{kovachev1990invariant} for a construction of KAM tori for the billiard map in a neighbourhood of an elliptic closed geodesic). In dimension 3, the known examples of invariant tori are of dimension 2 or lower, and so do not divide the 4-dimensional phase space. Then the question of existence of trajectories asymptotically terminating on the boundary becomes the classical question of Arnold diffusion: are there trajectories that move around the invariant tori and drift towards the boundary?

In Arnold's original paper on the instability of Hamiltonian systems, he constructed an example of a nearly integrable Hamiltonian system with orbits along which a component of the momentum drifts by any prescribed amount, and he conjectured: ``I believe that the mechanism of `transition chains' which guarantees that instability in our example is also applicable to the general case''\cite{arnold1964instability}. Arnold's example is a very special case of a system with a normally hyperbolic invariant cylinder, the stable and unstable manifolds of which have a homoclinic intersection. The dynamics near such an object in the general setting is the subject of the modern theory of Arnold diffusion. Our work belongs to the same realm. 

Assumptions [A1,2] imply the existence of a normally hyperbolic invariant manifold $\bar{A}$ for the geodesic flow, the stable and unstable manifolds of which have a transverse homoclinic intersection. We show (see Section \ref{sec_nearboundary}) that when the initial angle of reflection $\alpha_1$ is close to 0, the billiard map can be considered as a small perturbation of a time-shift of the geodesic flow, but with fluctuating speed. We adapt the theory of Fenichel \cite{fenichel1974asymptotic, fenichel1977asymptotic, fenichel1971persistence} to infer that the billiard map inherits the normally hyperbolic invariant manifold, and that the transverse homoclinic intersection persists. Therefore this is the \emph{a priori chaotic} setting as defined in \cite{delshams2000geometric}. The union of the normally hyperbolic invariant manifold with several homoclinic manifolds is called the \emph{homoclinic channel} \cite{delshams2008geometric}. We prove that, generically, there are trajectories of the billiard map that move up the channel, and that the angle of reflection tends to 0 for these trajectories. These trajectories will mostly stay in a neighbourhood of the hyperbolic closed geodesic $\gamma$, occasionally making excursions near a transverse homoclinic intersection. 

Note that the genericity obtained in the main theorem relies on making perturbations within the class $\V$ of strictly convex, analytic hypersurfaces satisfying [A1,2]. However, we do not know of any examples of hypersurfaces in $\V$ that do not contain trajectories tending to the boundary, and indeed it is possible that these trajectories exist for every hypersurface in the class $\V$. One would need to adopt a different, potentially non-perturbative, approach to this problem in order to make such an observation.  

A related concept is that of Fermi acceleration \cite{bolotin1999unbounded,delshams2008geometric2, gelfreich2008unbounded}. It takes place in the setting of nonautonomous Hamiltonian systems, and is characterised by unbounded growth of energy. In the context of billiards, an example of a nonautonomous system would be the billiard dynamics inside a table with moving walls, for instance. There have been several results in recent years establishing Fermi acceleration in such circumstances \cite{dettmann2018splitting, gelfreich2011robust, gelfreich2012fermi}. In our case the system is autonomous, and so the billiard map conserves energy. However the energy conserved by the billiard map is, in general, different to that of the geodesic flow, and it is the energy of the geodesic flow that drifts when we consider the billiard map as a perturbation of a time shift of the geodesic flow.

The geometric techniques used here rely on the scattering map of a normally hyperbolic invariant manifold, introduced in \cite{delshams2000geometric}. It is known that there are actual orbits (of the billiard map, in this case) that shadow orbits of the scattering map (see e.g. \cite{gelfreich2017arnold, gidea2014general}). Therefore one can prove the existence of diffusive trajectories by destroying the invariant curves of scattering maps on an invariant cylinder. This idea was first considered in \cite{moeckel2002generic}, and developed further in \cite{gelfreich2017arnold, le2007drift, nassiri2012robust}. Variational techniques pioneered by Mather \cite{mather2012arnold} have also had success in proving generic existence of diffusive orbits. These results however typically use a lemma of Cheng and Yan which requires essentially the use of bump functions \cite{cheng2004existence}. This has presented an obstacle to producing results in the real-analytic category. We overcome this difficulty using techniques from \cite{gelfreich2017arnold}.

Our results also imply the generic existence of oscillatory motions. In other words, generically there exist trajectories whose corresponding sequence of angles of reflection has positive limit superior, but limit inferior equal to 0 (see Theorem \ref{theorem_main2}). See for example \cite{guardia2016oscillatory,llibre1980oscillatory,sitnikov1960existence} for results regarding the existence of oscillatory motions in the three body problem.

\vspace{5mm}
\noindent
\textit{Acknowledgements.}$\quad$The authors would like to thank Misha Bialy, Amadeu Delshams,  Marian Gidea, Alexey Glutsyuk, Vadim Kaloshin, Tom\'{a}s L\'{a}zaro, Andr\'{e} Neves, Felix Schlenk, and Tere Seara for useful discussions. The first author has received funding for this project from an EPSRC grant,  and from the European Research Council (ERC) under the European Union’s Horizon 2020 research and innovation programme (Grant Agreement No 757802). The second author has received funding for this project from Leverhulme Trust RPG-2021-072 and RScF 19-11-00280 and 19-71-10048. 

\section{Geometry of the Domain, Dynamics, and Results} \label{sec_results}
\subsection{The Domain}
Let $d \geq 3$, and consider a closed and strictly convex hypersurface $\Gamma$ of $\mathbb{R}^d$ given by
\begin{equation} \label{eq_surface}
\Gamma = \Gamma (Q) = \{ x \in \mathbb{R}^d : Q(x) = 0 \}
\end{equation}
where $Q: \mathbb{R}^d \to \mathbb{R}$ is real-analytic. For $x \in \Gamma$, we assume the unit normal vector
\begin{equation} \label{eq_normalvector}
n(x) = - \frac{ \nabla Q (x)}{\| \nabla Q (x) \|}
\end{equation}
is inward-pointing. Clearly this assumption can be satisfied by changing the sign of $Q$ if necessary. The curvature matrix is the Hessian of $Q$ divided by the norm of the gradient of $Q$:
\begin{equation} \label{eq_curvaturematrix}
C(x) = \| \nabla Q (x) \|^{-1} \left( \frac{\partial^2 Q}{ \partial x_i \partial x_j}(x) \right)_{i,j=1, \ldots, d}
\end{equation}
Let $u \in T_x \Gamma$. The shape operator $S(x) : T_x \Gamma \to T_x \Gamma$, defined using the derivative of the unit normal $n(x)$,
\begin{equation} \label{eq_shapeoperator}
S(x) \, u = - Dn(x) \, u = C(x) \, u - \langle C(x) \, u, n(x) \, \rangle \, n(x)
\end{equation}
sends a tangent vector to the tangential component of its image under $C(x)$. This enables a definition of the normal curvature at $x$ in the direction of $u$ via
\begin{equation} \label{eq_normalcurvature}
\kappa (x,u) = \langle S(x) \, u, u \rangle = \langle C(x) \, u, u \rangle. 
\end{equation}
Strict convexity of $\Gamma$ means $\kappa$ is strictly positive whenever $u$ is nonzero. Let
\begin{align}
\pi : T \Gamma & \longrightarrow \Gamma \\
T_x \Gamma \ni u & \longmapsto x
\end{align}
denote the canonical projection along fibres of the tangent bundle. We use the notation $(x,u) \in T \Gamma$ to mean $x \in \Gamma$ and $u \in T_x \Gamma$, so that $\pi (x,u) = x$.

\subsection{The Billiard Map}\label{subsec_billiardmapdef}

The billiard map takes a point on the surface $\Gamma$ and an inward-pointing velocity vector of norm 1, and follows the velocity vector to the next point of collision with the boundary. At this point, the new velocity vector is obtained via the optical law of reflection: the angle of reflection equals the angle of incidence. Billiard dynamics were first considered by Birkhoff \cite{Birkhoff1927}. See also e.g. \cite{chernov2006chaotic}\cite{tabachnikov1995panoramas}\cite{tabachnikov2005geometry} for modern expositions.

The phase space $M$ of the billiard map consists of the closed unit ball around the origin of each tangent space to $\Gamma$,
\begin{equation} \label{eq_phasespace}
M = M(Q) = \{(x,u) \in T \Gamma : \| u \| \leq 1 \}.
\end{equation}
For a point $(x,u) \in M$, the inward-pointing velocity vector is 
\begin{equation} \label{eq_velocityvector}
v = v(x,u) = u + \sqrt{1-u^2} \, n(x),
\end{equation}
where $u^2 = \| u \|^2$. Note that $\| v \| = 1$ since $u$ is a tangent vector to $\Gamma$ at $x$ and $n(x)$ is normal to $\Gamma$ at $x$. The boundary $\partial M$ of the phase space is the set of points $(x,u) \in M$ for which $\| u \| = 1$, and the interior is $\mathrm{Int} (M) = M \setminus \partial M$. Define the free flight time $\tau : \mathrm{Int} (M) \to \mathbb{R}$ between collisions with the boundary as
\begin{equation} \label{eq_flighttime}
\tau (x,u) =    t > 0 \; \; \textrm{such that} \; \; x + t v(x,u) \in \Gamma.
\end{equation}
This is well-defined on the interior of $M$ by strict convexity. Extend it continuously to the boundary of $M$ by setting $\tau |_{\partial M} \equiv 0$. It can be shown that $\tau |_{\mathrm{Int}(M)}$ is real-analytic under our assumptions on $\Gamma$, but analyticity of $\tau$ fails on $\partial M$ in these coordinates (see Section \ref{sec_nearboundary}). Since $\| v \| = 1$, the flight time $\tau (x,u)$ is equal to the Euclidean distance from $x$ to the next point of collision with the boundary. Define the billiard map $f : M \to M$ by $(\bar{x}, \bar{u}) = f(x,u)$ where
\begin{equation} \label{eq_billiardmap}
	\begin{cases}
		\bar{x} = x + \tau (x,u) \, v \\
		\bar{u} = v - \langle v, n(\bar{x}) \rangle \, n(\bar{x}).
	\end{cases}
\end{equation}
The second equation of \eqref{eq_billiardmap} is simply the optical law of reflection. Clearly $f |_{\partial M} \equiv \mathrm{Id}$ since the flight time vanishes identically on the boundary. Moreover, a point $(x,u) \in M$ is a fixed point of the billiard map if and only if $\| u \| = 1$.

The billiard map satisfies the following reversibility property: if $\mathcal{I} :T \Gamma \to T \Gamma$ denotes the operator $\mathcal{I} (x,u) = (x, -u)$ then $f^{-1} = \mathcal{I} \circ f \circ \mathcal{I}$. It follows that the billiard map is bijective.

\begin{remark}
If $\lambda = -u \, dx$ is the Liouville 1-form, then the standard symplectic form on $T \Gamma$ is given by $\omega = d \lambda = dx \wedge du$. Direct differentiation of \eqref{eq_billiardmap} yields
\begin{equation} \label{eq_exactsymplectic}
f^* \lambda - \lambda = - d \tau,
\end{equation}
and so the billiard map is exact symplectic.
\end{remark}

\subsection{The Geodesic Flow}

The geodesic flow on $\Gamma$ takes a point $(x,u) \in T \Gamma$, and follows the uniquely defined geodesic starting at $x$ in the direction of $u$ at a constant speed of $\| u \|$. Typically, it is introduced via the Hamiltonian function $H(x,u) = \frac{1}{2} g(x) (u,u)$ where $g$ is a Riemannian metric, and $(x,u)$ are intrinsic coordinates on the hypersurface $\Gamma$. In our case we use the induced metric and the coordinates of the ambient Euclidean space $\mathbb{R}^d$, so a different formulation is required.

Consider a smooth curve $\gamma : [a,b] \to \Gamma$. The tangential component $\gamma''(t)^T$ of its acceleration is given by
\begin{equation} \label{eq_geodesicacceleration}
\gamma''(t)^T = \gamma''(t) - \langle \gamma''(t), n (\gamma(t)) \rangle n(\gamma (t)).
\end{equation}
Since $\gamma'(t) \in T_{\gamma (t)} \Gamma$ and $n(\gamma (t))$ is normal to $\Gamma$ at $\gamma (t)$ for all $t \in [a,b]$ we have
\begin{equation} \label{eq_secondderivativecurvaturerelation}
0 = \frac{d}{dt} \langle \gamma'(t), n(\gamma (t)) \rangle = \langle \gamma''(t), n(\gamma(t)) \rangle - \kappa (\gamma (t), \gamma'(t)).
\end{equation}
The curve $\gamma$ is a geodesic if and only if $\gamma''(t)^T=0$, so from \eqref{eq_geodesicacceleration} and \eqref{eq_secondderivativecurvaturerelation} we see that $\gamma$ is a geodesic if and only if
\begin{equation}
\gamma^{\prime \prime}(t) = \kappa (\gamma (t), \gamma^{\prime}(t)) \, n(\gamma (t)).
\end{equation}
It follows that the geodesic flow $\phi^t : T \Gamma \to T \Gamma$ is defined by 
\begin{equation} \label{eq_geodesicflowdef}
\left. \frac{d}{dt} \right|_{t=0} \phi^t (x,u) = X(x,u),
\end{equation}
where the vector field $X(x,u) = (\dot{x}, \dot{u})$ is given by
\begin{equation} \label{eq_geodesicfloweom}
	\begin{dcases}
		\dot{x} = u \\
		\dot{u} = \kappa (x,u) \, n (x).
	\end{dcases}
\end{equation}
Consider the function $H : T \Gamma \to \mathbb{R}$ given by
\begin{equation} \label{eq_geodesicflowhamiltonian}
H(x,u) = \frac{u^2}{2} + \kappa (x,u) \frac{Q(x)}{\| \nabla Q(x) \|}.
\end{equation}
It is not hard to see that $X = \Omega \nabla H$ where 
\begin{equation} \label{eq_standardsymplecticmatrix}
\Omega = \left(
\begin{array}{cc}
0 & I_d \\
-I_d & 0
\end{array} \right)
\end{equation}
is the standard symplectic matrix. Therefore the geodesic flow is the Hamiltonian flow associated with the Hamiltonian function $H$.

Notice that the second term of \eqref{eq_geodesicflowhamiltonian} vanishes identically on $T \Gamma$, and so the energy $\frac{u^2}{2}$ is conserved. Since the Hamiltonian $H$ is homogeneous of second order in $u$, the dynamics of the geodesic flow is the same on every energy level. Despite the second term of the Hamiltonian vanishing on the phase space, its contribution to the second equation of motion in \eqref{eq_geodesicfloweom} is nontrivial. In particular, if we did not include the second term, the $x$-component would travel along straight lines in $\mathbb{R}^d$, and would not remain on the hypersurface $\Gamma$.

As mentioned in Section \ref{sec_introduction}, we assume that there is a hyperbolic periodic orbit $\gamma$ of the geodesic flow, and a transverse homoclinic orbit $\xi$. Assumption [A1] ensures that the geodesic flow has a normally hyperbolic invariant manifold consisting of the hyperbolic orbit $\gamma$ on an interval of energy levels, and assumption [A2] ensures that the stable and unstable manifolds of the normally hyperbolic manifold have a transverse homoclinic intersection. It was shown in \cite{contreras2010geodesic} that for a $C^2$-open and dense set of $C^{\infty}$-smooth Riemannian metrics, assumptions [A1,2] are satisfied. In \cite{knieper2002c}, it was proved that these assumptions hold for a $C^{\infty}$-open and dense set of positively curved Riemannian metrics on $\mathbb{S}^2$. There are two significant differences between these cases and ours. First, the analytic category requires special treatment as bump functions cannot be used explicitly. Second, making perturbations only by perturbing the manifold rather than the Riemannian metric is significantly more restrictive, so results about generic metrics do not apply. In the case of closed, strictly convex and real-analytic hypersurfaces of Euclidean space, the paper \cite{clarke2019generic} offers two positive results: for surfaces in $\mathbb{R}^3$, conditions [A1,2] are $C^{\omega}$-open and dense; and for hypersurfaces of $\mathbb{R}^d$ with an elliptic closed geodesic, [A1,2] are satisfied $C^{\omega}$-generically. We use the notation $\gamma$, $\xi$ liberally to refer either to orbits in $T \Gamma$, or to curves in $\Gamma$.

\subsection{Results}

Let $\mathcal{V}$ denote the set of all real-analytic functions $Q : \mathbb{R}^d \to \mathbb{R}$ such that the set $\Gamma = \Gamma (Q)$ defined as in \eqref{eq_surface} satisfies:
\begin{itemize}
\item
$\Gamma$ is a closed and strictly convex hypersurface of $\mathbb{R}^d$; and
\item
Assumptions [A1,2] hold on $\Gamma$.
\end{itemize}
We define the real-analytic topology on $\mathcal{V}$ as follows. Let $K \subset \mathbb{R}^d$ be a compact set, and $\hat{K}$ a compact complex neighbourhood of $K$. If $Q_1, Q_2 \in \mathcal{V}$, by definition they admit holomorphic extensions $\hat{Q}_1, \hat{Q}_2$ on $\hat{K}$. We say that $Q_1, Q_2$ are close on the compact set $K$ in the real-analytic topology if $\hat{Q}_1, \hat{Q}_2$ are uniformly close on $\hat{K}$.

Recall that a set is residual if it is a countable intersection of open dense sets. We now restate our main theorem.

\begin{theorem} \label{theorem_main}
There is a residual subset $\bar{\mathcal{V}}$ of $\mathcal{V}$ such that for every $Q \in \bar{\mathcal{V}}$, there is $(x,u) \in M=M(Q)$ with $\| u \| < 1$ and $\| u^n \| \to 1$ as $n \to \infty$, where $(x^n, u^n) = f^n(x,u)$.
\end{theorem}

A billiard trajectory $\{ (x^n, u^n) = f^n(x,u) \}_{n \in \mathbb{N}_0}$ is said to be \emph{oscillatory} if there is $\delta >0$ such that
\begin{equation}
\liminf_{n \to \infty} \| u^n \| = 1 - \delta, \quad \limsup_{n \to \infty} \| u^n \| = 1.
\end{equation}
The methods of this paper also imply the generic existence of oscillatory motions.

\begin{theorem} \label{theorem_main2}
Let $\bar{\V}$ denote the residual subset of $\V$ described in Theorem \ref{theorem_main}. Then for every $Q \in \bar{\V}$ there is $(x,u) \in M = M(Q)$ for which the trajectory $\{ (x^n, u^n) = f^n(x,u) \}_{n \in \mathbb{N}_0}$ is oscillatory. 
\end{theorem}

Note that the trajectories described in Theorem \ref{theorem_main} (and obviously those of Theorem \ref{theorem_main2}) cannot occur in a finite number of iterations: since the billiard map is bijective, the preimage of a fixed point (i.e. $\| u \| = 1$) is a fixed point. In fact, the process is extremely slow. Since the proof is valid in the analytic category, the Nekhoroshev estimates apply, giving an exponential lower bound on the stability time of finite segments of diffusive trajectories \cite{nekhoroshev1977exponential}. The methods used to obtain diffusive trajectories in this paper are abstract, and do not provide any upper bounds on diffusion time. However, one factor that may point to a particularly long diffusion time is the use of several secondary homoclinic geodesics to $\gamma$ (whose existence is guaranteed by the existence of $\xi$, due to standard results; see for example Theorem 6.5.5 of \cite{katok1995introduction}).

Let us describe these trajectories. We prove that, under our hypotheses, there is a normally hyperbolic invariant manifold $A$ of the billiard map. The manifold $A$ has the structure of a 2-dimensional cylinder, where the height component can be thought of as the reciprocal of the angle of reflection. The diffusive trajectories occur arbitrarily close to (but not on) this cylinder. Due to normal hyperbolicity, $A$ has stable and unstable invariant manifolds $W^{s,u}(A)$. Moreover, these invariant manifolds have a transverse homoclinic intersection. The union of some small neighbourhoods of the normally hyperbolic cylinder and submanifolds of the transverse homoclinic intersection is called the \emph{homoclinic channel}. Diffusive trajectories occur in the homoclinic channel, starting near the cylinder, and making excursions near a transverse homoclinic intersection before returning close to $A$ and remaining in a neighbourhood of $A$ for some iterations (see Figure \ref{figure_diffusivetrajectories}). This process is repeated infinitely many times.

\begin{figure}[!ht]
\includegraphics[width = \textwidth,keepaspectratio]{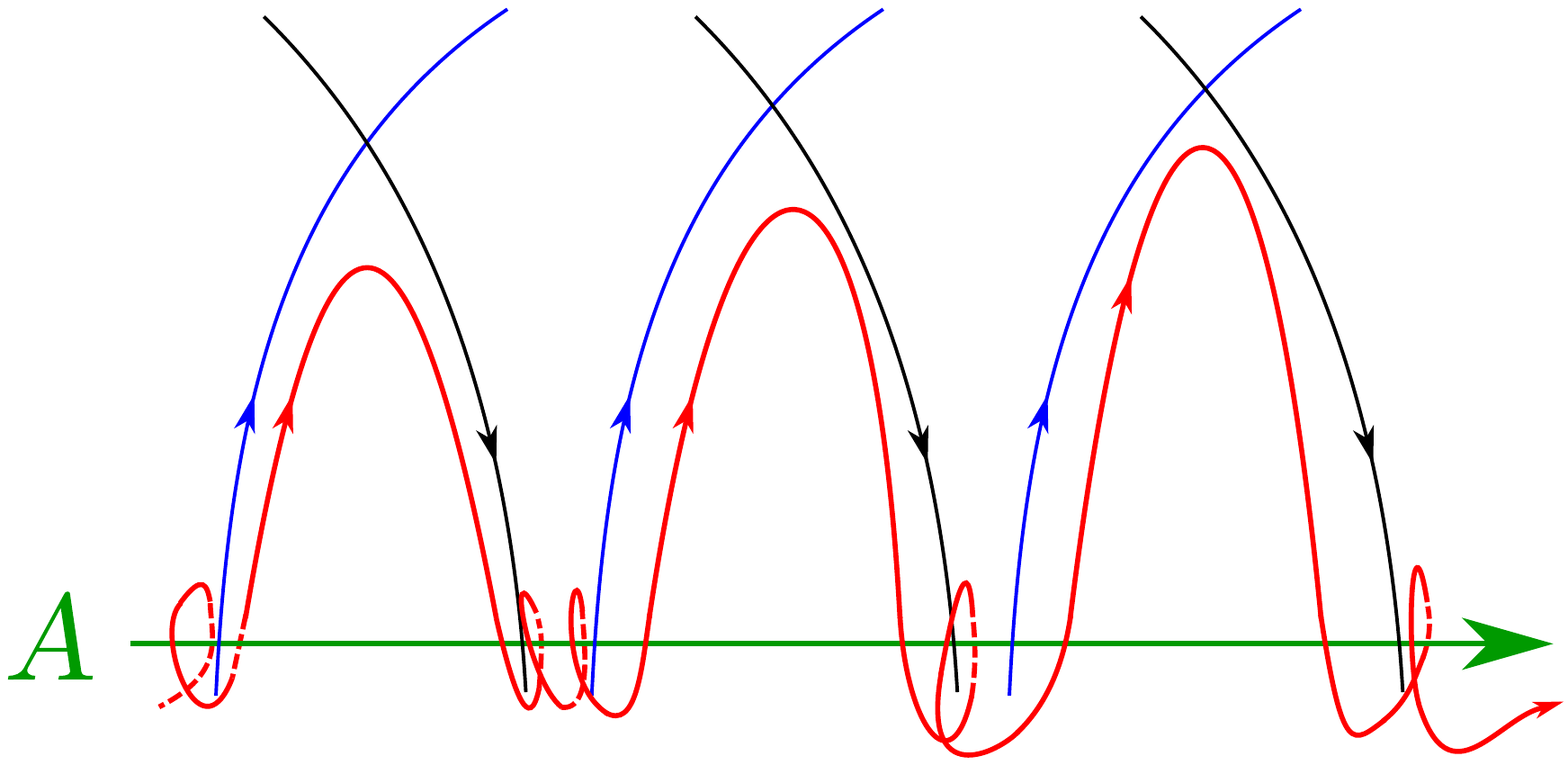}
\caption{The horizontal line represents the height component of the normally hyperbolic invariant cylinder $A$, which is equal to $\| u \|$ plus a small correction. The almost vertical curved lines represent pieces of the stable and unstable manifolds $W^{s,u}(A)$. The curve winding around $A$ is a diffusive trajectory.}
\label{figure_diffusivetrajectories}
\end{figure}

In Section \ref{sec_scatteringmaps} we define the following objects: normally hyperbolic invariant manifolds, homoclinic cylinders, and scattering maps. The stable and unstable manifolds of a normally hyperbolic invariant manifold $A$ have invariant foliations by submanifolds. These are called the strong stable and strong unstable foliations. Each point $x$ on the manifold $A$ uniquely defines a leaf $W^s(x)$ of the strong stable foliation and a leaf $W^u(x)$ of the strong unstable foliation. Suppose now that $y$ is a point of transverse homoclinic intersection of $W^{s,u}(A)$. The implicit function theorem then implies that there is a neighbourhood $B$ of $y$ in $W^s(A) \cap W^u(A)$ such that the homoclinic intersection is transverse at each point in $B$. Suppose moreover that the intersection at $z$ is \emph{strongly transverse} (this will be defined precisely later) for each $z \in B$. Let $A^{\prime}$ denote the subset of $A$ consisting of points $x$ for which the leaf $W^u(x)$ of the strong unstable foliation intersects $B$ at exactly one point. The scattering map $s_B:A^{\prime} \to A$ is defined as follows: from $x \in A^{\prime}$, follow the leaf $W^u(x)$ of the strong unstable foliation until it reaches $B$. At this point there is a unique leaf of the strong stable foliation passing through $B$. Follow this leaf back to $A$. The resulting point is $s_B(x)$ (see Figure \ref{figure_scatteringmap}). The scattering map was introduced by Delshams, de la Llave, and Seara \cite{delshams2000geometric, delshams2003geometric}, and is the key tool for analysing dynamics near a transverse homoclinic intersection.

\begin{figure}[!ht]
\includegraphics[width = \textwidth,keepaspectratio]{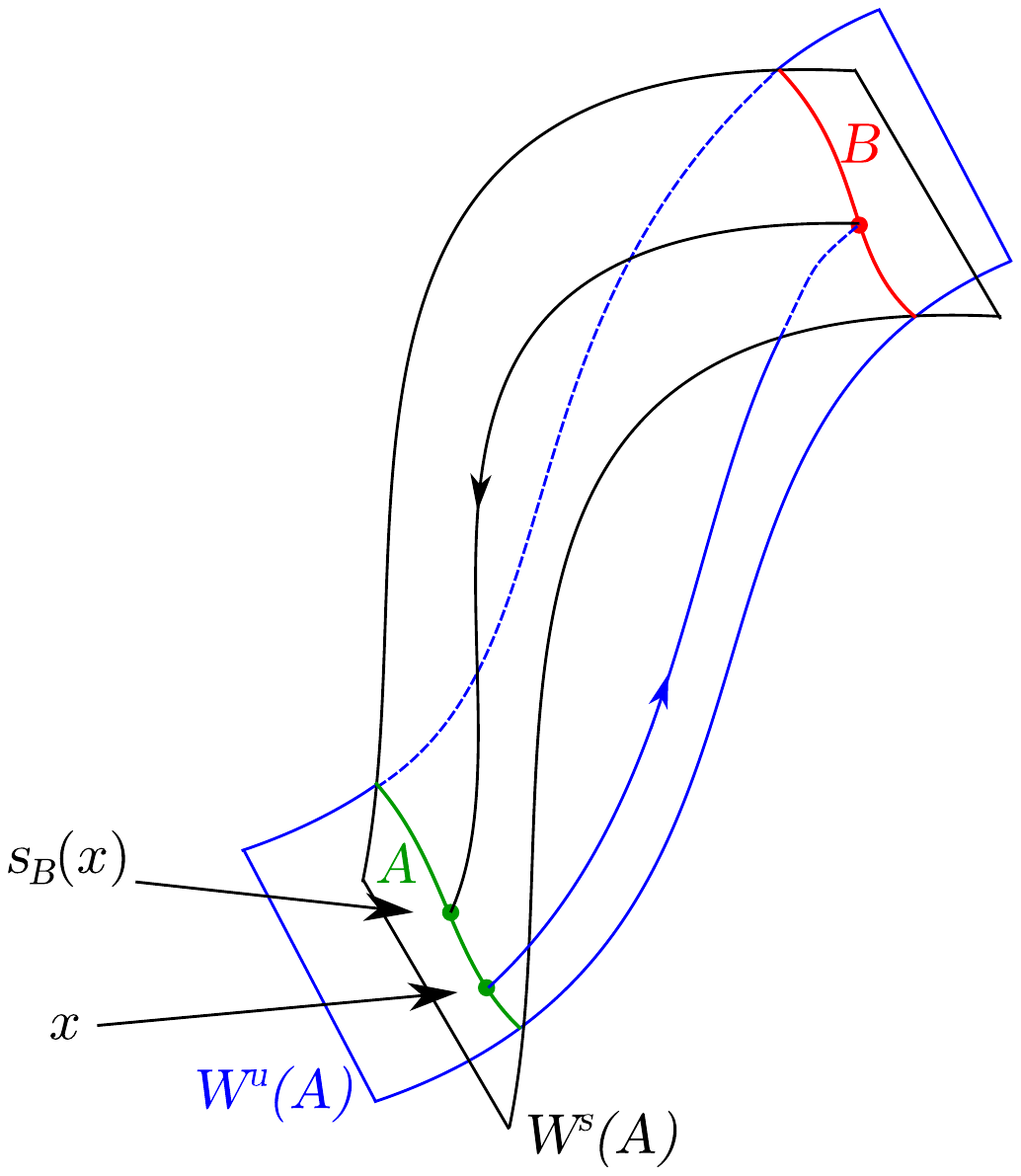}
\caption{The scattering map $s_B$ takes a point $x$ on the normally hyperbolic invariant manifold $A$, follows a leaf of the strong unstable foliation of $W^u(A)$ until it reaches the homoclinic manifold $B$, and then follows a leaf of the strong stable foliation of $W^s(A)$ back to $A$.}
\label{figure_scatteringmap}
\end{figure}

In Section \ref{sec_nearboundary}, we prove that the billiard map has a normally hyperbolic invariant cylinder $A$, with height  component inversely proportional to the angle of reflection, and such that the stable and unstable manifolds of $A$ have a transverse homoclinic intersection. The invariant manifold $A$ is 2-dimensional, and the stable and unstable manifolds are each $d$-dimensional. Each leaf of the strong stable and strong unstable foliations is $(d-2)$-dimensional. We prove the existence of $A$ by using the fact that, as mentioned above, the geodesic flow under assumptions [A1,2] has such objects. Expanding the Taylor series of the billiard map in powers of the flight time $\tau$, we find that the term of order $\tau$ in this expansion is the vector field of the geodesic flow. The problem with taking this approach directly is that the `small parameter' $\tau$ is in fact a function, and this function has unbounded $u$-derivative as $\|u\| \to 1$. This implies that the approximation of the billiard map by a time shift of the geodesic flow is only good in the $C^0$ topology, in these coordinates. Therefore we fix some small constant value $\tau_* >0$ of the flight time, and make a change of coordinates so that the billiard map is real-analytic in some subset of the phase space. If we expand the billiard map now in powers of $\tau_*$, we find that the term of order $\tau_*$ is a vector field $Z$, whose flow traces the geodesics of $\Gamma$, but at a fluctuating (nonzero) speed. Moreover $Z$ preserves an integral, and so has a normally hyperbolic invariant cylinder that consists of the orbit defined by the hyperbolic closed geodesic $\gamma$ on an interval of energy levels (i.e. an interval of values of the integral). The angular component of the cylinder is an angular variable on the closed geodesic $\gamma$, and the height corresponds to the energy level. We conclude that the billiard map is $O(\tau_*^2)$ close to the time-$\tau_*$ shift of the flow of $Z$. Fenichel theory guarantees \cite{fenichel1974asymptotic, fenichel1977asymptotic, fenichel1971persistence} the persistence of the normally hyperbolic invariant manifold $A$ for any dynamics $C^1$-close enough to $Z$. Since we can make $\tau_*$ as small as we like by considering trajectories nearer to the boundary, we get the desired persistent manifold for the billiard map. 

Notice that as $\tau \to 0$, the billiard map tends to the identity. Therefore as we move up the cylinder $A$, the hyperbolicity parameters become weaker. To deal with this, we look at a large (i.e. $1/\tau$) iterate of the billiard map, and compare this to the time-1 shift of the geodesic flow. Of course this means that we must exclude the case $\tau=0$, which corresponds to the upper boundary of the normally hyperbolic invariant manifold $A$. As a consequence, the resulting persistent manifold $A$ for the billiard map is a non-compact non-uniformly normally hyperbolic invariant manifold. Aside from this, the non-uniformity of hyperbolicity parameters does not explicitly impact our proof, although it does mean that diffusive trajectories become slower the further up the cylinder they move. Non-compactness, however, means that we must adapt some proofs of results we use from \cite{gelfreich2017arnold} and \cite{gidea2014general}.

We conduct an analysis of the homoclinic cylinders in Section \ref{section_homoclinicylindersandpseudoorbits}. We show in this section that it is sufficient to make perturbations along (most of) a fundamental domain of the billiard map on a homoclinic trajectory. We also give conditions for the existence of diffusive pseudo-orbits, and describe how there are orbits of the billiard map that shadow pseudo-orbits arbitrarily well.

In Section \ref{sec_perturbations} we show how to perturb the billiard map by perturbing the function $Q$ that defines the manifold $\Gamma$. We make a perturbation of size $\epsilon$, and compute the first order expansion of the map resulting from this change of the surface. Since the billiard map is symplectic, the first order term is a Hamiltonian vector field, corresponding to some Hamiltonian $H_{\mathrm{pert}}$. Similarly, the first order term in the expansion of the scattering map in powers of $\epsilon$ is a Hamiltonian vector field. Delshams, de la Llave, and Seara proved a formula (Theorem 31 of \cite{delshams2008geometric}) for the Hamiltonian function of this vector field in terms of $H_{\mathrm{pert}}$ and the unperturbed billiard dynamics. We make use of this formula here.

In Section \ref{sec_diffusiveorbits}, we use the machinery developed in \cite{gelfreich2017arnold} and \cite{gidea2014general} to determine conditions which, when satisfied, guarantee the existence of diffusive trajectories. We then prove that these conditions are satisfied generically in the real-analytic category. Since we have a homoclinic cylinder for the billiard map, there is a horseshoe near the normally hyperbolic manifold $A$. This in turn implies the existence of infinitely many independent homoclinic cylinders (see e.g. Proposition 2 of \cite{gelfreich2017arnold}). Take 8 such cylinders, and construct an iterated function system $\{ f, s_1, \ldots, s_8 \}$ on $A$, consisting of the 8 scattering maps and the restriction of the billiard map to $A$. Orbits of the IFS are often called pseudo-orbits, and it can be shown that there are orbits of the map $f$ that shadow pseudo-orbits arbitrarily well \cite{gidea2014general}. An adaptation of a result from \cite{gelfreich2017arnold} implies that if the IFS has no essential invariant curves on the cylinder $A$, then there exist diffusive pseudo-orbits: orbits of the IFS that drift up the cylinder $A$. We then prove that the absence of invariant essential curves of the IFS is a generic property. This is done by embedding the billiard map in a two-parameter family of real-analytic billiard maps. In Lemma \ref{lemma_perturbationconditions}, we give open conditions on this family which, if satisfied, imply that the set of parameters for which there are no invariant essential curves is of full measure. The lemma is proved by contradiction, using a combinatorial argument that requires 8 scattering maps. It is based on a result in \cite{gelfreich2017arnold}, and finishes the proof of Theorem \ref{theorem_main}. The idea of using families of real-analytic mappings to approximate a system perturbed by a bump function was used in \cite{broer1986differentiable, clarke2019generic, gelfreich2017arnold, gonchenko2007homoclinic}.

\section{The Scattering Map of a Normally Hyperbolic Invariant Manifold: Definitions} \label{sec_scatteringmaps}

In this section, let $M$ be a sufficiently smooth manifold, and $\phi^t : M \to M$ a sufficiently smooth flow with
\begin{equation}
\left. \frac{d}{dt} \right|_{t=0} \phi^t = X.
\end{equation}
Let $A \subset M$ be a compact invariant manifold (possibly with boundary) for $\phi^t$, in the sense that for each $x \in A$, we have $\phi^t (x) \in A$ for all $t \in \mathbb{R}$. We say that $A$ is a normally hyperbolic invariant manifold if there is $0 < \lambda < \mu^{-1} < 1$, a positive constant $C$, and an invariant splitting of the tangent bundle
\begin{equation}
T_A M = TA \oplus E^s \oplus E^u
\end{equation}
such that\footnote{This is actually stronger than the standard definition. Notice that the rate of expansion $\lambda$ is equal to the rate of contraction. Such invariant manifolds are called symmetrically normally hyperbolic. Moreover, the standard definition of normally hyperbolic invariant manifolds allows the parameters $\lambda, \mu$ to depend on the point in $A$. However the definition given here is sufficient for our case.}:
\begin{equation} \label{eq_normalhyperbolicity}
\def\arraystretch{1.5}
\begin{array}{c}
\| D \phi^t |_{E^s} \| \leq C \lambda^t \mbox{ for all } t \geq 0, \\
\| D \phi^t |_{E^{u}} \| \leq C \lambda^{-t} \mbox{ for all } t \leq 0, \\
\| D \phi^t |_{TA} \| \leq C \mu^{| t |} \mbox{ for all } t \in \mathbb{R}.
\end{array}
\end{equation}
Moreover, $A$ is called an $r$-normally hyperbolic invariant manifold whenever $0 < \lambda < \mu^{-r} < 1$ for $r \geq 1$. This is also referred to as a \emph{large spectral gap} condition. Taking integer values of $t$ in \eqref{eq_normalhyperbolicity} gives the definition of a normally hyperbolic invariant manifold for the map $F = \phi^1$.

Due to the normal hyperbolicity of $A$, there exist stable and unstable manifolds, defined by
\begin{equation}
\def\arraystretch{1.5}
\begin{array}{c}
W^s (A) = \{ x \in M : d(\phi^t (x), A) \leq C(x) \tilde{\lambda}^t \mbox{ for all } t \geq 0 \}, \\
W^u (A) = \{ x \in M : d( \phi^t (x), A) \leq C(x) \tilde{\lambda}^{-t} \mbox{ for all } t \leq 0 \},
\end{array}
\end{equation}
where $C$ is a constant that is allowed to depend on $x$, and $\tilde{\lambda} \in [\lambda,\mu^{-1})$ may be chosen close to $\lambda$. On the stable and unstable manifolds we have, respectively, the strong stable and strong unstable foliations. The leaves of the foliations corresponding to $x \in A$ are
\begin{equation} 
\def\arraystretch{1.5}
\begin{array}{c}
W^s (x) = \{ y \in M : d( \phi^t (x), \phi^t (y)) \leq C(x,y) \tilde{\lambda}^t \mbox{ for all } t \geq 0 \}, \\
W^u (x) = \{ y \in M : d( \phi^t (x), \phi^t (y)) \leq C(x,y) \tilde{\lambda}^{-t} \mbox{ for all } t \leq 0 \}.
\end{array}
\end{equation}
The tangent spaces of the leaves of the foliations are given by $T_x W^s(x) = E^s_x$ and $T_x W^u (x) = E^u_x$. Two leaves of the strong stable or strong unstable foliation corresponding to different points of $A$ are disjoint, and
\begin{equation} \label{eq_foliationequations}
W^s (A) = \bigcup_{x \in A} W^s (x), \quad W^u (A) = \bigcup_{x \in A} W^u (x). 
\end{equation}
The foliations are invariant in the sense that for any $x \in A$ and $t \in \mathbb{R}$ we have
\begin{equation} \label{eq_invariantfoliations}
\phi^t (W^s(x)) = W^s ( \phi^t (x)), \quad \phi^t (W^u(x)) = W^u (\phi^t (x)).
\end{equation}
Since we assume $A$ to be sufficiently smooth (i.e. more than $r$ times continuously differentiable), the stable and unstable manifolds of $A$ are $C^r$-smooth, where $r$ is the size of the spectral gap. So too are the leaves of the strong stable and strong unstable foliations \cite{hirsch1977invariant}. Moreover, the foliations are $C^{r-1}$-regular, meaning that the fields of tangents to the leaves are $C^{r-1}$-smooth.

It was proved by Fenichel in the 70's that open normally hyperbolic invariant manifolds are stable under small perturbations: for any vector field $Y$ in some $C^1$-small neighbourhood of $X$, there is a persistent $C^1$-smooth open normally hyperbolic invariant manifold, $C^1$-close to $A$ \cite{fenichel1974asymptotic, fenichel1977asymptotic, fenichel1971persistence}. Moreover, if $X$ is $C^r$-smooth, $A$ is an $r$-normally hyperbolic invariant manifold, and $Y$ is $C^r$-close to $X$, then the persistent manifold is $C^r$-smooth and is $C^r$-close to $A$. Moreover, the stable and unstable manifolds of the perturbed manifold are $C^r$-close to those of $A$. In the case of open normally hyperbolic manifolds, invariance means that the vector field is tangent to the manifold at each point. 

Similar results hold for maps: if $F$ is a $C^r$ smooth map of a smooth manifold $M$ with an open $r$-normally hyperbolic invariant manifold $A$, then there is a $C^r$ neighbourhood of $F$ such that for any map $G$ from this neighbourhood, the invariant manifold $A$ persists as an open $C^r$-smooth normally hyperbolic invariant manifold \cite{hirsch1977invariant}.

Now suppose there is a homoclinic point $x \in (W^s(A) \cap W^u (A)) \setminus A$. We call $x$ a \emph{transverse homoclinic} point if $W^s(A)$ and $W^u(A)$ are transverse at $x$, in which case we write $x \in (W^s(A) \pitchfork W^u(A)) \setminus A$. If $x$ is a transverse homoclinic point, the implicit function theorem implies that there is an open neighbourhood $V$ of $x$ contained in $W^s(A) \cap W^u(A)$ such that every $y \in V$ is a transverse homoclinic point.

Given a transverse homoclinic point $x \in (W^s(A) \pitchfork W^u(A)) \setminus A$, there are unique $x_{\pm} \in A$ such that $x \in W^s(x_+) \cap W^u (x_-)$. The homoclinic intersection at $x$ is said to be \emph{strongly transverse} if 
\begin{equation} \label{eq_strongtransversality}
\def\arraystretch{1.5}
\begin{array}{c}
T_x W^s (x_+) \oplus T_x (W^s(A) \cap W^u(A)) = T_x W^s (A), \\
T_x W^u (x_-) \oplus T_x (W^s(A) \cap W^u(A)) = T_x W^u (A).
\end{array}
\end{equation}
In this case we define the holonomy maps $\pi^{s,u} : V \to A$ to be projections along the leaves of the strong stable and strong unstable foliations, so $\pi^s(x) = x_+$ and $\pi^u (x) = x_-$. Notice that the holonomy maps are locally invertible whenever the strong transversality condition holds. Moreover they are $C^{r-1}$-smooth by the regularity of the foliations. Therefore, taking $V$ to be small enough so that \eqref{eq_strongtransversality} holds at each point, the holonomy maps are $C^r$-diffeomorphisms onto their respective images. Therefore we can define the \emph{scattering map} from $\pi^u (V) \to \pi^s(V)$:
\begin{equation} \label{eq_scatteringmap}
s = \pi^s \circ (\pi^u)^{-1} : x_- \longmapsto x_+.
\end{equation}
Now, consider the case of a map $F = \phi^1$. Assume that $A$ has the structure of a cylinder, so there is a diffeomorphism $h : \mathbb{T} \times [0,1] \to A$. Let $B \subset (W^s(A) \pitchfork W^u (A)) \setminus A$ be a smooth connected two-dimensional submanifold. 
\begin{definition}
We say that $B$ is a \emph{homoclinic cylinder} if the scattering map $s : A \longrightarrow A$ relative to $B$ is a diffeomorphism. 
\end{definition}

Suppose the map $F$ is exact symplectic. Recall that a number is said to be Diophantine if it is sufficiently poorly approximable by rationals. Suppose the boundary curves $\gamma^{\pm}$ of the cylinder $A$ are KAM curves, meaning the restriction of the time-1 shift of the flow to $\gamma^{\pm}$ is smoothly conjugate to a rotation by a Diophantine angle. KAM theory implies that they survive all exact symplectic $C^4$-small perturbations (see e.g. Appendix 8 of \cite{arnol1978mathematical}) of an exact symplectic map. We can extend the normally hyperbolic invariant manifold $A$ to an open normally hyperbolic invariant manifold $A'$. Due to Fenichel theory, the open cylinder $A'$ persists under small perturbations, and due to KAM theory the boundary curves of $A$ persist as invariant curves of the exact symplectic map on the perturbed open cylinder. If $A$ is a 2-dimensional manifold, this implies the persistence of the cylinder $A$ under small exact symplectic perturbations. Moreover the leaves of the strong stable and strong unstable foliations have smooth dependence on the corresponding point in $A$, and so the perturbed foliations will be close to those of the unperturbed system. Therefore the perturbed stable and unstable manifolds themselves will be close to the unperturbed invariant manifolds. This reasoning also implies that a tranverse homoclinic intersection of stable and unstable manifolds remains transverse under small perturbations \cite{hirsch1977invariant}.

It is clear that if $B$ is a homoclinic cylinder, then so too is $F^n(B)$ for any $n \in \mathbb{Z}$. The following identities are an immediate consequence of the invariance \eqref{eq_invariantfoliations} of the strong stable and strong unstable foliations:
\begin{equation} \label{eq_holonomycommute}
\pi^s_{F(B)} = F \circ \pi^s_B \circ F^{-1}, \quad \pi^u_{F(B)} = F \circ \pi^u_B \circ F^{-1}.
\end{equation}
We also have analogous statements for time shifts of the flow $\phi^t$. Similar expressions hold for $F^n (B)$ for any $n \in \mathbb{Z}$. Furthermore, combining \eqref{eq_scatteringmap} and \eqref{eq_holonomycommute} gives an analogous expression for the scattering map $s_{F^n(B)}$.

\section{Billiard Trajectories Near the Boundary} \label{sec_nearboundary}

The goal of this section is to show that the billiard map inherits a normally hyperbolic invariant manifold from the geodesic flow, and to analyse the homoclinic intersection of its stable and unstable manifolds. This is done by showing that, in some coordinates and for sufficiently small $\tau_*>0$, the billiard map $f$ is approximated up to order $\tau_*^2$ by the time-$\tau_*$ shift of the flow of some vector field $Z$ in the $C^r$-topology for any finite $r$. We show that $Z$ has a normally hyperbolic invariant manifold that does not depend on $\tau_*$ as a result of assumptions [A1,2]. Therefore, shrinking $\tau_*$ if necessary, we can use the theory of Fenichel to imply the persistence of the normally hyperbolic invariant manifold for $f$ \cite{fenichel1974asymptotic, fenichel1977asymptotic, fenichel1971persistence}. This process is repeated a countable number of times, and the resulting invariant manifolds are glued together to obtain a noncompact normally hyperbolic invariant manifold for $f$.

Let $Q \in \mathcal{V}$ and let $\Gamma = \Gamma (Q)$, $M = M(Q)$ be as in \eqref{eq_surface}, \eqref{eq_phasespace} respectively. In this section we do not make any perturbations to the surface, so we may assume without loss of generality that $\| \nabla Q (x) \| = 1$ for all $x \in \Gamma$. Indeed, we can replace $Q$ by 
\begin{equation}
\tilde{Q}(x) = \frac{Q(x)}{\| \nabla Q (x) \|}.
\end{equation}
Then $\tilde{Q}(x)=0$ if and only if $Q(x)=0$, so $\Gamma (\tilde{Q})= \Gamma$. Moreover if $x \in \Gamma$ then 
\begin{equation} 
\nabla \tilde{Q}(x) = \frac{\nabla Q (x)}{\| \nabla Q (x) \|}
\end{equation}
since $Q$ vanishes identically on $\Gamma$, and so the gradient of $\tilde{Q}$ has norm 1 on $\Gamma$. This assumption simplifies computations since the unit normal takes the form
\begin{equation} \label{eq_scalednormal}
n(x) = - \nabla Q(x)
\end{equation}
which implies that the shape operator is
\begin{equation} \label{eq_scaledshapeoperator}
S(x) = -Dn(x) = \left( \frac{\partial^2 Q}{\partial x_i \partial x_j}(x) \right)_{i,j=1, \ldots, d} = C(x).
\end{equation}

\subsection{Estimates Near the Boundary}

We consider trajectories near the boundary by taking $\delta > 0$ small and restricting our attention to the subset
\begin{equation} \label{eq_restrictedphasespace}
M_{\delta} = \{ (x,u) \in T \Gamma : 1 - \delta \leq \| u \| \leq 1 \}
\end{equation}
of the billiard map's phase space.
\begin{lemma} \label{lemma_preliminary}
Let $(x,u) \in M_{\delta}$. Write $\tau = \tau (x,u)$, and recall $v = u + \sqrt{1 - u^2} \, n(x)$.
\begin{enumerate}[(i)]
\item \label{item_lemmaprelim1}
We have
\begin{equation} \label{eq_lemmaprelim1}
\sqrt{1 - u^2} = \frac{1}{2} \tau \kappa (x,v) + \frac{1}{3} \tau^2 R(x,v) + O(\tau^3)
\end{equation}
where
\begin{equation} \label{eq_lemmaprelim1_2}
R (x,v) = \sum_{i,j,k=1}^d \frac{\partial^3 Q}{\partial x_i \partial x_j \partial x_k}(x) v_i v_j v_k.
\end{equation}
\item \label{item_lemmaprelim2}
Write $(\bar{x}, \bar{u}) = f(x,u)$. Then
\begin{equation} \label{eq_lemmaprelim2}
(\bar{x}, \bar{u}) = (x,u) + \tau X(x,u) + O(\tau^2)
\end{equation}
where $X$ is the vector field of the geodesic flow, given by \eqref{eq_geodesicfloweom}.
\item \label{item_lemmaprelim3}
With $n(x)$ as in \eqref{eq_scalednormal},
\begin{equation} \label{eq_lemmaprelim3}
n(\bar{x}) = n(x) - \tau S(x) v - \frac{1}{2} \tau^2 r(x,v) + O(\tau^3)
\end{equation}
where $r$ is the vector-valued function with components
\begin{equation} \label{eq_lemmaprelim3_2}
r_i (x,v) = \sum_{j,k=1}^d \frac{\partial^3 Q}{\partial x_i \partial x_j \partial x_k}(x) v_j v_k.
\end{equation}
\item \label{item_lemmaprelim4}
With $R(x,v)$ as in \eqref{eq_lemmaprelim1_2},
\begin{equation} \label{eq_lemmaprelim4}
\sqrt{1 - \bar{u}^2} = \sqrt{1 - u^2} - \frac{1}{6} \tau^2 R(x,v) + O(\tau^3).
\end{equation}
\end{enumerate}
\end{lemma}
\begin{proof}
These statements are proved by taking Taylor expansions in powers of $\tau$. Notice that
\begin{equation} \label{eq_innerprodissqrt}
\langle n(x), v \rangle = \langle n(x), u \rangle + \sqrt{1 - u^2} \langle n(x), n(x) \rangle = \sqrt{1 - u^2}
\end{equation}
 From \eqref{eq_scalednormal}, \eqref{eq_scaledshapeoperator}, \eqref{eq_innerprodissqrt}, and the fact that $Q(x) = Q(\bar{x}) = 0$ we get
\begin{align}
0 ={}& Q(\bar{x}) = Q(x + \tau v) \\
={}& Q(x) + \tau \langle \nabla Q(x), v \rangle + \frac{1}{2} \tau^2 \sum_{i,j=1}^d \frac{\partial^2 Q}{\partial x_i \partial x_j}(x) v_i v_j + \\
& \hspace{8em} +\frac{1}{3} \tau^3 \sum_{i,j,k=1}^d \frac{\partial^3 Q}{\partial x_i \partial x_j \partial x_k}(x) v_i v_j v_k + O(\tau^4) \\
={}& - \tau \langle n(x), v \rangle + \frac{1}{2} \tau^2 \langle S(x) v, v \rangle + \frac{1}{3} \tau^3 R(x,v) + O(\tau^4) \\
={}& - \tau \sqrt{1 - u^2} + \frac{1}{2} \tau^2 \kappa (x,v) + \frac{1}{3} \tau^3 R(x,v) + O(\tau^4).
\end{align}
Dividing by $\tau$ and rearranging terms gives \eqref{item_lemmaprelim1}, where the terms of higher order are uniformly bounded due to the real-analyticity of $Q$.

For \eqref{item_lemmaprelim2} we must compute the term of order $\tau$ in the expansion of $(\bar{x}, \bar{u})$. From \eqref{eq_lemmaprelim1} we get
\begin{equation} \label{eq_xbarexpansion}
\bar{x} = x + \tau v = x + \tau u + \tau \sqrt{1 - u^2} \, n(x) = x + \tau u + O(\tau^2).
\end{equation}
Also by \eqref{eq_innerprodissqrt},
\begin{equation}
\bar{u}|_{\tau =0} = u + \sqrt{1 - u^2} \, n(x) - \langle n(x), v \rangle n(x) = u
\end{equation}
and moreover
\begin{align}
\left. \frac{d}{d \tau}\right|_{\tau =0} \bar{u} & = - \left. \langle v, Dn(\bar{x}) v \rangle \right|_{\tau =0} - \left. \langle v, n(\bar{x}) \rangle Dn(\bar{x})v \right|_{\tau =0} \\
&= \langle v, S(x)v \rangle n(x) + \sqrt{1 - u^2} \, S(x) v \\
&= \kappa (x, u) n(x) + O(\tau)
\end{align}
where $D$ denotes the derivative, and where \eqref{eq_lemmaprelim1} was used. It follows that
\begin{equation} \label{eq_ubarexpansion}
\bar{u} = u + \tau \kappa (x,u) n(x) + O(\tau^2).
\end{equation}
Comparing \eqref{eq_xbarexpansion} and \eqref{eq_ubarexpansion} with \eqref{eq_geodesicfloweom} completes the proof of \eqref{item_lemmaprelim2}.

From \eqref{eq_scalednormal} we get
\begin{align}
n(\bar{x}) &= n(x) + \tau Dn(x)v + \frac{1}{2} \tau^2 D^2 n(x) v^{\otimes 2} + O(\tau^3) \\
&= n(x) - \tau S(x) v - \frac{1}{2} \tau^2 r(x,v) + O(\tau^3)
\end{align}
which is \eqref{item_lemmaprelim3}.

For part \eqref{item_lemmaprelim4} notice that the second equation of \eqref{eq_billiardmap} combined with basic geometry implies that
\begin{equation}
\bar{v} = v - 2 \langle n(\bar{x}), v \rangle n(\bar{x}) = \bar{u} - \langle n(\bar{x}), v \rangle n(\bar{x})
\end{equation}
which, along with \eqref{eq_innerprodissqrt}, gives
\begin{equation} \label{eq_innerprodrelation}
\langle n(\bar{x}), \bar{v} \rangle = - \langle n(\bar{x}), v \rangle.
\end{equation}
Comparing \eqref{eq_lemmaprelim1_2} and \eqref{eq_lemmaprelim3_2} we see that
\begin{equation}
\langle r(x,v), v \rangle = \sum_{i,j,k=1}^d \frac{\partial^3 Q}{\partial x_i \partial x_j \partial x_k}(x) v_i v_j v_k = R (x,v).
\end{equation}
Using this together with \eqref{eq_lemmaprelim3},
\begin{align}
\langle n(\bar{x}), v \rangle &= \langle n(x), v \rangle - \tau \langle S(x) v, v \rangle - \frac{1}{2} \tau^2 \langle r(x,v), v \rangle + O(\tau^3) \\
&= \langle n(x), v \rangle - \tau \kappa (x,v) - \frac{1}{2} \tau^2 R(x,v) + O(\tau^3). \label{eq_innerprodrelation2}
\end{align}
From \eqref{eq_lemmaprelim1}, \eqref{eq_innerprodissqrt}, \eqref{eq_innerprodrelation}, and \eqref{eq_innerprodrelation2} we get
\begin{align}
\sqrt{1 - \bar{u}^2} ={}& \langle n(\bar{x}), \bar{v} \rangle = - \langle n(\bar{x}), v \rangle \\
={}& - \langle n(x), v \rangle + \tau \kappa (x,v) + \frac{1}{2} \tau^2 R(x,v) + O(\tau^3) \\
={}& -\frac{1}{2} \tau \kappa (x,v) - \frac{1}{3} \tau^2 R(x,v) + \tau \kappa (x,v) + \frac{1}{2} \tau^2 R(x,v) + O(\tau^3) \\
={}& \frac{1}{2} \tau \kappa (x,v) + \frac{1}{3} \tau^2 R(x,v) - \frac{1}{6} \tau^2 R(x,v) + O(\tau^3) \\
={}& \langle n(x), v \rangle - \frac{1}{6} \tau^2 R(x,v) + O(\tau^3) \\
={}& \sqrt{1 - u^2} - \frac{1}{6} \tau^2 R(x,v) + O(\tau^3)
\end{align}
which completes the proof of the lemma.
\end{proof}

\begin{remark}
Equation \eqref{eq_lemmaprelim2} implies that the billiard map is approximated up to order $\tau^2$ by the time-$\tau$ shift of the geodesic flow. Equation \eqref{eq_lemmaprelim1}, however, implies that the $u$-derivative of $\tau$ is unbounded as $\| u \|$ tends to 1. Therefore the approximation of the billiard map by a time shift of the geodesic flow is only good in the $C^0$ topology, which is not enough to use the theory of Fenichel. However, if we scale up $\tau$ by dividing by a small constant $\tau_* >0$ and consider it as a variable $z$ rather than a function, normalise $u$ to have norm 1 (i.e. let $w = \| u \|^{-1} u$), and consider the billiard map as a function of $x, w$, and $z$, then it is real-analytic and near to the identity by \eqref{eq_lemmaprelim1} and \eqref{eq_lemmaprelim4}, and estimates similar to those contained in Lemma \ref{lemma_preliminary} hold in the $C^r$-topology for any finite $r$ with respect to the new constant small parameter $\tau_*$. This process is carried out in the next section - see the change of coordinates \eqref{eq_rescaledvariables}.
\end{remark}

\subsection{Normally Hyperbolic Invariant Manifold for the Billiard Map}

Fix some small $\tau_* >0$ and consider the coordinates $(x,w,z)$ on the subset $M_{\delta}$ (see equation \eqref{eq_restrictedphasespace}) of the phase space $M$ where
\begin{equation} \label{eq_rescaledvariables}
w = \frac{u}{\| u \|}, \quad z = \tau_*^{-1} \frac{2 \sqrt{1 - u^2}}{\kappa (x,w)}.
\end{equation}
Here, we are replacing the radial variable $\| u \|$ by the scaled variable $z$. Fix any $b > a > 0$ and let
\begin{equation} \label{eq_nbhdddef}
\D_* = \{ (x,w,z): a < z < b \}.
\end{equation}
Note that $\D_* \subset M$. Consider the vector field $Z(x,w,z) = (\dot{x}, \dot{w}, \dot{z})$ where
\begin{equation} \label{eq_vfzdefinition}
\begin{cases}
\dot{x} = zw \\
\dot{w} = z \kappa (x,w) n(x) \\
\dot{z} = - \frac{4}{3} z^2 \kappa (x,w)^{-1} R(x,w),
\end{cases}
\end{equation}
where $R$ is given by \eqref{eq_lemmaprelim1_2}. Comparing $Z$ with $X$ (see \eqref{eq_geodesicfloweom}), it is clear that the orbits of $Z$ trace the geodesics of $\Gamma$, but at a fluctuating speed $z$. The following lemma is the general principle mentioned in the introduction: the dynamics of the billiard map in the limit as the angle of reflection goes to 0 is determined by the geodesic flow.

\begin{lemma} \label{lemma_interpolatingvectorfield}
On $\D_*$ we have
\begin{equation} \label{eq_interpolatingvectorfield}
f = Id + \tau_* Z + O(\tau_*^2)
\end{equation}
in the $C^r$-topology for any finite $r$.
\end{lemma}

\begin{proof}
The second equation of \eqref{eq_rescaledvariables} implies that
\begin{equation}
\| u \| = \sqrt{1 - \frac{1}{4} \tau_*^2 \kappa (x,w)^2 z^2} = 1 + O(\tau_*^2).
\end{equation}
This together with the first equation of \eqref{eq_rescaledvariables} gives
\begin{equation} \label{eq_wisu}
w = u + O(\tau_*^2).
\end{equation}
From \eqref{eq_lemmaprelim1} and \eqref{eq_wisu} we get
\begin{equation}
\tau = \frac{2 \sqrt{1 - u^2}}{\kappa (x,v)} + O(\tau^2) = \frac{2 \sqrt{1 - u^2}}{\kappa (x,w)} + O(\tau^2),
\end{equation}
and so
\begin{equation} \label{eq_tiszconst}
\tau = \tau_* z + O(\tau^2) = \tau_* z + O(\tau_*^2).
\end{equation}
Write $ (\bar{x}, \bar{w}, \bar{z}) = f (x,w,z)$. From \eqref{eq_lemmaprelim2} and \eqref{eq_wisu} we get 
\begin{equation} \label{eq_rescaledximage}
\bar{x} = x + \tau v = x + \tau_*z u + O(\tau_*^2) = x + \tau_*z w + O(\tau_*^2)
\end{equation}
and
\begin{align} 
\bar{w} ={}& \bar{u} + O(\tau_*^2) = u + \tau_* z \kappa (x,u) n(x) + O(\tau_*^2) \\
={}& w + \tau_* z \kappa (x,w) n(x) + O(\tau_*^2). \label{eq_rescaledwimage}
\end{align}
Let us determine $\bar{z}$. Equations \eqref{eq_scaledshapeoperator}, \eqref{eq_rescaledximage}, and \eqref{eq_rescaledwimage} imply
\begin{align}
\left. \frac{d}{d \tau_*} \right|_{\tau_* =0} \kappa (\bar{x}, \bar{w}) ={}& \left. \frac{d}{d \tau_*} \right|_{\tau_* =0} \sum_{i,j=1}^d \frac{\partial^2 Q}{\partial x_i \partial x_j}(\bar{x}) \bar{w}_i \bar{w}_j \\
={}& z \sum_{i,j,k=1}^d \frac{\partial^3 Q}{\partial x_i \partial x_j \partial x_k}(x) w_i w_j w_k + \\
& \hspace{4em} + 2 z \kappa (x,w) \sum_{i,j=1}^d \frac{\partial^2 Q}{\partial x_i \partial x_j}(x) w_i n_j (x) \\
={}& z R(x,w) + 2z \kappa (x,w) \langle S(x) w, n(x) \rangle \\
={}& z R(x,w) \label{eq_kappatauexpansion}
\end{align}
since $S(x) w$ is tangent to $\Gamma$ and $n(x)$ is normal. Therefore
\begin{equation}
\kappa(\bar{x},\bar{w})^{-1} = \kappa (x,w)^{-1} - \tau_* z \kappa (x,w)^{-2} R(x,w) + O(\tau_*^2).
\end{equation}
Moreover \eqref{eq_lemmaprelim1}, \eqref{eq_lemmaprelim4}, \eqref{eq_wisu}, and \eqref{eq_tiszconst} imply
\begin{equation}
\sqrt{1 - u^2} = \frac{1}{2} \tau_* z \kappa (x,w) + O(\tau_*^2)
\end{equation}
and
\begin{equation}
\sqrt{1 - \bar{u}^2} = \sqrt{1 - u^2} - \frac{1}{6} \tau_*^2 z^2 R(x,w) + O(\tau_*^3).
\end{equation}
Combining the last three equations with the second equation of \eqref{eq_rescaledvariables},
\begin{align}
\bar{z} ={}& \tau_*^{-1} \frac{2 \sqrt{1 - \bar{u}^2}}{\kappa (\bar{x}, \bar{w})} \\
={}& 2 \tau_*^{-1} \left( \sqrt{1 - u^2} - \frac{1}{6} \tau_*^2 z^2 R(x,w) \right) \times \\
& \hspace{4em} \times \left( \kappa (x,w)^{-1} - \tau_* z \kappa (x,w)^{-2} R(x,w) \right) + O(\tau_*^2) \\
={}& \tau_*^{-1} \frac{2 \sqrt{1 - u^2}}{\kappa (x,w)} - 2z \sqrt{1 - u^2} \kappa (x,w)^{-2} R(x,w) - \\
& \hspace{4em} - \frac{1}{3} \tau_* z^2 \kappa (x,w)^{-1} R(x,w) + O(\tau_*^2) \\
={}& z - \frac{4}{3} \tau_* z^2 \kappa (x,w)^{-1} R(x,w) + O(\tau_*^2).
\end{align}
It follows that
\begin{equation}
f = Id + \tau_* Z + O(\tau_*^2)
\end{equation}
where the vector field $Z(x,w,z)$ is as in \eqref{eq_vfzdefinition}.
\end{proof}

Recall that $\gamma : \mathbb{T} \to \Gamma$ is a hyperbolic closed geodesic of length 1, and $\xi : \mathbb{R} \to \Gamma$ is a transverse homoclinic geodesic, both parametrised by arclength. 

 \begin{lemma} \label{lemma_nhimexistence}
The vector field $Z$ has an integral $y$ defined by
\begin{equation} \label{eq_defofintegralofvfz}
y^{-1} = z \kappa (x,w)^{ \frac{4}{3}},
\end{equation} 
and the set
\begin{equation} \label{eq_nhicdef}
\bar{A}_* = \left\{ (x,w,z)= \left( \gamma (q), \gamma' (q), y^{-1} \kappa (\gamma (q), \gamma' (q))^{- \frac{4}{3}} \right) : q \in \mathbb{T}, y \in \left[ \frac{1}{2}, \frac{3}{2} \right] \right\}
\end{equation}
is an $r$-normally hyperbolic invariant manifold for the flow of $Z$ for any finite $r$.
\end{lemma}
\begin{proof}
Notice that, from \eqref{eq_vfzdefinition}, and by a similar computation to \eqref{eq_kappatauexpansion},
\begin{equation}
\frac{d}{dt} \kappa (x,w) = z R(x,w).
\end{equation}
It follows that
\begin{align}
\frac{d}{dt} \ln z = \frac{\dot{z}}{z} = - \frac{4}{3} z \kappa (x,w)^{-1} R(x,w) = - \frac{4}{3} \frac{d}{dt} \ln \kappa (x,w) = \frac{d}{dt} \ln \kappa (x,w)^{- \frac{4}{3}}
\end{align}
which implies that $z = C \kappa (x,w)^{- \frac{4}{3}}$ for some positive constant $C$. Therefore the quantity $y$ defined by \eqref{eq_defofintegralofvfz} is constant along trajectories. Since $z = y^{-1} \kappa (x,w)^{- \frac{4}{3}}$ it is bounded along trajectories, and bounded away from 0 for strictly positive $y$. It follows that the vector field $Z$ follows the geodesics on $\Gamma$ at a fluctuating speed that is bounded away from 0. Therefore $\bar{A}_*$ as defined in \eqref{eq_nhicdef} is a normally hyperbolic invariant manifold for $Z$. 

Notice that every point $(q, y) \in \bar{A}_*$  is periodic with minimal period $l = l(y)$ depending only on $y$. We have
\begin{equation}
l(y) = \int_{q \in \mathbb{T}} z(q) dq = y^{-1} \int_{q \in \mathbb{T}} \kappa (\gamma (q), \gamma' (q))^{- \frac{4}{3}} dq = K^{-1} y^{-1}
\end{equation}
where
\begin{equation}
K^{-1} = \int_{s \in \mathbb{T}} \kappa (\gamma (s), \gamma' (s))^{- \frac{4}{3}} ds
\end{equation}
is a positive constant depending only on the periodic orbit $\gamma$. Then we can introduce coordinates $(\theta, y)$ so that the cylinder is written as 
\begin{equation} \label{eq_nhicdefwiththetacoord}
\bar{A}_* = \left\{ (x,w,z) = \phi_Z^{\theta l(y)} \left( \gamma (0), \gamma' (0), y^{-1} \kappa (\gamma (0), \gamma' (0))^{- \frac{4}{3}} \right) : \theta \in \mathbb{T}, y \in \left[ \frac{1}{2}, \frac{3}{2} \right] \right\},
\end{equation}
where $\phi^t_Z$ denotes the time-$t$ shift of the flow of $Z$. In this equation the time $t$ along the orbit is equal to $\theta l(y)$, and so we get $\dot{\theta} = Ky$. Therefore on $\bar{A}_*$ in $(\theta, y)$ coordinates we have
\begin{equation} \label{eq_vfzdefonnhic}
Z(\theta,y) = (Ky,0)
\end{equation}
and
\begin{equation} \label{eq_flowofzonnhic}
\phi^t_Z (\theta,y) = (\theta + tKy,y).
\end{equation}
It follows that the differential of the time-$t$ shift of the flow of $Z$ is
\begin{equation}
D \phi^t_Z = 
\begin{pmatrix}
1 & tK \\
0 & 1
\end{pmatrix}
\end{equation}
which has a double eigenvalue of 1. Therefore there is an infinitely large spectral gap in these coordinates, and so $\bar{A}_*$ is an $r$-normally hyperbolic invariant manifold.
\end{proof}

\begin{remark}
By choosing $b>a>0$ appropriately in equation \eqref{eq_nbhdddef} we can guarantee that $\bar{A}_* \subset \D_*$. We can extend the normally hyperbolic invariant manifold $\bar{A}_*$ to an open normally hyperbolic invariant manifold $\bar{A}_*' \subset \D_*$. Now, equation \eqref{eq_interpolatingvectorfield} implies that $f^N$ is approximated up to terms of order $\tau_*$ by the time-1 shift $\phi^1_Z$ of the vector field $Z$ uniformly in the $C^r$ topology for $r$ finite, where $N$ is the integer part of $\tau_*^{-1}$. Fenichel theory implies that there is a $C^r$ neighbourhood of $\phi^1_Z$ such that for any map in this neighbourhood, the open $r$-normally hyperbolic invariant manifold $\bar{A}_*'$ persists as an open $C^r$-smooth normally hyperbolic invariant manifold \cite{fenichel1974asymptotic, fenichel1977asymptotic, fenichel1971persistence}. Shrinking $\tau_*$ as required, we can ensure that $f^N$ lies in this $C^r$ neighbourhood of $\phi^1_Z$, and so we obtain an open $C^r$-smooth normally hyperbolic invariant manifold $A_*'$ for $f^N$ that is $O(\tau_*)$-close to $\bar{A}_*'$ in the $C^r$-topology for $r$ finite. Moreover, this manifold $A_*'$ is itself an open normally hyperbolic invariant manifold for the billiard map $f$. Choosing the boundary curves of $\bar{A}_*$ to be KAM curves, we find that the manifold $A_*'$ contains a compact $C^r$ normally hyperbolic invariant manifold $A_*$ of $f$ that is $O(\tau_*)$-close to $\bar{A}_*$ in the $C^r$ topology, as explained in Section \ref{sec_scatteringmaps}. Furthermore, the results imply that when we perturb the surface $\Gamma$, the normally hyperbolic invariant manifold $A_*$ of the billiard map persists. This is because the invariant manifold $\bar{A}_*$ for $Z$ persists in the tangent bundle of the perturbed hypersurface, and so we can apply the results again to obtain a normally hyperbolic invariant manifold for the billiard map on the perturbed manifold.  
\end{remark}

The next result shows that we can choose a sequence $\{\bar{A}_n \}_{n \in \mathbb{N}}$ of normally hyperbolic invariant cylinders for the flow of $Z$ such that: the cylinders overlap; as we move up the cylinders $\| u \| \to 1$; and the corresponding persistent cylinders $A_n$ of the billiard map fit together to form a manifold, itself having the structure of a noncompact cylinder (see Figure \ref{figure_noncompactcylinder}).

\begin{figure} [!ht]
\includegraphics[width = \textwidth,keepaspectratio]{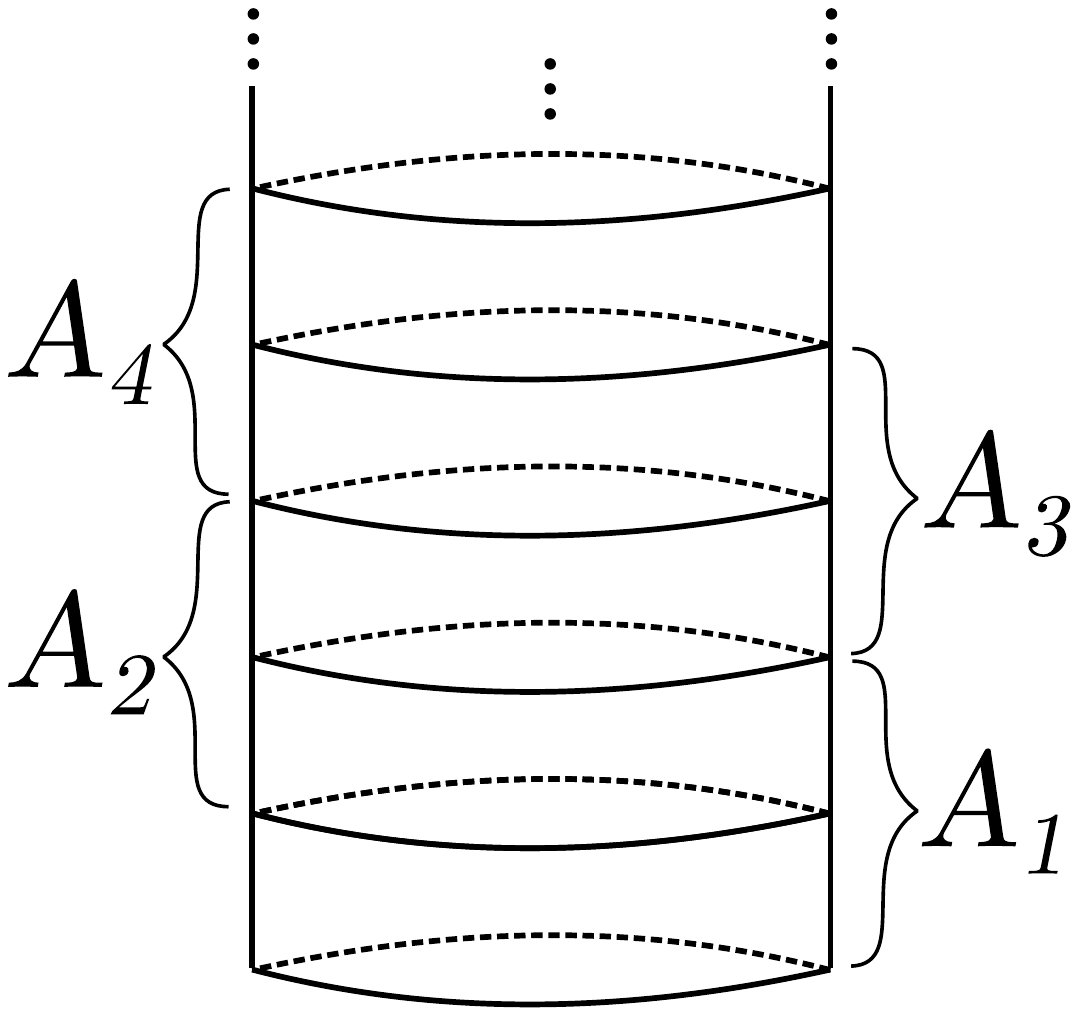}
\caption{We choose the constants $\{ \tau_n \}_{n \in \mathbb{N}}$ so that the cylinders $A_n$ fit together to form a noncompact cylinder, and so that the angle of reflection tends to 0 as we move up the cylinders.}
\label{figure_noncompactcylinder}
\end{figure}

\begin{proposition} \label{proposition_noncompactcylinder}
We can choose a sequence of constants $\{ \tau_n \}_{n \in \mathbb{N}}$ such that $\tau_n \to 0$ as $n \to \infty$, and so that the corresponding cylinders $A_n$ for the billiard map fit together to form a noncompact nonuniformly normally hyperbolic invariant manifold. Moreover, $\| u \| \to 1$ as we move up the cylinder.
\end{proposition}

\begin{proof}
Let $\bar{A}_n$ denote the cylinder $\bar{A}_*$ in \eqref{eq_nhicdef} corresponding to $\tau_* = \tau_n$. We choose the sequence $\{ \tau_n \}_{n \in \mathbb{N}}$ so that the top of the cylinder $\bar{A}_1$ which is $y_1 = \frac{3}{2}$ matches up with the middle of $\bar{A}_2$ which is $y_2 = 1$ (see equation \eqref{eq_nhicdef}). If we can choose $\tau_1, \tau_2$ so that $y_1, y_2$ correspond to the same value of $\| u \|$ at $(\gamma (q), \gamma' (q))$ for each $q \in \mathbb{T}$, then the cylinders match up. We have
\begin{equation} \label{eq_yjformula}
y_j^{-1} = z_j \kappa (\gamma (q), \gamma' (q))^{\frac{4}{3}} = 2 \tau_j^{-1} \sqrt{1 - u^2} \, \kappa (\gamma (q), \gamma' (q))^{\frac{1}{3}},
\end{equation}
so we need
\begin{equation}
\frac{2}{3} \tau_1 = \tau_1 y_1^{-1} = 2 \sqrt{1-u^2} \, \kappa (\gamma (q), \gamma' (q))^{\frac{1}{3}} = \tau_2 y_2^{-1} = \tau_2
\end{equation}
or more generally,
\begin{equation} \label{eq_tausequence}
\tau_{n+1} = \frac{2}{3} \tau_n = \cdots = \left( \frac{2}{3} \right)^n \tau_1 
\end{equation}
which tends to 0 as $n \to \infty$. Therefore we choose $\tau_1 > 0$ sufficiently small so that Fenichel theory and KAM theory guarantee the persistence of the cylinder $\bar{A}_1$ for the billiard map, along with its transverse homoclinic intersection. Then define $\tau_n$ according to \eqref{eq_tausequence}. Since the sequence is strictly decreasing, and since $\bar{A}_n$ as defined in \eqref{eq_nhicdef} does not depend explicitly on $\tau_n$, each cylinder $\bar{A}_n$ persists for the billiard map, along with its transverse homoclinic intersection. Clearly $\| u \| \to 1$ as we move up the cylinders since, from \eqref{eq_yjformula}, and for any $q \in \mathbb{T}$,
\begin{equation}
\sqrt{1 - u^2} = \frac{1}{2} \tau_n \kappa (\gamma (q) , \gamma' (q))^{- \frac{1}{3}} y^{-1} \longrightarrow 0 \text{ as } n \to \infty.
\end{equation}

Now, we must prove that
\begin{equation}
\bigcup_{n=1}^{\infty} A_n
\end{equation}
has the structure of a cylinder. It is clear that $\mathrm{Int}(\bar{A}_n) \cap \mathrm{Int}(\bar{A}_{n+1}) \neq \emptyset$. Recall that each closed essential curve on $\bar{A}_n$ corresponding to constant $y$ is an invariant circle of the flow of $Z$ with frequency $\rho = \rho (y)$ depending only on $y$. Moreover, $\rho$ is strictly monotone in $y$. KAM theory ensures that curves corresponding to Diophantine frequencies survive perturbation and so are invariant circles for the billiard map. We may assume without loss of generality (as we may slightly shift the boundary curves of the cylinders $\bar{A}_n$ if necessary) that each cylinder $\bar{A}_n$ is bounded by curves with Diophantine frequencies.

Let $\bar{\gamma}_1$ be the lower boundary curve of $\bar{A}_{n+1}$ corresponding to $y = y_1$. For any $\epsilon > 0$ we can find $\eta_2 \in (y_1, y_1 + \epsilon)$ such that $\rho (\eta_2)$ is Diophantine with respect to the map $f^N$. Let $\bar{\gamma}_2$ be the invariant circle corresponding to $y = \eta_2$. Then there are invariant circles $\gamma_1$, $\gamma_2$ for the billiard map that are $\tau_{n}$ close in the $C^r$-topology to $\bar{\gamma}_1$, $\bar{\gamma}_2$, respectively by KAM theory since the perturbation is exact symplectic. Clearly we must have $\gamma_1, \gamma_2 \subset A_n \cap A_{n+1}$. Let $D_1$ be the region in $A_n$ bounded by $\gamma_1$ and $\gamma_2$. Then $D_1$ is an invariant set for the billiard map. There is a characterisation of the local stable (respectively unstable) manifold of $A_{n+1}$ as the set of points whose forward (resp. backward) orbit never leaves a neighbourhood $V$ of $A_{n+1}$ \cite{hirsch1977invariant}. Therefore the manifold $A_{n+1}$ is the set of points whose forward and backward orbits never leave $V$. By choosing $\epsilon$ small enough, we can guarantee that $D_1 \subset V$ since $A_n$ and $A_{n+1}$ are $\tau_n$-close in the $C^r$-topology. It follows that $D_1 \subset A_{n+1}$, and so the cylinders $A_n$ and $A_{n+1}$ agree between the curves $\gamma_1$ and $\gamma_2$. Now choose curves $\bar{\gamma}_3, \ldots, \bar{\gamma}_N$ for some $N \in \mathbb{N}$ such that $\bar{\gamma}_N$ is the upper boundary of $\bar{A}_n$, and such that each $\bar{\gamma}_j \subset \bar{A}_n$ corresponds to $y = \eta_j$ where $\rho (\eta_j)$ is a Diophantine number and $0 < \eta_j - \eta_{j-1} < \epsilon$. Then this argument applies to the area between each pair of consecutive curves $\gamma_j, \gamma_{j+1}$, and so the cylinders $A_n, A_{n+1}$ agree where they overlap. This clearly applies for each $n \in \mathbb{N}$.
\end{proof}

\section{Homoclinic Cylinders and Pseudo-Orbits} \label{section_homoclinicylindersandpseudoorbits}

In this section we construct an iterated function system consisting of the restriction of the billiard map to the non-compact normally hyperbolic invariant manifold together with $n$ scattering maps. An orbit of the iterated function system is called a \emph{pseudo-orbit}. We give sufficient conditions for the existence of diffusive pseudo-orbits, and explain how there are orbits of the billiard map that shadow pseudo-orbits arbitrarily well.
\subsection{Analysis of Hyperbolic and Homoclinic Cylinders} \label{sec_analysisofhypandhomcylinders}

It was shown in Section \ref{sec_nearboundary} that in $(x,w,z)$-coordinates defined by \eqref{eq_rescaledvariables}, the billiard map $f$ is approximated up to order $\tau_*^2$ by the time $\tau_*$-shift of the flow of the vector field $Z$, given by \eqref{eq_vfzdefinition}. Recall it was shown in the proof of Lemma \ref{lemma_nhimexistence} that we can write the normally hyperbolic invariant manifold $\bar{A}_*$ of the vector field $Z$ as in equation \eqref{eq_nhicdefwiththetacoord}, which gives us natural coordinates $(\theta, y)$. Moreover, each point $(\theta, y) \in \bar{A}_*$ is a periodic point of the flow of the vector field $Z$ (which is given by \eqref{eq_vfzdefonnhic}, \eqref{eq_flowofzonnhic}) with minimal period $l(y) = K^{-1} y^{-1}$. Since $K > 0$ depends only on $\gamma$, we can replace $Ky$ by $y$ and write
\begin{equation}
\bar{A}_* = \left\{ (x,w,z) = \phi_Z^{\theta y^{-1}} \left( \gamma (0), \gamma' (0), K y^{-1} \kappa (\gamma (0), \gamma' (0))^{- \frac{4}{3}} \right) : \theta \in \mathbb{T}, y \in \left[ \frac{K}{2}, \frac{3K}{2} \right] \right\}.
\end{equation}
Then the vector field and the flow take the simple form
\begin{equation} \label{eq_cylindervectorfieldeq1}
Z (\theta, y) = (y,0)
\end{equation}
and
\begin{equation} \label{eq_cylinderfloweq1}
\phi^t_Z (\theta,y) = (\theta + ty,y).
\end{equation}
Notice that the speed $z$ of the flow of $Z$ is now given by
\begin{equation} \label{eq_rescaledvfzspeed}
z = K y^{-1} \kappa (x,w)^{- \frac{4}{3}}.
\end{equation}

Since the geodesic flow is Hamiltonian, the stable and unstable manifolds of the hyperbolic closed geodesic $\gamma$, restricted to an energy level in the $2(d-1)$-dimensional manifold $M$, are of dimension $d-1$. Thus the 2-dimensional normally hyperbolic invariant manifold $\bar{A}_*$ has $d$-dimensional stable and unstable manifolds $W^{s,u}(\bar{A}_*)$. On the stable (resp. unstable) manifold there is the strong stable (resp. strong unstable) foliation, each leaf of which is $(d-2)$-dimensional, and is uniquely determined by a point $(\theta, y) \in \bar{A}_*$ (see Section \ref{sec_scatteringmaps} for definitions). 

\begin{lemma} \label{lemma_homocliniccylinderlemma1}
The stable and unstable manifolds $W^{s,u}(\bar{A}_*)$ of the normally hyperbolic invariant manifold $\bar{A}_*$ of the flow of $Z$ intersect strongly transversely along the set
\begin{equation}
B = \left\{ (x,w,z) = \left(\xi (q), \xi'(q), Ky^{-1} \kappa (\xi (q), \xi' (q))^{- \frac{4}{3}} \right) : q \in \mathbb{R}, y \in \left[ \frac{K}{2}, \frac{3K}{2} \right] \right\}
\end{equation}
where $\xi$ is the homoclinic geodesic, and we can introduce coordinates $(\theta, y)$ on $B$ such that the holonomy maps
\begin{equation}
\bar{\pi}^{s,u}_B : B \longrightarrow \bar{A}_*
\end{equation}
defined as projections onto $\bar{A}_*$ along leaves of the strong stable and strong unstable foliations, are given by 
\begin{equation} \label{eq_holonomymapsdef1}
\bar{\pi}^s_{B} (\theta, y) = (\theta + a_+, y), \quad \bar{\pi}^u_{B} (\theta, y) = (\theta + a_-, y)
\end{equation}
where $a_{\pm} \in \mathbb{R}$ are constants that do not depend on the point $(\theta, y)$.
\end{lemma}
\begin{proof}
Since the orbit $\xi$ is a transverse homoclinic orbit of the geodesic flow in each energy level, it is a transverse homoclinic orbit for the flow of the vector field $Z$ in each energy level. Therefore the stable and unstable manifolds $W^{s,u}(\bar{A}_*)$ intersect transversely along $B$. It is clear that this homoclinic intersection is strongly transverse, as it is strongly transverse for the geodesic flow.

It follows that the holonomy maps are well-defined and invertible on $B$ since every point on the homoclinic orbit $\xi$ is the intersection of exactly one leaf of the strong stable foliation and one leaf of the strong unstable foliation. Moreover they are $C^r$-smooth due to the regularity of the foliations \cite{hirsch1977invariant}. 

Let us introduce the coordinates $(\theta,y)$ on $B$. If $(q,y) \in \mathbb{R} \times \left[ \frac{K}{2}, \frac{3K}{2} \right]$ defines the point $\left(\xi (q), \xi'(q), Ky^{-1} \kappa (\xi (q), \xi' (q))^{- \frac{4}{3}} \right)$ in $B$ then there is a unique $\theta \in \mathbb{R}$ such that 
\begin{equation} \label{eq_defofthetaycoordsonhomocliniccylinder}
\phi^{\theta y^{-1}}_Z \left( \xi (0), \xi'(0), Ky^{-1} \kappa (\xi (0), \xi' (0))^{- \frac{4}{3}} \right) = \left(\xi (q), \xi'(q), Ky^{-1} \kappa (\xi (q), \xi' (q))^{- \frac{4}{3}} \right)
\end{equation}
Notice that, in these coordinates, the restrictions of the vector field $Z$ and its flow $\phi^t_Z$ to the homoclinic manifold $B$ are given by \eqref{eq_cylindervectorfieldeq1} and \eqref{eq_cylinderfloweq1} respectively. For the rest of the proof we work in $(\theta, y)$ coordinates.

Consider the point $(0,y) \in B$, and let $(a_+,y_+) = \bar{\pi}^s_B(0,y)$. Since $y$ is constant along trajectories of the flow, it is constant along leaves of the strong stable foliation. Since the holonomy map is defined via projection along leaves of the strong stable folation, we have $y_+=y$. Now let $(\theta,y) \in B$ be another point with a different $\theta$-value and the same $y$-value. Then there is $t_* \in \mathbb{R}$ such that
\begin{equation} \label{eq_thetabarvalue}
(\theta,y) = \phi^{t_*}_Z (0,y) = (t_*y,y)
\end{equation}
where we have used \eqref{eq_cylinderfloweq1}. From \eqref{eq_holonomycommute}, \eqref{eq_cylinderfloweq1}, \eqref{eq_defofthetaycoordsonhomocliniccylinder}, and \eqref{eq_thetabarvalue} we get
\begin{align}
\bar{\pi}^s_{B} (\theta,y) ={}& \bar{\pi}^s_{B} \circ \phi^{t_*}_Z (0,y) = \phi^{t_*}_Z \circ \bar{\pi}^s_{B} (0,y) = \\
={}&  \phi^{t_*}_Z (a_+,y) = (a_+ + t_* y,y) \\
={}& (\theta + a_+,y)
\end{align}
so $a_+$ does not depend on $\theta$. Moreover it does not depend on $y$ since the dynamics of the flow \eqref{eq_cylinderfloweq1} is the same on every energy level $y=const$. The second equation of \eqref{eq_holonomymapsdef1} is proved analogously.
\end{proof}

Notice that the scattering map relative to the strip $B$ has (infinitely) many branches. The following corollary of Lemma \ref{lemma_homocliniccylinderlemma1} gives us a closed-form expression for the scattering map.

\begin{corollary}
The scattering map
\begin{equation}
\bar{s}_{B} : \bar{A}_* \longrightarrow \bar{A}_*,
\end{equation}
defined by $\bar{s}_{B} (\theta,y) = \bar{\pi}^s_{B} \circ \left( \bar{\pi}^u_{B} \right)^{-1} $, is given by
\begin{equation} \label{eq_bbarstarscatteringmaps}
\bar{s}_{B} (\theta,y) = (\theta + a_*,y)
\end{equation}
where
\begin{equation}
a_* = a_+ - a_-.
\end{equation}
\end{corollary}

Recall that a manifold $\bar{B}_* \subset \left( W^s(\bar{A}_*) \pitchfork W^u(\bar{A}_*) \right) \setminus \bar{A}_*$ is called a homoclinic cylinder if the scattering map
\begin{equation}
\bar{s}_{\bar{B}_*} = \bar{\pi}^s_{\bar{B}_*} \circ \left( \bar{\pi}^u_{\bar{B}_*} \right)^{-1} : \bar{A}_* \longrightarrow \bar{A}_*
\end{equation}
is a diffeomorphism, where
\begin{equation}
\bar{\pi}^{s,u}_{\bar{B}_*} = \left. \bar{\pi}^{s,u}_{B} \right|_{\bar{B}_*} : \bar{B}_* \to \bar{A}_*
\end{equation}
are the holonomy maps. It is clear from Lemma \ref{lemma_homocliniccylinderlemma1} that the set
\small
\begin{equation} \label{eq_homocliniccylinderdef1}
\bar{B}_* = \left\{ (x,w,z)=\phi_Z^{\theta y^{-1}} \left( \xi (0), \xi' (0), K y^{-1} \kappa (\xi (0), \xi' (0))^{- \frac{4}{3}} \right) : \theta \in [0,1), y \in \left[ \frac{K}{2}, \frac{3K}{2} \right] \right\}
\end{equation}
\normalsize
is a homoclinic cylinder to $\bar{A}_*$. Observe that $\bar{A}_*$ is a topological cylinder, while $\bar{B}_*$ is not. 

Let us now discuss the normally hyperbolic invariant cylinder of the billiard map $f$. Recall that $f=g+O(\tau_*^2)$ in $(x,w,z)$ coordinates. As a result of this, it was shown in Section \ref{sec_nearboundary} that there is a normally hyperbolic invariant cylinder $A_*$ for $f$ that is $O(\tau_*)$ close to $\bar{A}_*$. Therefore we only know a priori that $f |_{A_*}$ is $O(\tau_*)$ close to $g |_{\bar{A}_*}$, where $g$ is the time-$\tau_*$ shift of the flow of $Z$. The following lemma improves this approximation; the coordinates $((\theta, y), \alpha, \beta)$ introduced in this lemma are referred to as ``Fenichel coordinates''.

\begin{lemma} \label{lemma_tau2closenessoffandgona}
There are coordinates $((\theta,y), \alpha, \beta)$ in a neighbourhood of $\bar{A}_*$ such that $(\theta,y)$ are coordinates both on $\bar{A}_*$ and on the normally hyperbolic invariant cylinder $A_*$ of the billiard map $f$. Moreover the restriction of $f$ to $A_*$ is given by
\begin{equation} \label{eq_billiardmapinthetaycoords}
f  (\theta, y) = g(\theta, y) + O(\tau_*^2) = (\theta + \tau_* y, y) + O(\tau_*^2)
\end{equation}
where the higher order terms are uniformly bounded in the $C^r$ topology.
\end{lemma}
\begin{proof}
Since $\bar{A}_*$ is an $r$-normally hyperbolic invariant manifold, there is a neighbourhood $U$ of $\bar{A}_*$ in $M$ in which we can introduce $C^r$ coordinates $((\theta, y), \alpha, \beta)$, where $\alpha = (\alpha_1, \ldots, \alpha_{d-2})$ and $\beta = (\beta_1, \ldots, \beta_{d-2})$, which straighten the stable and unstable manifolds of $\bar{A}_*$ in the following sense \cite{jones2009generalized}. Fix a point $(\theta_0, y_0)$ on $\bar{A}_*$. Then we have:
\begin{itemize}
\item
$W^s(\bar{A}_*) = \{ \alpha = 0 \}$
\item
$W^u(\bar{A}_*) = \{ \beta = 0 \}$
\item
$W^s(\theta_0, y_0) = \{ ((\theta, y), \alpha, \beta) \in U : (\theta, y) = (\theta_0, y_0), \alpha = 0 \}$
\item
$W^u(\theta_0, y_0) = \{ ((\theta, y), \alpha, \beta) \in U : (\theta, y) = (\theta_0, y_0), \beta = 0 \}$.
\end{itemize}
The vector field $Z$ in these coordinates is given by $Z((\theta, y), \alpha, \beta) = ((\dot{\theta}, \dot{y}), \dot{\alpha}, \dot{\beta})$ with
\begin{equation} \label{eq_vectorfieldinfenichelcoords}
	\begin{cases}
		\dot{\theta} = y +O(\| \alpha \| \| \beta \|) \\
		\dot{y} = O(\| \alpha \| \| \beta \|)\\
		\dot{\alpha} = h_1 ((\theta,y), \alpha, \beta) \\
		\dot{\beta} = h_2 ((\theta, y), \alpha, \beta)
	\end{cases}
\end{equation}
where
\begin{equation}
h_1 ((\theta,y), 0, \beta) = 0 = h_2 ((\theta,y), \alpha, 0).
\end{equation}
Since this is a $C^r$ change of coordinates, and since $f=g+O(\tau_*^2)$ in $(x,w,z)$ coordinates, we have $f = g + O(\tau_*^2)$ uniformly in $((\theta, y), \alpha, \beta)$ coordinates in the $C^r$ topology.

Notice that, in $((\theta, y), \alpha, \beta)$ coordinates, the normally hyperbolic invariant cylinder of $Z$ is given by
\begin{equation}
\bar{A}_* = \{ ((\theta, y), \alpha, \beta) : \alpha = \beta = 0 \}.
\end{equation}
Fenichel theory implies that the normally hyperbolic invariant cylinder $A_*$ of $f$ can be written as a graph over $\bar{A}_*$. Thus there are maps $\nu, \eta : \mathbb{T} \times \left[ \frac{K}{2}, \frac{3K}{2} \right] \to \mathbb{R}^{d-2}$ that are uniformly $O(\tau_*)$ close to 0 in $C^r$ such that
\begin{equation}
A_* = \left\{ ((\theta, y), \nu (\theta, y), \eta (\theta, y)) : (\theta, y) \in \mathbb{T} \times \left[ \frac{K}{2}, \frac{3K}{2} \right] \right\}.
\end{equation}
It follows that we can use the same $(\theta, y)$ coordinates on both $\bar{A}_*$ and $A_*$, and since $f=g+O(\tau_*^2)$ uniformly in $C^r$ in the Fenichel coordinates $((\theta,y),\alpha,\beta)$, we see, using \eqref{eq_vectorfieldinfenichelcoords}, that 
\begin{equation}
f|_{A_*} ( \theta, y) = g(\theta, y) + O(\tau_*^2, \| \nu \| \| \eta \|) = g(\theta, y) + O(\tau_*^2)
\end{equation}
since $\nu, \eta$ are $O(\tau_*)$ close to 0, which is \eqref{eq_billiardmapinthetaycoords}.
\end{proof}

Now, there is a homoclinic strip $B_*$ to $A_*$ that is $O(\tau_*)$ close to $B$ (where $B$ is the homoclinic strip for the vector field $Z$ defined in Lemma \ref{lemma_homocliniccylinderlemma1}) as a result of the implicit function theorem. The strip $B_*$ is a connected component of the transverse homoclinic intersection of the stable and unstable manifolds of $A_*$, and it is $f$-invariant. Denote by $s_{B_*} : A_* \to A_*'$ the multibranched scattering map defined by the homoclinic strip $B_*$, where $A_*'\subset A$ is an open normally hyperbolic invariant manifold containing $A_*$, and $A$ is the noncompact cylinder comprising the union of all compact invariant cylinders $A_n$. Since the perturbed homoclinic strip $B_*$ is $O(\tau_*)$ close to $B$ in the $C^r$-topology, we have that any particular branch of the map locally satisfies
\begin{equation} \label{eq_billiardscatteringclosetogeodesicscattering}
s_{B_*} (\theta, y) = \bar{s}_{B} (\theta, y) + O(\tau_*) = (\theta + a_*, y) + O(\tau_*)
\end{equation}
uniformly in the $C^r$ topology for $r$ finite. The map $s_{B_*}$ has infinitely many branches along the strip $B_*$. If we restrict the homoclinic channel to subsets of $B_*$, the corresponding scattering map may be a single-branched mapping defined on a subset $U$ of $A_*$, as long as $U$ does not contain an essential curve of $A_*$. However, it is possible that the perturbation from the time-$\tau_*$ map of the vector field $Z$ to the billiard map $f$ introduces monodromy for the scattering map $s_{B_*}$. 

In Section \ref{sec_diffusiveorbits} we introduce a perturbation scheme that requires the scattering map to be globally defined on a cylinder. Since it is not clear that this is the case for $s_{B_*}$, we instead focus on a fundamental domain of $f$ on its normally hyperbolic cylinder $A_*$. We glue the boundaries of the fundamental domain to make it a cylinder, and introduce the \emph{modified} scattering maps, which we prove to be exact symplectic, globally defined diffeomorphisms that are homotopic to the identity on the glued fundamental domain. 

Denote by $D$ the fundamental domain of $f$ on $A_*$ contained between the line $\{\theta = 0\}$ and its image under $f$. We assume that the left boundary $\{ \theta = 0 \}$ is contained in $D$ and the right boundary is not. 

Denote by $(t,y)$ the time-energy coordinates of the vector field $Z$ in a neighbourhood of $D$ in $A_*$, so $y$ is the integral defined by \eqref{eq_rescaledvfzspeed}, $t$ is the time along orbits of the flow of $Z$, the symplectic form is $dt \wedge dy$, and $Z(t,y)=(1,0)$. 
\begin{lemma} \label{lemma_coordsonfundamentaldomain}
We can introduce coordinates $(\hat{\varphi}, \hat{y})$ in a small neighbourhood of the closure of the fundamental domain $D$ in $A_*$ such that
\begin{equation} \label{eq_fundamentaldomaindef}
D = \{ (\hat{\varphi}, \hat{y}): \hat{\varphi} \in [0, \tau_*), \hat{y} \in [a,b] \}
\end{equation}
where $a,b$ are $O(\tau_*)$-close to $\frac{K}{2}$, $\frac{3K}{2}$ respectively, and where the coordinate transformation $(t,y) \mapsto (\hat{\varphi}, \hat{y})$ is $O(\tau_*)$-close to the identity in the $C^1$ topology. Moreover in these coordinates the restriction of the symplectic form to $A_*$ is $d \hat{\varphi} \wedge d \hat{y}$, and the billiard map takes the form
\begin{equation}
f(\hat{\varphi}, \hat{y}) = (\hat{\varphi} + \tau_*, \hat{y}).
\end{equation}
\end{lemma}
\begin{proof}
In the coordinates $(t,y)$, the billiard map takes the form
\begin{equation} \label{eq_billiardmapintimeenergycoordsofvfz}
f : (t,y) \longmapsto (t + \tau_* + \tau_*^2 q(t,y; \tau_*), y + \tau_*^2 p(t, y; \tau_*)),
\end{equation}
for some bounded $C^r$ functions $q,p$. Moreover, $f$ preserves the restriction of the symplectic form $\omega$ to $A_*$. Denote by $\alpha, \beta$ the hyperbolic part of the coordinates (see Lemma \ref{lemma_tau2closenessoffandgona}). Since $\alpha |_{A_*}, \beta |_{A_*} = O(\tau_*)$, $\omega |_{A_*}$ takes the form $h(t,y) dt \wedge dy$ where $h = 1 + O(\tau_*^2)$ in the $C^{r-1}$ topology, because the $\alpha, \beta$-subspaces in the tangent space at each point in $\bar{A}_*$ are $\omega$-orthogonal to $\bar{A}_*$. Let
\begin{equation}
\tilde{y} = \int_0^y h (t,s) ds = y + O(\tau_*^2)
\end{equation}
so that $f$ maintains the form \eqref{eq_billiardmapintimeenergycoordsofvfz} in the new coordinates, and preserves the symplectic form $\omega = dt \wedge d \tilde{y}$. 

Now, in a small neighbourhood of the fundamental domain $D$ in $A_*$, consider a function
\begin{equation}
\check{y} = \tilde{y} + \tau_*^2 p(t, \tilde{y}) \eta (t)
\end{equation}
where $p$ comes from \eqref{eq_billiardmapintimeenergycoordsofvfz}, $\eta$ is a smooth function that is identically equal to 1 in a small neighbourhood of 0, and vanishes identically near the right boundary of $D$. Notice that the $t$-derivative of $\eta$ is of order $\tau_*^{-1}$. Clearly $\check{y}$ is $O(\tau_*^2)$-close to $\tilde{y}$ in $C^0$, and its $\tilde{y}$-derivative is $1 + O(\tau_*^2)$, but its $t$-derivative is $O(\tau_*)$. By construction, we have
\begin{equation}
\check{y} (t, \tilde{y})= \check{y} (f(t, \tilde{y})),
\end{equation}
so $\check{y}$ is an integral of $f$ in a neighbourhood of the left boundary of $D$. In the variables $(t, \check{y})$ $f$ takes the form
\begin{equation}
f: (t, \check{y}) \longmapsto (t + \tau_* + \tau_*^2 r (t, \check{y}; \tau_*), \check{y}),
\end{equation}
where $r$ is bounded along with its first derivatives. Moreover the symplectic form is 
\begin{equation}
dt \wedge d \tilde{y} = \left( \frac{\partial \check{y}}{\partial \tilde{y}} \right)^{-1} dt \wedge d \check{y} = d \varphi \wedge d \check{y}
\end{equation}
where the coordinate $\varphi$ is defined by
\begin{equation}
\varphi = \int_0^t \left( \frac{\partial \check{y}}{\partial \tilde{y}} (\tilde{t}, \tilde{y}) \right)^{-1} d \tilde{t}.
\end{equation}
Then $f$ preserves the standard symplectic form in the new coordinates, and is given by
\begin{equation}
f: (\varphi, \check{y}) \longmapsto (\varphi + \tau_* + \tau_*^2 F(\check{y}), \check{y}).
\end{equation}
Finally, replacing $\varphi, \check{y}$ by 
\begin{equation}
\hat{\varphi} = \left( 1 + \tau_* F(\check{y}) \right)^{-1} \varphi, \quad \hat{y} = \check{y} + \tau_* \int_0^{\check{y}} F(s) ds
\end{equation}
respectively, we obtain coordinates in which the symplectic form is $d \hat{\varphi} \wedge d \hat{y}$, and $f$ takes the form
\begin{equation}
f: (\hat{\varphi}, \hat{y}) \longmapsto (\hat{\varphi} + \tau_*, \hat{y}).
\end{equation}
Since each coordinate transformation we have made is (at least) $O(\tau_*)$ close to the identity in $C^1$, the lemma is proved. 
\end{proof}

Now, drop the `hat' notation, and write $(\varphi,y)$ instead of $(\hat{\varphi},\hat{y})$. Since $D$ is a fundamental domain of the near-identity map $f$ on the cylinder $A_*$, we can find a (connected) preimage $\Delta = \left( \pi_{B_*}^u \right)^{-1} \left(D \right)$ of $D$ on the homoclinic strip $B_*$. Then the scattering map $s_{B_*}$, relative to a neighbourhood of $\Delta$ in $B_*$, is an exact symplectic diffeomorphism onto its image $s_{B_*} (D)$ when restricted to a neighbourhood of $D$ in $A_*$. This is due to results of \cite{delshams2008geometric} combined with the facts that $f$ is an exact symplectic diffeomorphism, and that the restriction of the symplectic form $\omega$ to $A_*$ is nondegenerate. Denote by $D'$ the extension of the fundamental domain $D$ to the open normally hyperbolic invariant manifold $A_*'$. Define the function $n: s_{B_*} \left( D \right)\to \mathbb{N}$ by
\begin{equation}
n \left( s_{B_*} (\varphi, y) \right) = \min \{ m>0 : f^m \left( s_{B_*} (\varphi,y) \right) \in D' \}
\end{equation}
for $(\varphi,y) \in D$. This is well-defined because $D'$ is a fundamental domain of $f$ in $A_*'$. We then define the \emph{modified scattering map} $\tilde{s} : D \to D'$ by
\begin{equation} \label{eq_defofmodifiedscatteringmap}
\tilde{s} (\varphi, y) = f^{n \circ s_{B_*} (\varphi, y)} \circ s_{B_*} (\varphi, y).
\end{equation}
The following result follows from the definition of the modified scattering map. As the proof is lengthy, it is deferred to an appendix. 
\begin{lemma} \label{lemma_modscatisexactsymplectic}
If we consider $D$ as a cylinder by identifying points on the line $\{\varphi=0\}$ with their image under $f$, and use the coordinates $(\varphi,y) \in [0, \tau_*) \times \left[ a,b \right]$, then the modified scattering map $\tilde{s} : D \to \tilde{s}(D) \subset D'$ is an exact symplectic $C^r$-diffeomorphism onto its image, and its image contains an essential curve in $D'$.
\end{lemma}
\begin{remark}
In the statement of Lemma \ref{lemma_modscatisexactsymplectic}, it is claimed that the modified scattering map is a $C^r$-diffeomorphism. If the cylinder $A_*$ is only $C^r$ smooth, then actually the scattering map is only $C^{r-1}$, and so the modified scattering map is also only $C^{r-1}$. However, we are able to choose $r$ to be any finite positive integer (shrinking $\tau_*$ if necessary). Therefore if we initially choose $r+1$, then the modified scattering map is $C^r$ smooth, as stated in the lemma.
\end{remark}
In the coming chapters we analyse the effect that a perturbation of the hypersurface $\Gamma$ near a fundamental domain of $f$ on the homoclinic strip $B_*$ has on the modified scattering map.

Now, the billiard map $f$ does not map the fundamental domain $D$ to itself. Let us define the inner map $\Phi : D \to D$ as follows. For $(\varphi,y) \in D$ let
\begin{equation}
N = N( \varphi,y) = \min \{ m > 0 : f^m (\varphi,y) \in D \}.
\end{equation}
This is well-defined since $D$ is a fundamental domain of $f$ on $A_*$. Then we define
\begin{equation}\label{eq_defofinnermaponfd}
\Phi (\varphi,y) = f^{N(\varphi,y)} (\varphi,y). 
\end{equation}

\begin{lemma}\label{lemma_innermaptwist}
The map $\Phi : D \to D$ is an exact symplectic $C^r$-diffeomorphism and satisfies a twist property. 
\end{lemma}

\begin{proof}
The proof that $\Phi$ is an exact symplectic diffeomorphism is similar to the proof of Lemma \ref{lemma_modscatisexactsymplectic}, and so we do not include it here. We show that $\Phi$ satisfies a twist property. 

Denote by
\begin{equation}
\bar{D} = \left\{ \left(\theta, y \right) \in [0, \tau_* y) \times \left[ \frac{K}{2}, \frac{3K}{2} \right] \right\} \subset \bar{A}_*
\end{equation}
the fundamental domain of $g = \phi^{\tau_*}_Z$ in $\bar{A}_*$, where $(\theta, y)$ are the coordinates on the cylinder $\bar{A}_*$ defined at the beginning of the section. Define
\begin{equation}
\bar{N} = \bar{N}( \theta,y) = \min \{ m > 0 : g^m (\theta,y) \in \bar{D} \} = O(\tau_*^{-1})
\end{equation}
and define the return map of $g$ to $\bar{D}$ as
\begin{equation}
\Psi (\theta, y) = g^{\bar{N}(\theta, y)} (\theta, y). 
\end{equation}

Let $(\theta_0,y_0) \in \bar{D}$. Then there is a unique $t_0 \in [0, \tau_*)$ such that 
\begin{equation}
(\theta_0,y_0) = \phi_Z^{t_0} (0, y_0) = (t_0 \, y_0, y_0).
\end{equation} 
Therefore $t_0 = \theta_0 \, y_0^{-1}$. Recall that the period of $(\theta_0,y_0)$ under the flow of $Z$ is $y_0^{-1}$. It follows that the image of $(\theta_0,y_0)$ under the time-$(y_0^{-1} - t_0)$ map of $Z$ is the point $(0, y_0)$. Let $t_1= (t_0 - y_0^{-1}) \mod \tau_*$. Then $\tau_* \bar{N} (\theta_0,y_0) = (y_0^{-1} - t_0) + t_1$, and so
\begin{equation}
(\theta_1, y_1) = \Psi (\theta_0, y_0) = \phi^{t_1}_Z (0, y_0). 
\end{equation}
We have
\begin{equation}
\frac{ \partial t_1}{\partial y_0} = y_0^{-2}  > 0
\end{equation}
and so $\Psi$ satisfies a twist condition in time-energy coordinates. By Lemma \ref{lemma_tau2closenessoffandgona}, $f,g$ are $O(\tau_*^2)$-close. Therefore $\Psi, \Phi$ are $O(\tau_*)$-close, and so $\Phi$ also satisfies a twist condition. Since the coordinate transformation is $O(\tau_*)$-close to the identity, the twist condition is also satisfied in $(\hat{\varphi}, \hat{y})$ coordinates (see Lemma \ref{lemma_coordsonfundamentaldomain}).
\end{proof}

\subsection{Shadowing of Pseudo-Orbits} \label{sec_shadowing}

Let $\{ A_n \}_{n \in \mathbb{N}}$ denote the sequence of normally hyperbolic invariant cylinders for the billiard map $f$ obtained in Proposition \ref{proposition_noncompactcylinder}, and let
\begin{equation}
A = \bigcup_{n \in \mathbb{N}} A_n
\end{equation}
denote the noncompact cylinder.

Standard results imply that the existence of a single transverse homoclinic geodesic $\xi$ implies the existence of infinitely many (see e.g. Theorem 6.5.5 of \cite{katok1995introduction}). Suppose we take $N$ such geodesics $\xi_1, \ldots, \xi_N$. Let $B_1, \ldots, B_N$ denote the corresponding homoclinic strips for the billiard map, and consider the scattering maps $s_j = s_{B_j}$. We now construct an iterated function system $\left\{ f|_{A}, s_1, \ldots, s_N \right\}$ on $A$, and call orbits of the IFS pseudo-orbits.

The following theorem is a generalisation of a theorem of Moeckel \cite{moeckel2002generic}. See also \cite{gelfreich2017arnold, le2007drift}. The proof here relies on the global definition of the modified scattering maps, but recent progress has relaxed this assumption \cite{gidea2017diffusion}. The theorem tells us that, given a prescribed sequence $\{A_{m_n} \}$ of consecutive cylinders, we can find a pseudo-orbit $\{ \mu_n \}$ visiting the cylinders $A_{m_n}$ in the prescribed order.

\begin{theorem}{(Diffusive and Oscillatory Pseudo-Orbits)} \label{theorem_gtdiffusion}
For each $n \in \mathbb{N}$, denote by $D_n$ the fundamental domain in $A_n$ defined by \eqref{eq_fundamentaldomaindef}. Denote by $\tilde{s}_j$ the modified scattering maps, and by $\Phi$ the inner map, defined by \eqref{eq_defofmodifiedscatteringmap}, \eqref{eq_defofinnermaponfd}, respectively. Suppose for each $n \in \mathbb{N}$ the maps $\Phi, \tilde{s}_1, \ldots, \tilde{s}_N$ have no common invariant essential curves on the glued fundamental domain $D_n$.

Let $\{ m_n \}_{n \in \mathbb{N}} \subseteq \mathbb{N}$ be such that $m_{n+1} = m_n \pm 1$ for each $n$. Then there is a pseudo-orbit $\{ \mu_n \}_{n \in \mathbb{N}} \subset A$ of the iterated function system $f, s_1, \ldots, s_N$ and $\{ k_n \}_{n \in \mathbb{N}} \subset \mathbb{N}$ with
\begin{equation}
k_{n+1} > k_n, \quad \mu_{k_n} \in A_{m_n}
\end{equation}
for each $n \in \mathbb{N}$, and
\begin{equation}
\mu_j \in A_{m_n}
\end{equation}
for each $j =  k_n +1, k_n +2, \ldots, k_{n+1}-1$.
\end{theorem}

\begin{proof}
The theorem is proved by applying the proof of Theorem 3 of \cite{gelfreich2017arnold} inductively to the fundamental domain $D_{m_n}$ in each compact invariant subcylinder $A_{m_n}$. Indeed, it follows from Lemma \ref{lemma_innermaptwist} that the inner map $\Phi$ is an exact symplectic twist map, and from Lemma \ref{lemma_modscatisexactsymplectic} that the modified scattering maps $\tilde{s}_j$ are exact symplectic diffeomorphisms (onto their images), and are homotopic to the identity. Therefore the assumptions of Theorem 3 of \cite{gelfreich2017arnold} are satisfied.

Let $\gamma_{\pm}$ denote the essential curves bounding the bottom and top of the cylinder $D_{m_1}$, and assume $m_2 = m_1 +1$. Construct a curve $\tilde{\gamma}$ as follows: the curves $\Phi(\gamma_-), \tilde{s}_1 (\gamma_-), \ldots$, $\tilde{s}_N (\gamma_-)$ divide $D_{m_1}$ into connected components. If one of these curves intersects $\gamma_+$ then we have already found a pseudo-orbit connecting $A_{m_1}$ and $A_{m_2}$. If not, let $\tilde{\gamma}$ denote the lower boundary of the connected component of $D_{m_1}$ whose upper boundary is $\gamma_+$. Write $\tilde{\gamma} = F({\gamma_-})$. Here $F$ is just an operation on closed curves. By Theorem 3 of \cite{gelfreich2017arnold}, by iterating $F$ on $\gamma_-$, we can obtain a curve $F^{n_1} (\gamma_-)$ that intersects $\gamma_+$, which in turn gives a pseudo-orbit connecting the bottom of $D_{m_1}$ to the top. Since $A_{m_1}$ and $A_{m_2}$ overlap (see Proposition \ref{proposition_noncompactcylinder}), the connecting pseudo-orbit lands in $D_{m_2}$, since it is a fundamental domain of $f$ in $A_{m_2}$. Therefore we can perform the process again to show that there is some $n_2 \in \mathbb{N}$ such that $F^{n_1 + n_2} (\gamma_-)$ intersects the upper boundary curve of $D_{m_2}$. Similarly if $m_2 = m_1 -1$ we can find pseudo-orbits starting on $\gamma_+$ and ending on $\gamma_-$, thus connecting $A_{m_1}$ and $A_{m_2}$ in that case. Since orbits of the iterated function system $\Phi, \tilde{s}_1, \ldots, \tilde{s}_N$ correspond to orbits of the iterated function system $f, s_1, \ldots, s_N$, we can repeat the process inductively to obtain Theorem \ref{theorem_gtdiffusion}.
 \end{proof}
 
\begin{lemma}{(Shadowing Lemma)} \label{lemma_gtshadowing}
Let $\epsilon >0$, and let $\{ \mu_n \}_{n=0}^{\infty}$ be a pseudo-orbit, so for each $n=0,1, \ldots$, we have either $\mu_{n+1} = f(\mu_n)$, or there is an $\alpha_n \in \{ 1, \ldots, N \}$ such that
\begin{equation}
\mu_{n+1} = s_{\alpha_n} (\mu_n). 
\end{equation}
Then there is an orbit $\{ \nu_n \}_{n=0}^{\infty}$ of $f$ in $M$ such that for some $\{ m_n \}_{n=0}^{\infty} \subset \mathbb{N}$ we have
\begin{equation}
\nu_{n+1} = f^{m_n}(\nu_n)
\end{equation}
and
\begin{equation}
\| \mu_n - \nu_n \| \leq \epsilon.
\end{equation}
\end{lemma}
The Lemma is equivalent to Lemma 3 of \cite{gelfreich2017arnold} and Lemma 3.11 of \cite{gidea2014general}, so we do not include it here. In those papers it is assumed that the manifold $A$ is compact. The purpose of compactness is to allow application of the Poincar\'e recurrence theorem. In our case, the structure of $A$ as a countable union of compact invariant subcylinders allows us to adapt the proof as we can apply the recurrence theorem in each compact invariant subcylinder.

It follows from Theorem \ref{theorem_gtdiffusion} and Lemma \ref{lemma_gtshadowing} that if we can show that the property of the inner map $\Phi$ and the $N$ modified scattering maps having no common invariant essential curves in each fundamental domain $D_n$ in $A$ is a residual property of $Q \in \mathcal{V}$, then Theorems \ref{theorem_main} and \ref{theorem_main2} are proved.

\section{Effect of a Perturbation on the Billiard and Scattering Maps} \label{sec_perturbations}
\subsection{Perturbing the Billiard Map} \label{sec_billiardperturbation}
In order to perturb the billiard map, we perturb the surface $\Gamma$. Recall that $\Gamma$ is given implicitly as the set of zeros of a real-analytic function $Q \in \mathcal{V}$. In this section, we will show how the billiard map is affected by compactly supported perturbations of the surface $\Gamma$. Since such perturbations are not real-analytic, we will later approximate the perturbed system $Q+ \epsilon \psi$ by a real-analytic family $Q_{\epsilon} \subset \mathcal{V}$.

Fix some $(x,u) \in M$ such that $\| u \| < 1$ and let $(\bar{x}, \bar{u}) = f (x,u)$. Let $\epsilon > 0$ and $\psi : \mathbb{R}^d \to \mathbb{R}$ be a smooth function supported in a neighbourhood of $\bar{x}$. In particular, assume $x \notin \mathrm{supp}(\psi)$. Make the perturbation
\begin{equation} \label{eq_perturbation}
Q \longrightarrow Q_{\epsilon} = Q + \epsilon \psi.
\end{equation}
Denote by $f_{\epsilon}$ the perturbed billiard map, and let $(\tilde{x}, \tilde{u}) = f_{\epsilon} (x,u)$. Referring to \eqref{eq_billiardmap}, we have
\begin{equation} \label{eq_perturbedbilliardmap}
\begin{cases}
	\tilde{x} = x + \tau_{\epsilon} (x,u) \, v \\
	\tilde{u} = v - \langle v, \tilde{n} (\tilde{x}) \rangle \, \tilde{n} (\tilde{x}),
\end{cases}
\end{equation}
where $\tau_{\epsilon} (x,u) =  t > 0$ such that $Q_{\epsilon} (x + t v) = 0$, and
\begin{equation} \label{eq_perturbednormal}
\tilde{n} (x) = - \frac{\nabla Q_{\epsilon} (x)}{\| \nabla Q_{\epsilon} (x) \| }.
\end{equation}
Since $\epsilon$ is sufficiently small, the new surface is still strictly convex and so $\tau_{\epsilon}$ is still well-defined. Clearly for the $\tilde{x}$-component, the perturbation is manifested only through $\tau_{\epsilon}$. Therefore the term of order $\epsilon$ in the Taylor series of $\tilde{x}$ is $\Theta (\bar{x}, \bar{u}) v$, where $\Theta : M \to \mathbb{R}$ is a smooth function. We determine $\Theta$ by computing the first term in the Taylor series of $Q_{\epsilon} (\tilde{x})=0$. Since this term must itself vanish, we have 
\begin{align}
0 &= \left. \frac{d}{d \epsilon} \right|_{\epsilon = 0} Q_{\epsilon} (\tilde{x}) = \left. \frac{d}{d \epsilon} \right|_{\epsilon = 0} \left[ Q(\bar{x} + \epsilon \Theta (\bar{x}, \bar{u}) v + O(\epsilon^2)) + \epsilon \psi (\tilde{x}) \right] \\[0.5em]
& = \left. \left[ \Theta (\bar{x}, \bar{u}) \langle \nabla Q (\tilde{x}),  v \rangle + \psi( \tilde{x}) + O(\epsilon) 			\right] \right|_{\epsilon = 0} \\[0.5em]
& = -\Theta (\bar{x}, \bar{u}) \| \nabla Q (\bar{x}) \| \langle n (\bar{x}), v \rangle + \psi (\bar{x})
\end{align}
Elementary geometrical considerations imply that we can write $v$ in terms of $(\bar{x}, \bar{u})$ as $v = \bar{u} - \sqrt{1- \bar{u}^2} \, n(\bar{x})$. Therefore we have 
\begin{equation} \label{eq_thetaexpression}
\Theta (\bar{x}, \bar{u}) = - \frac{\psi (\bar{x})}{\| \nabla Q (\bar{x}) \| \sqrt{1 - \bar{u}^2}}.
\end{equation}
In order to compute the effect of the perturbation on the $u$-component, we write $\tilde{u} = \bar{u} + \epsilon \Pi (\bar{x}, \bar{u}) + O ( \epsilon^2)$. This computation is more direct: we first compute the expansion of $\tilde{n} (\tilde{x})$ up to first order, using \eqref{eq_perturbednormal} and \eqref{eq_thetaexpression}, and then substitute this into the expression for $\tilde{u}$ in \eqref{eq_perturbedbilliardmap}. Write $\tilde{n} (\tilde{x}) = n (\bar{x}) + \epsilon V(\bar{x}, \bar{u}) + O(\epsilon^2)$. A straightforward computation shows that
\begin{align} \label{eq_Vexpression}
V(\bar{x}, \bar{u}) &= \left. \frac{d}{d \epsilon} \right|_{\epsilon = 0} \tilde{n} (\tilde{x}) \\
&= - \Theta(\bar{x}, \bar{u}) S(\bar{x}) v - \| \nabla Q( \bar{x}) \|^{-1} [ \nabla \psi (\bar{x}) - \langle \nabla \psi (\bar{x}), n (\bar{x}) \rangle \, n ( \bar{x}) ].
\end{align}
Notice that $V(\bar{x}, \bar{u}) \in T_{\bar{x}} \Gamma$. It follows from \eqref{eq_perturbedbilliardmap} that
\begin{align}
\tilde{u} &= v - \langle v, n (\bar{x}) + \epsilon V(\bar{x}, \bar{u})\rangle (n(\bar{x}) + \epsilon V(\bar{x}, \bar{u})) + O (\epsilon^2) \\
&= v - \langle v, n(\bar{x}) \rangle \, n(\bar{x}) - \epsilon [ \langle v, V(\bar{x}, \bar{u}) \rangle \, n (\bar{x}) + \langle v, n(\bar{x}) \rangle \, V(\bar{x}, \bar{u})] + O (\epsilon^2) \\
&= \bar{u} + \epsilon \left[ \sqrt{1 - \bar{u}^2} \, V(\bar{x}, \bar{u}) - \langle v, V(\bar{x}, \bar{u}) \rangle \, n(\bar{x}) \right] + O(\epsilon^2).
\end{align}
We compute
\begin{equation}
\begin{split}
\langle v, V (\bar{x}, \bar{u}) \rangle ={}& \langle \bar{u}, V (\bar{x}, \bar{u}) \rangle \\
={}& -\Theta (\bar{x}, \bar{u}) \kappa (\bar{x}, \bar{u}) + \sqrt{1 - \bar{u}^2} \, \Theta(\bar{x}, \bar{u}) \langle C(\bar{x}) \bar{u}, n(\bar{x}) \rangle \\
& - \| \nabla Q (\bar{x}) \|^{-1} \langle \nabla \psi (\bar{x}), \bar{u} \rangle
\end{split}
\end{equation}
We thus obtain the following expressions for $\Pi$:
\begin{align}
\begin{split} \label{eq_piexpression1}
\Pi (\bar{x}, \bar{u}) ={}& \sqrt{1 - \bar{u}^2} \, V( \bar{x}, \bar{u}) - \langle v, V(\bar{x}, \bar{u}) \rangle \, n(\bar{x}) 
\end{split} \\
\begin{split} \label{eq_piexpression2}
={}& - \sqrt{1 - \bar{u}^2} \, \Theta (\bar{x}, \bar{u}) S(\bar{x})v - \\
	& - \sqrt{1 - \bar{u}^2} \, \| \nabla Q(\bar{x}) \|^{-1} \left[ \nabla \psi (\bar{x}) - \langle \nabla \psi (\bar{x}), n 			(\bar{x}) \rangle \, n(\bar{x}) \right] +  \\
	& + \Big[ \Theta (\bar{x}, \bar{u}) \kappa (\bar{x}, \bar{u}) - \Theta (\bar{x}, \bar{u}) \sqrt{1 - \bar{u}^2} 		\langle C(\bar{x}) \bar{u}, n (\bar{x}) \rangle + \\
	& + \| \nabla Q(\bar{x}) \|^{-1} \langle \nabla \psi (\bar{x}), \bar{u} \rangle \Big] n(\bar{x}) 
\end{split}
\end{align}
Equation \eqref{eq_piexpression1} will be useful later since $V(\bar{x}, \bar{u}) \in T_{\bar{x}} \Gamma$, whereas the second term is normal to $\Gamma$ at $\bar{x}$. 

Now, let
\begin{equation} \label{eq_vfofperturbation0}
\tilde{X} (\bar{x}, \bar{u}) = \left( \Theta (\bar{x}, \bar{u}) v, \Pi (\bar{x}, \bar{u}) \right).
\end{equation}
The time-$\epsilon$ shift of the flow of $\tilde{X}$ maps the phase space $M=M(Q)$ of the unperturbed billiard to the phase space $M_{\epsilon}=M(Q+ \epsilon \psi)$ of the perturbed billiard. The following lemma provides us with a way of describing the dynamics of the perturbed billiard map $f_{\epsilon}$ in terms of the coordinates on the unperturbed phase space $M$. 

\begin{lemma}
There is a smooth symplectic map $h: M \to M$ with the following properties.
\begin{enumerate}
\item
The map $h$ is $O(\epsilon)$ close to the identity in the $C^r$ topology, and coincides with the identity at points $(x,u) \in M$ where $x \notin \mathrm{supp}(\psi)$. 
\item
For $(x,u) \in M$ and $n \in \mathbb{N}$ such that $f_{\epsilon}^n (x,u) \in M$, we have
\begin{equation}
f_{\epsilon}^n (x,u) = \left(h \circ f \right)^n (x,u).
\end{equation}
\item
The map $h$ is approximated up to terms of order $\epsilon^2$ by the time-$\epsilon$ shift of the flow of the Hamiltonian function
\begin{align} 
\begin{split} \label{eq_perturbationhamiltonian}
H_{\mathrm{pert}} (\bar{x}, \bar{u}) ={}& 2 \frac{ \psi (\bar{x})}{\| \nabla Q(\bar{x}) \|} \sqrt{1 - \bar{u}^2} + \Big[ 2 \Theta (\bar{x}, \bar{u}) \left\langle C(\bar{x}) v, \bar{v} \right\rangle + \\
& \quad +2 \| \nabla Q (\bar{x}) \|^{-1} \sqrt{1 - \bar{u}^2} \left\langle \nabla \psi (\bar{x}), n(\bar{x}) \right\rangle \Big] \frac{Q(\bar{x})}{\| \nabla Q(\bar{x}) \|}.
\end{split}
\end{align}
\end{enumerate}
\end{lemma}

\begin{proof}
Let $(x^0,u^0) \in M $ such that $x^0 \notin \mathrm{supp}(\psi)$, and $\bar{x} \in \mathrm{supp} (\psi)$ where $(\bar{x}, \bar{u}) = f(x^0,u^0)$. Let $(\tilde{x}, \tilde{u}) = f_{\epsilon} (x^0, u^0)$, $(x^1, u^1) = f_{\epsilon}(\tilde{x}, \tilde{u})$, and $(\hat{x}, \hat{u}) = f^{-1} (x^1, u^1)= \I \circ f \circ \I (x^1, u^1)$ where the involution operator $\I :  T \mathbb{R}^d \longrightarrow  T \mathbb{R}^d$ sends $ (x,u)\longmapsto (x,-u)$. We also assume that $x^1 \notin \mathrm{supp}(\psi)$, as the remaining case can be treated similarly. Consider the near-to-the-identity maps
\begin{equation}
h_0 : (\bar{x},\bar{u}) \mapsto (\tilde{x}, \tilde{u}), \quad h_1 : (\tilde{x}, \tilde{u}) \mapsto (\hat{x}, \hat{u}). 
\end{equation}
Let
\begin{equation}
h = h_1 \circ h_0.
\end{equation}
By construction we have 
\begin{equation}
h_0^* \left( \left. \omega \right|_M \right) = \left(f_{\epsilon} \circ f^{-1} \right)^* \left( \left. \omega \right|_M \right) = \left. \omega \right|_{M_{\epsilon}}
\end{equation}
and
\begin{equation}
h_1^* \left( \left. \omega \right|_{M_{\epsilon}} \right) = \left(f^{-1} \circ f_{\epsilon} \right)^* \left( \left. \omega \right|_{M_{\epsilon}} \right) = \left. \omega \right|_M
\end{equation}
where $\omega = dx \wedge du$ is the standard symplectic form on $\mathbb{R}^{2d}$. Therefore $h$ is symplectic, and so is approximated up to terms of order $\epsilon^2$ by the time-$\epsilon$ shift of some Hamiltonian vector field on $M$ with Hamiltonian $H_{\mathrm{pert}}$ which we must determine. 

The map $h_0$ is approximated up to terms of order $\epsilon^2$ by the time-$\epsilon$ map of the vector field $\tilde{X}$, defined by \eqref{eq_vfofperturbation0}. Let
\begin{equation}
v_{\pm} (\bar{x}, \bar{u}) = \bar{u} \pm \sqrt{1 - \bar{u}^2} \, n(\bar{x}),
\end{equation}
and denote by
\begin{equation}
\nabla \psi (\bar{x})^T = \nabla \psi (\bar{x}) - \langle \nabla \psi (\bar{x}), n(\bar{x}) \rangle \, n(\bar{x})
\end{equation}
the tangential component of $\nabla \psi (\bar{x})$. It can be seen that
\begin{equation}
h_1 = \I \circ h_0^{-1} \circ \I.
\end{equation}
Denote by $\tilde{\phi}^t$ the flow defined by 
\begin{equation}
\left. \frac{d}{dt} \right|_{t=0} \tilde{\phi}^t = \tilde{X}.
\end{equation}
We have
\begin{equation}
h_0^{-1} = \tilde{\phi}^{- \epsilon} + O(\epsilon^2) = Id - \epsilon \tilde{X} + O(\epsilon^2). 
\end{equation}
Therefore 
\begin{align}
h_1 (\tilde{x}, \tilde{u}) ={}& \I \circ h_0^{-1} (\tilde{x}, - \tilde{u}) \\
={}&(\tilde{x}, \tilde{u}) - \epsilon \left( \Theta (\tilde{x}, - \tilde{u}) v_- (\tilde{x}, - \tilde{u}), - \Pi (\tilde{x}, - \tilde{u}) \right) + O(\epsilon^2)
\end{align}
and so
\begin{equation}
h (\bar{x}, \bar{u}) = (\bar{x}, \bar{u}) + \epsilon X_{\mathrm{pert}} (\bar{x}, \bar{u}) + O(\epsilon^2)
\end{equation}
where
\begin{equation}
X_{\mathrm{pert}} (\bar{x}, \bar{u}) = (\Theta (\bar{x}, \bar{u}) v_- (\bar{x}, \bar{u}) - \Theta (\bar{x}, - \bar{u}) v_- (\bar{x}, - \bar{u}), \Pi (\bar{x}, \bar{u}) + \Pi (\bar{x}, - \bar{u})).
\end{equation}
We must compute $X_{\mathrm{pert}}$ and its Hamiltonian. We have
\begin{equation}
\Theta  (\bar{x}, \bar{u}) = \Theta  (\bar{x}, - \bar{u}), \quad v_- (\bar{x}, - \bar{u}) = - v_+ (\bar{x}, \bar{u}),
\end{equation}
and
\begin{equation}
V (\bar{x}, - \bar{u}) = \Theta (\bar{x}, \bar{u}) S(\bar{x}) v_+ (\bar{x}, \bar{u}) - \| \nabla Q (\bar{x}) \|^{-1} \nabla \psi (\bar{x})^T. 
\end{equation}
Therefore
\begin{equation}
\Theta (\bar{x}, \bar{u}) v_- (\bar{x}, \bar{u}) - \Theta (\bar{x}, - \bar{u}) v_- (\bar{x}, - \bar{u}) = 2 \Theta (\bar{x}, \bar{u}) \bar{u},
\end{equation}
\begin{equation}
V (\bar{x}, \bar{u}) + V (\bar{x}, - \bar{u}) = 2 \sqrt{1 - \bar{u}^2} \Theta (\bar{x}, \bar{u}) S(\bar{x}) n(\bar{x}) - 2 \| \nabla Q (\bar{x}) \|^{-1} \nabla \psi (\bar{x})^T
\end{equation}
and
\begin{align}
\left\langle v_- (\bar{x}, \bar{u}), V (\bar{x}, \bar{u}) \right\rangle + \left\langle v_- (\bar{x}, - \bar{u}), V (\bar{x}, - \bar{u}) \right\rangle ={}& - \Theta (\bar{x}, \bar{u}) \left\langle S(\bar{x}) v_- (\bar{x}, \bar{u}), \bar{u} \right\rangle - \\
& \quad -\Theta (\bar{x}, \bar{u}) \left\langle S(\bar{x}) v_+ (\bar{x}, \bar{u}), \bar{u} \right\rangle \\
={}& - 2 \Theta  (\bar{x}, \bar{u}) \kappa  (\bar{x}, \bar{u}).
\end{align}
It follows that
\begin{align}
\Pi  (\bar{x}, \bar{u}) + \Pi  (\bar{x}, - \bar{u}) ={}& 2 (1 - \bar{u}^2) \Theta  (\bar{x}, \bar{u}) S(\bar{x}) n(\bar{x}) - \\
& \quad - 2 \sqrt{1 - \bar{u}^2} \| \nabla Q (\bar{x}) \|^{-1} \nabla \psi (\bar{x})^T + 2 \Theta (\bar{x}, \bar{u}) \kappa (\bar{x}, \bar{u}) n(\bar{x}). 
\end{align}
Let
\begin{equation}
A = 2 \frac{\psi (\bar{x})}{\| \nabla Q(\bar{x}) \|} \sqrt{1 - \bar{u}^2}
\end{equation}
so that
\begin{equation}
\frac{\partial A}{\partial \bar{x}} = 2 \sqrt{1 - \bar{u}^2} \left[ \frac{\nabla \psi (\bar{x})}{\| \nabla Q(\bar{x} \|} + \frac{\psi (\bar{x})}{\| \nabla Q(\bar{x} \|}  C (\bar{x}) n(\bar{x}) \right]
\end{equation}
and
\begin{equation}
\frac{\partial A}{\partial \bar{u}} = - 2 \frac{\psi (\bar{x}) }{\| \nabla Q(\bar{x}) \| \sqrt{1 - \bar{u}^2}} \bar{u} = 2 \Theta (\bar{x}, \bar{u}) \bar{u}.
\end{equation}
Let 
\begin{equation}
H_{\mathrm{pert}} = A+B \frac{Q(\bar{x})}{\| \nabla Q (\bar{x}) \|}
\end{equation}
where $B$ is to be determined. Then, on $M$, 
\begin{equation}
\frac{\partial H_{\mathrm{pert}}}{\partial \bar{u}} = 2 \Theta (\bar{x}, \bar{u}) \bar{u}.
\end{equation}
We choose $B$ so that, on $M$, 
\begin{equation}
- \Pi (\bar{x}, \bar{u}) - \Pi (\bar{x}, - \bar{u}) = \frac{\partial H_{\mathrm{pert}}}{\partial \bar{x}} = \frac{\partial A}{\partial \bar{x}} - B n(\bar{x}). 
\end{equation}
Therefore
\begin{align}
B n(\bar{x}) ={}& \Pi (\bar{x}, \bar{u}) + \Pi (\bar{x}, - \bar{u}) + \frac{\partial A}{\partial \bar{x}} \\
={}& - 2 \left(1 - \bar{u}^2 \right) \Theta (\bar{x}, \bar{u}) \left\langle C(\bar{x}) n(\bar{x}), n(\bar{x}) \right\rangle n(\bar{x}) + \\
& \quad + 2 \| \nabla Q(\bar{x}) \|^{-1} \sqrt{1 - \bar{u}^2} \left\langle \nabla \psi (\bar{x}), n(\bar{x}) \right\rangle n(\bar{x}) + 2 \Theta (\bar{x}, \bar{u}) \kappa (\bar{x}, \bar{u}) n(\bar{x}). 
\end{align}
Since
\begin{equation}
\left\langle C(\bar{x}) v, \bar{v} \right\rangle = \left\langle C(\bar{x}) v_-(\bar{x}, \bar{u}), v_+(\bar{x}, \bar{u}) \right\rangle = \kappa (\bar{x}, \bar{u}) - \left( 1 - \bar{u}^2 \right) \left\langle C(\bar{x}) n(\bar{x}), n(\bar{x}) \right\rangle
\end{equation}
we find that
\begin{align}
H_{\mathrm{pert}} (\bar{x}, \bar{u}) ={}& 2 \frac{ \psi (\bar{x})}{\| \nabla Q(\bar{x}) \|} \sqrt{1 - \bar{u}^2} + \Big[ 2 \Theta (\bar{x}, \bar{u}) \left\langle C(\bar{x}) v, \bar{v} \right\rangle + \\
& \quad +2 \| \nabla Q (\bar{x}) \|^{-1} \sqrt{1 - \bar{u}^2} \left\langle \nabla \psi (\bar{x}), n(\bar{x}) \right\rangle \Big] \frac{Q(\bar{x})}{\| \nabla Q(\bar{x}) \|}.
\end{align}
\end{proof}

\subsection{Perturbative Computation of the Scattering Map} \label{sec_scatteringmapcomputation}

Recall from Section \ref{section_homoclinicylindersandpseudoorbits} that $D$ is a fundamental domain for the billiard map $f$ on the normally hyperbolic invariant cylinder $A_*$. Moreover, there is an $f$-invariant homoclinic strip $B_*$ relative to $A_*$ with respect to which the scattering map $s_{B_*} : A_* \to A_*'$ is a $C^r$ multibranched map into an open normally hyperbolic invariant cylinder $A_*'$ containing $A_*$. On $D$ we defined the modified scattering map $\tilde{s} : D \to D'$ relative to a fundamental domain $\Delta$ of $f$ in $B_*$. Recall that there are coordinates $(\varphi, y) \in [0, \tau_*) \times \left[ a,b \right]$ on $D$ for some $b>a>0$ defined by Lemma \ref{lemma_coordsonfundamentaldomain}.

Recall moreover that $s_{B_*} = \pi^s \circ (\pi^u)^{-1}$ where the holonomy maps $\pi^s, \pi^u : B_* \to A_*$ are $C^r$ smooth. Let 
\begin{equation} \label{eq_modifiedhomocliniccylinder}
\tilde{B}_* =  f^{n \circ \pi^s(\cdot)} (\Delta) \subset B_*
\end{equation}
where $n: s_{B_*} \left(D \right) \to \mathbb{Z}$ was defined in Section \ref{sec_analysisofhypandhomcylinders}. Then $\tilde{B}_*$ can be considered as a homoclinic cylinder relative to $D$ with respect to the modified scattering map $\tilde{s}$ (despite $\tilde{B}_*$ not being a topological cylinder). Indeed, denoting by
\begin{equation}
\pi^{s,u}_{\Delta} = \left. \pi^{s,u}_{B_*} \right|_{\Delta}, \quad s_{\Delta} = \pi^s_{\Delta} \circ \left( \pi^u_{\Delta} \right)^{-1},
\end{equation}
we see that
\begin{align}
\tilde{s} = f^{n \circ s_{\Delta}} \circ s_{\Delta} = f^{n \circ s_{\Delta}} \circ \pi^s_{\Delta} \circ \left( \pi^u_{\Delta} \right)^{-1} = \pi^s_{\tilde{B}_*} \circ \left( \pi^u_{\tilde{B}_*} \right)^{-1} \circ f^{n \circ s_{\Delta}} = s_{\tilde{B}_*} \circ f^{n \circ s_{\Delta}}.
\end{align}

Suppose now we make a perturbation $Q \to Q + \epsilon \psi$ to the hypersurface where $\psi : \mathbb{R}^d \to \mathbb{R}$ is a smooth function supported near the projection $\pi \circ f (\tilde{B}_*) \subset \Gamma$ of the homoclinic cylinder $f (\tilde{B}_*)$ to the hypersurface $\Gamma$. In this section we determine the effect of the $C^{\infty}$ perturbation $\psi$ on the modified scattering map. This is a key computation for our proof: in Section \ref{sec_diffusiveorbits} we give conditions on the terms of order $\epsilon$ in the expansion of the perturbed modified scattering map which, if satisfied, imply diffusion. Using the results of this section, we show that these conditions are satisfied by certain families of smooth perturbations. We then approximate the smooth family $Q+ \epsilon \psi$ by a real-analytic family $Q_{\epsilon}$ where $Q_0 \equiv Q$.

Denote by $\tilde{s}_{\epsilon}$ the perturbed modified scattering map. Since $\tilde{s}_{\epsilon}$ has smooth dependence on $\epsilon$, we can write
\begin{equation}
\tilde{s}_{\epsilon} (\varphi, y) = \left( \Psi (\varphi, y, \epsilon), Y (\varphi, y, \epsilon) \right).
\end{equation}
Finally, recall that the coordinates $(x,w,z)$ are defined by \eqref{eq_rescaledvariables}.

\begin{proposition}
Let $(\varphi, y) \in D$. Then
\begin{equation} \label{eq_scatteringpertformula1}
\frac{\partial \Psi}{\partial \epsilon}  (\varphi, y, 0) = -  \tau_* \psi (\bar{x}) Ky^{-2} \kappa (\bar{x}, \bar{w})^{- \frac{1}{3}} + O(\tau_*^2)
\end{equation}
and
\begin{equation} \label{eq_scatteringpertformula2}
\frac{\partial Y}{\partial \epsilon} (\varphi, y, 0) = -  \tau_*  K^2 y^{-2} \kappa (\bar{x}, \bar{w})^{- \frac{5}{3}} \langle \nabla \psi (\bar{x}), \bar{w} \rangle + O(\tau_*^2)
\end{equation}
where the higher order terms are uniformly bounded in the $C^0$ topology, and where $(\bar{x}, \bar{w}, \bar{z})$ are the $(x,w,z)$ coordinates of the point
\begin{equation} \label{eq_xwzcoordsdefintermsofphiy}
f \circ \left( \pi^s_{\tilde{B}_*} \right)^{-1} \circ \tilde{s} (\varphi, y) = f \circ \left(\pi^u_{\tilde{B}_*} \right)^{-1} \circ f^{n \circ s_{\Delta} (\varphi, y)} (\varphi, y) \in f \left(\tilde{B}_* \right)
\end{equation}
with the modified scattering map $\tilde{s}$ defined by \eqref{eq_defofmodifiedscatteringmap}. 
\end{proposition}
\begin{proof}
First, we derive a formula for the Hamiltonian of the perturbation of the modified scattering map, by using the Hamiltonian of the perturbation of the unmodified scattering map. By Lemma \ref{lemma_coordsonfundamentaldomain}, the restriction of the symplectic form to $A_*$ is given by $\omega|_{A_*} = d \varphi \wedge dy$. Drop the $\tilde{B}_*$ subscript and denote by $s$ the unperturbed scattering map relative to $\tilde{B}_*$ (i.e. not the modified scattering map). Since we are interested in the behaviour of $s$ only near a fundamental domain of the near-identity exact symplectic map $f$ in $A_*$, we can use the results of \cite{delshams2008geometric} to infer that $s$ is locally exact symplectic. Denote by $s_{\epsilon}$ the perturbed scattering map. Then $s_{\epsilon}$ is a deformation of $s$ in the class of exact symplectic diffeomorphisms defined on a subset of the cylinder $A_*$. Therefore we can write $s_{\epsilon} = \tilde{\phi}^{\epsilon} \circ s$ where $\tilde{\phi}^t$ is the time-$t$ map of some Hamiltonian vector field $\Omega\nabla \tilde{H}_{\mathrm{pert}}$ where $\tilde{H}_{\mathrm{pert}}$ is a Hamiltonian function. Therefore the perturbed scattering map can be written
\begin{equation}
s_{\epsilon} = s + \epsilon  \left( \Omega\nabla \tilde{H}_{\mathrm{pert}} \right) \circ s + O(\epsilon^2).
\end{equation}
Recall that $H_{\mathrm{pert}}$, the Hamiltonian of the perturbation of $f$, is defined by \eqref{eq_perturbationhamiltonian}. Notice that the assumptions of Theorem 31 of \cite{delshams2008geometric} are satisfied. Therefore we have
\begin{align} 
\tilde{H}_{\mathrm{pert}} ={}& \lim_{N_{\pm} \to + \infty} \left[ \sum_{j=0}^{N_--1} \left( H_{\mathrm{pert}} \circ f^{-j} \circ (\pi^u)^{-1} \circ s^{-1} - H_{\mathrm{pert}} \circ f^{-j} \circ s^{-1} \right) + \right. \nonumber \\
& \left. + \sum_{j=1}^{N_+} \left(H_{\mathrm{pert}} \circ f^j \circ (\pi^s)^{-1} - H_{\mathrm{pert}} \circ f^j \right) \right] \nonumber \\ 
\begin{split} \label{eq_hampertmodscatformula} 
={}& H_{\mathrm{pert}} \circ f \circ (\pi^s)^{-1} 
\end{split} \\
\end{align}
where all but one term has vanished due to $\psi$ being supported only near $\pi \circ f (\tilde{B}_*)$. Indeed, since $\tilde{B}_*$ is a fundamental domain of $f$ on the homoclinic strip $B_*$, we find that the sets $f^m \left( \tilde{B}_* \right)$ are mutually disjoint (see also the discussion at the beginning of Section \ref{sec_diffusiveorbits} regarding the support of perturbations).

Now, by Lemma \ref{lemma_modscatisexactsymplectic}, the \emph{modified} scattering map is exact symplectic. Moreover, the perturbed modified scattering map $\tilde{s}_{\epsilon}$ satisfies
\begin{equation}
\left. \frac{d}{d \epsilon} \right|_{\epsilon =0} \tilde{s}_{\epsilon} = \left( \left. \frac{d}{d \epsilon} \right|_{\epsilon =0} s_{\epsilon} \right) \circ f^{n \circ s_{\Delta}} = \left( \Omega \nabla \tilde{H}_{\mathrm{pert}} \right) \circ \tilde{s}.
\end{equation}
Therefore the term of order $\epsilon$ is a Hamiltonian vector field, and the Hamiltonian of the perturbation of $\tilde{s}_{\epsilon}$ is the same as that of $s_{\epsilon}$ in $(\varphi, y)$ coordinates.

We must compute the derivatives of $\tilde{H}_{\mathrm{pert}}$ with respect to $(\varphi, y)$. Recall that the coordinates $(\varphi, y)$ defined by \eqref{eq_fundamentaldomaindef} are $O(\tau_*)$-close in $C^1$ to the time-energy coordinates $(t, y_Z)$ of the vector field $Z$ (defined by \eqref{eq_vfzdefinition}), where $t$ is the time along orbits, and $y_Z$ is the integral defined as in \eqref{eq_rescaledvfzspeed}. By Lemma \ref{lemma_coordsonfundamentaldomain}, the Jacobian matrix $J$ of the coordinate transformation $(t, y_Z) \mapsto (\varphi, y)$ satisfies
\begin{equation} \label{eq_jacobianoftimeenergycoordchange}
J = \left( \frac{\partial (t, y_Z)}{\partial (\varphi, y)} \right) = \left(
\begin{matrix}
1 & 0 \\
0 & 1
\end{matrix}
\right) + O(\tau_*)
\end{equation}
where the higher order terms are uniformly bounded in the $C^0$ topology. It follows that
\begin{equation} \label{eq_timeenergyderivativesofwclosetoz}
\left(
\begin{matrix}
\frac{\partial \tilde{H}_{\mathrm{pert}}}{\partial \varphi} &
\frac{\partial \tilde{H}_{\mathrm{pert}}}{\partial y}
\end{matrix}
\right) = \left(
\begin{matrix}
\frac{\partial \tilde{H}_{\mathrm{pert}}}{\partial t} &
\frac{\partial \tilde{H}_{\mathrm{pert}}}{\partial y_Z}
\end{matrix}
\right) J.
\end{equation}
The idea of the forthcoming computation is to compute the derivatives of $\tilde{H}_{\mathrm{pert}}$ with respect to $t, y_Z$, and then use \eqref{eq_jacobianoftimeenergycoordchange}, \eqref{eq_timeenergyderivativesofwclosetoz} to approximate the $\varphi, y$-derivatives of $\tilde{H}_{\mathrm{pert}}$.

Now, we can write the Hamiltonian $H_{\mathrm{pert}}$ of the perturbation of the billiard map in terms of $(\bar{x}, \bar{w}, \bar{z})$ coordinates, and therefore in terms of $(\bar{x}, \bar{w}, \bar{y}_Z)$ coordinates. Indeed, \eqref{eq_rescaledvariables} and \eqref{eq_rescaledvfzspeed} imply that
\begin{equation} \label{eq_sqrtoneminususquaredintermsoftaustar}
\sqrt{1 - \bar{u}^2} = \frac{1}{2} \tau_* K \bar{y}_Z^{-1} \kappa (\bar{x}, \bar{w})^{- \frac{1}{3}},
\end{equation}
which gives
\begin{equation} \label{eq_ubarnormforphiandy}
\| \bar{u} \| = \sqrt{1 - \frac{1}{4} \tau_*^2 K^2 \bar{y}_Z^{-2} \kappa (\bar{x}, \bar{w})^{- \frac{2}{3}} }.
\end{equation}
Equations \eqref{eq_sqrtoneminususquaredintermsoftaustar} and \eqref{eq_ubarnormforphiandy}, and the fact that $\bar{u} = \| \bar{u} \| \bar{w}$ combined with equation \eqref{eq_perturbationhamiltonian} give the formula for $H_{\mathrm{pert}}$ in terms of $(\bar{x}, \bar{w}, \bar{y}_Z)$. Then the derivatives of $H_{\mathrm{pert}}$, on $M$, with respect to $(\bar{x}, \bar{w}, \bar{y}_Z)$ are
\begin{equation} \label{eq_hpertxwzderivatives1}
\begin{split}
\frac{\partial H_{\mathrm{pert}}}{\partial \bar{x}} ={}& 2 \frac{ \nabla \psi (\bar{x})}{\| \nabla Q (\bar{x}) \|} \sqrt{1 - \bar{u}^2} + 2 \frac{ \psi (\bar{x})}{\| \nabla Q (\bar{x}) \|} \sqrt{ 1 - \bar{u}^2} \, C (\bar{x}) n (\bar{x}) - \\
& \qquad - \bigg[ 2 \Theta (\bar{x}, \bar{u}) \left\langle C (\bar{x}) v, \bar{v} \right\rangle + \\
& \qquad \qquad + 2 \| \nabla Q (\bar{x}) \|^{-1} \sqrt{1 - \bar{u}^2} \left\langle \nabla \psi (\bar{x}), n (\bar{x}) \right\rangle \bigg] n (\bar{x}),
\end{split}
\end{equation}
\begin{equation} \label{eq_hpertxwzderivatives2}
\frac{\partial H_{\mathrm{pert}}}{\partial \bar{w}} =   - \frac{2}{3} \frac{\psi (\bar{x})}{\| \nabla Q (\bar{x}) \|} \tau_* K \bar{y}_Z^{-1} \kappa (\bar{x}, \bar{w})^{- \frac{4}{3}} C(\bar{x}) \bar{w},
\end{equation}
and
\begin{equation} \label{eq_hpertxwzderivatives3}
\frac{\partial H_{\mathrm{pert}}}{\partial \bar{y}_Z} = - \tau_* \frac{\psi (\bar{x})}{\| \nabla Q (\bar{x}) \|} K \bar{y}_Z^{-2} \kappa (\bar{x}, \bar{w})^{- \frac{1}{3}}.
\end{equation}
Even though $w$ belongs to the unit sphere in the tangent space to $\Gamma$ at $x$, we take the derivative with respect to $w$ in the usual sense. Recall in Section \ref{sec_nearboundary} we assumed that $\| \nabla Q (x) \| = 1$ for all $x \in \Gamma$ which gave us the simple formula \eqref{eq_scaledshapeoperator} for the shape operator. Now that we have made perturbations, we may again make this assumption without loss of generality. It follows from this assumption (see \eqref{eq_scaledshapeoperator}) that
\begin{equation}\label{eq_scaledshapeopconsequence}
\left\langle S (\bar{x}) n (\bar{x}), \bar{w} \right\rangle = \left\langle C(\bar{x}) n (\bar{x}), \bar{w} \right\rangle = \left\langle  n (\bar{x}), C(\bar{x}) \bar{w} \right\rangle = \left\langle  n (\bar{x}), S(\bar{x}) \bar{w} \right\rangle = 0.
\end{equation}

Now, let $(\varphi, y) \in D$, and denote by $(\bar{x}, \bar{w}, \bar{y}_Z)$ the point $f \circ (\pi^s)^{-1} \circ f^{n(\varphi,y)} (\varphi, y)$ in $(x, w, y_Z)$ coordinates. With $t$ denoting the time along orbits of the vector field $Z$, differentiating \eqref{eq_hampertmodscatformula} and using equations \eqref{eq_vfzdefinition}, \eqref{eq_rescaledvfzspeed}, \eqref{eq_xwzcoordsdefintermsofphiy}, \eqref{eq_sqrtoneminususquaredintermsoftaustar}, \eqref{eq_hpertxwzderivatives1}, \eqref{eq_hpertxwzderivatives2}, and \eqref{eq_scaledshapeopconsequence} together with the facts that $\bar{y}_Z$ is an integral of $Z$ and $\left\| \nabla Q  \right\| |_{\Gamma} \equiv 1$ give
\begin{align}
\frac{\partial \tilde{H}_{\mathrm{pert}}}{\partial t}  ={}& \left\langle \frac{\partial H_{\mathrm{pert}}}{\partial \bar{x}} (\bar{x}, \bar{w}, \bar{y}_Z), \frac{d \bar{x}}{d t} \right\rangle + \left\langle \frac{\partial H_{\mathrm{pert}}}{\partial \bar{w}} (\bar{x}, \bar{w}, \bar{y}_Z), \frac{d \bar{w}}{d t} \right\rangle \\
={}& \bar{z} \left\langle \frac{\partial H_{\mathrm{pert}}}{\partial \bar{x}} (\bar{x}, \bar{w}, \bar{y}_Z), \bar{w} \right\rangle + \bar{z} \kappa (\bar{x}, \bar{w} ) \left\langle \frac{\partial H_{\mathrm{pert}}}{\partial \bar{w}} (\bar{x}, \bar{w}, \bar{y}_Z), n(\bar{x}) \right\rangle \\
={}& 2 \bar{z} \sqrt{1 - \bar{u}^2} \, \left\langle \nabla \psi (\bar{x}), \bar{w} \right\rangle \\
={}& \tau_* K^2 \bar{y}_Z^{-2} \kappa (\bar{x}, \bar{w})^{- \frac{5}{3}} \left\langle \nabla \psi (\bar{x}), \bar{w} \right\rangle.  \label{eq_ztimederivativeofhtildepert} 
\end{align}
Equations \eqref{eq_hpertxwzderivatives3} and \eqref{eq_ztimederivativeofhtildepert} combined with \eqref{eq_jacobianoftimeenergycoordchange} and \eqref{eq_timeenergyderivativesofwclosetoz}, and the fact that the term of order $\epsilon$ in the expansion of $\tilde{s}_{\epsilon}$ is $ \left( \Omega \nabla \tilde{H}_{\mathrm{pert}} \right) \circ \tilde{s}$ complete the proof of the proposition.
\end{proof}

\section{Diffusive Orbits} \label{sec_diffusiveorbits}

Recall that the existence of the homoclinic strip $B_*$ to the normally hyperbolic invariant cylinder $A_*$ of the billiard map $f$ hinges on the existence of a transverse homoclinic geodesic $\xi$ to the hyperbolic closed geodesic $\gamma$. As mentioned in Section \ref{sec_shadowing}, the existence of a single transverse homoclinic geodesic implies the existence of infinitely many. Choose 8 distinct transverse homoclinic geodesics $\xi_1, \ldots, \xi_8$. These homoclinic geodesics give 8 homoclinic strips $B_{*,j}$ for the billiard map that are independent (in the sense that their orbits under $f$ have no intersection), each of which defines a $C^r$ multibranched scattering map $s_j$ on $A_*$. We then defined the modified scattering maps $\tilde{s}_j : D \to D'$ for some fundamental domain $D$ of $f$ in $A_*$, where $D' \subset A_*'$ is an extension of the fundamental domain $D$ to the open normally hyperbolic cylinder $A_*'$. It was shown that some fundamental domain $\tilde{B}_{*,j} \subset B_{*,j}$ defined by \eqref{eq_modifiedhomocliniccylinder} can be considered a homoclinic cylinder for the modified scattering map $\tilde{s}_j$ relative to $D$. Moreover, we defined the inner map $\Phi : D \to D$ by \eqref{eq_defofinnermaponfd}. 

By shrinking $\tau_*$ and shifting the fundamental domain $\tilde{B}_{*,j}$ along the homoclinic intersection if necessary, we may assume that each projected set $\pi \circ f \left(\tilde{B}_{*,j} \right) \subset \Gamma$ contains no points belonging to a set of the form $\pi \circ f^m \left(\tilde{B}_{*,k} \right)$ where $k \neq j$ and $m \in \mathbb{Z}$. Therefore we can find some small neighbourhood $V_j$ of $\pi \circ f \left(\tilde{B}_{*,j} \right)$ for each $j = 1, \ldots, 8$ such that:
\begin{itemize}
\item
$V_j \cap \pi(A_*) = \emptyset$;
\item
$V_j \cap V_k = \emptyset$ if $j \neq k$; and
\item
$V_j \cap \left( \bigcup_{m \in \mathbb{Z}} \pi \circ f^m \left(\tilde{B}_{*,j} \right) \right) = \pi \circ f \left( \tilde{B}_{*,j} \right)$.
\end{itemize}
In what follows we will make perturbations inside each $V_j$ to affect only the scattering map corresponding to $\tilde{B}_{*,j}$, without interfering with the other scattering maps.

Recall that we have a sequence $\{ A_n \}_{n \in \mathbb{N}}$ of such cylinders so that the set
\begin{equation}
A = \bigcup_{n \in \mathbb{N}} A_n
\end{equation}
is a noncompact normally hyperbolic invariant manifold for $f$. Denote by $D_n$ the fundamental domain in $A_n$ for each $n \in \mathbb{N}$. Combining the following theorem with Theorem \ref{theorem_gtdiffusion} and Lemma \ref{lemma_gtshadowing} completes the proof of Theorems \ref{theorem_main} and \ref{theorem_main2}.

\begin{theorem} \label{theorem_destructionofinvariantcurves}
The maps $\Phi, \tilde{s}_1, \ldots, \tilde{s}_8$ having no common invariant essential curves on any of the glued fundamental domains $D_n$ is a residual property of $Q \in \mathcal{V}$.
\end{theorem}
\begin{proof}
Suppose $\Phi, \tilde{s}_1, \ldots, \tilde{s}_8$ have a common invariant essential curve $C \subset D$, where $D$ is the glued fundamental domain in $A_*$. We use on $D$ the coordinates $(\varphi, y)$ from Lemma \ref{lemma_coordsonfundamentaldomain}, where for convenience we have dropped the `hat' notation. Define
\begin{equation} \label{eq_shorttorus}
\mathbb{T}_* = \mathbb{R} / (\tau_* \mathbb{Z}).
\end{equation}
By Lemma \ref{lemma_innermaptwist}, the inner map $\Phi: D \to D$ is a twist map. Therefore we may apply Birkhoff's theorem: any invariant essential curve of $\Phi$ on $D$ is the graph of a Lipschitz function. It follows that there is a function $y' : \mathbb{T}_* \to \mathbb{R}$ and $L' >0$ such that 
\begin{equation}
C = \textrm{graph} (y')= \{ (\varphi, y'(\varphi)): \varphi \in \mathbb{T}_* \}
\end{equation}
and
\begin{equation}
| y' (\varphi_1) - y' (\varphi_2)| \leq L' | \varphi_1 - \varphi_2|.
\end{equation}
Furthermore, we drop the tilde notation, and denote by $s_j$ the modified scattering maps on $D$.

Fix some $L>0$ that bounds from above the Lipschitz constants of all essential invariant curves of the modified scattering maps in $D$, and all maps on $D$ in some $C^1$ neighbourhood of the modified scattering maps. Notice that the property of the 8 modified scattering maps having no $L$-Lipschitz common invariant curves in $D$ is an open property of the function $Q$ that defines the hypersurface. We show that by arbitrarily small perturbations of the hypersurface we can destroy all $L$-Lipschitz common invariant curves of the modified scattering maps in $D$ whenever $\tau_*>0$ is small enough.

Consider the space of all $L$-Lipschitz functions $y^{\prime} : \mathbb{T}_* \to \mathbb{R}$ equipped with the $C^0$-topology. Let $\mathcal{L}$ denote the subset consisting of all such functions whose graph is contained in $D$ and is invariant under each $s_j$. Clearly $\mathcal{L}$ is compact. If $\mathcal{L} = \emptyset$ then there is nothing to prove, so we assume that $\mathcal{L} \neq \emptyset$. Then for any $\mu > 0$ there are finitely many $y_1, \ldots, y_q \in \mathcal{L}$ such that for any $y^{\prime} \in \mathcal{L}$ we have
\begin{equation} \label{eq_finitecoverofl}
\mathrm{graph}(y^{\prime}) \subset D_p \coloneqq \{ (\varphi, y) \in D : |y - y_p (\varphi) | \leq \mu \}
\end{equation}
for some $p \in \{1, \ldots , q \}$. Inclusion \eqref{eq_finitecoverofl} is also true for Lipschitz invariant curves of all maps $\hat{s}_1, \ldots, \hat{s}_8$ in some $C^1$ neighbourhood of $s_1, \ldots, s_8$.

Let $R$ be some positive constant such that 
\begin{equation} \label{eq_scatteringderivativesbound}
\left\| \frac{\partial s_j}{\partial (\varphi, y)} \right\| < R
\end{equation}
for each $j=1, \ldots, 8$ and for all maps $\hat{s}_1, \ldots, \hat{s}_8$ in some $C^1$ neighbourhood of $s_1, \ldots, s_8$.

Since the modified scattering maps are defined on a cylinder, perturbations must be periodic in the angular component. To this end consider arcs $I_{1,2,3,4} \subsetneq \mathbb{T}_*$ such that $I_1 \cup I_2 = I_3 \cup I_4 = \mathbb{T}_*$. Write $I_{kl} = I_k \setminus I_l$, and assume that $I_{12}, I_{34}, I_{21}, I_{43}$ are mutually disjoint and lie on $\mathbb{T}_*$ in this order (see Figure \ref{figure_intervals}). Since each $I_k$ is a proper subset of $\mathbb{T}_* = \mathbb{R} / (\tau_* \mathbb{Z})$ we have $l(I_k) < \tau_*$.

\begin{figure}[!ht] 
\includegraphics[width = \textwidth,keepaspectratio]{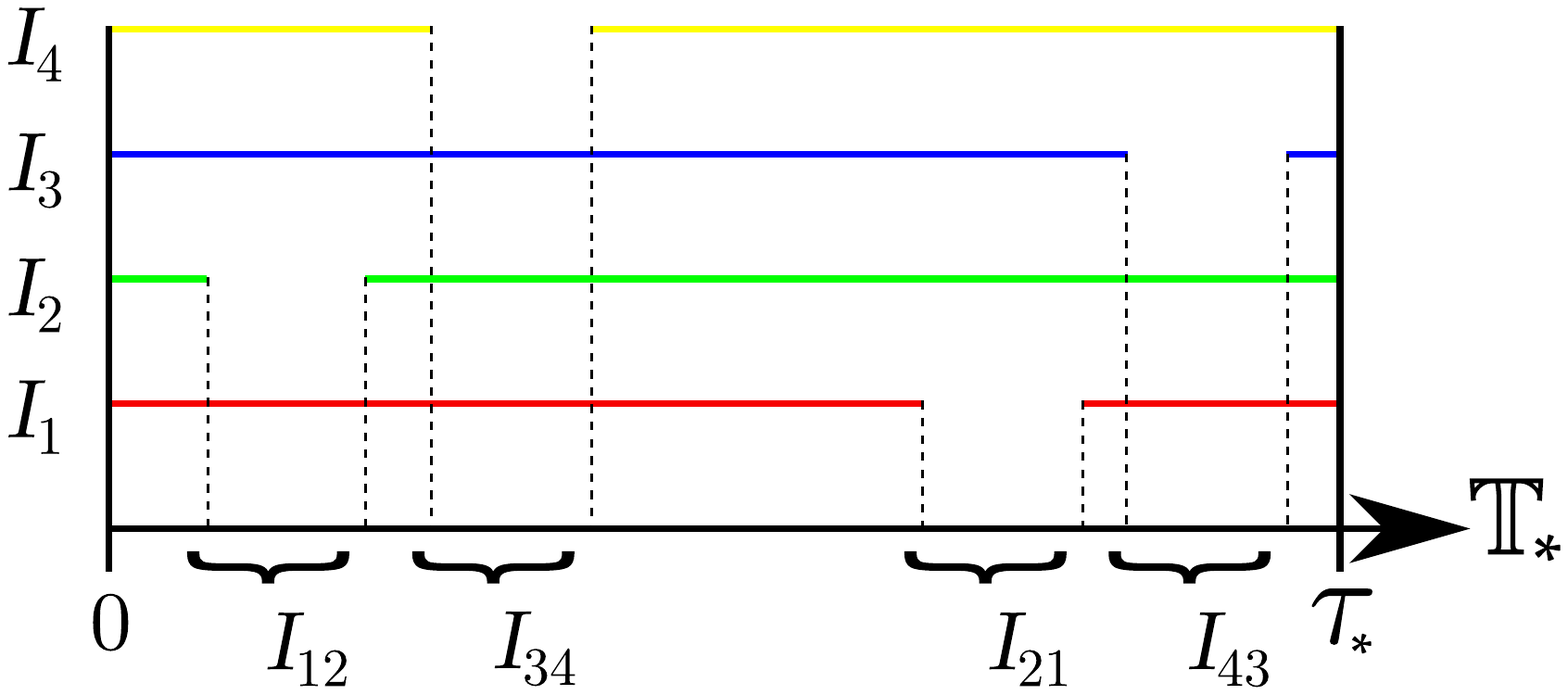}
\caption{We choose the arcs $I_{1,2,3,4} \subsetneq \mathbb{T}_*$ so that $\mathbb{T}_* = I_1 \cup I_2=I_3 \cup I_4$, and so that if $I_{kl}=I_k \setminus I_l$ then $I_{12},I_{34},I_{21},I_{43}$ are mutually disjoint and lie on $\mathbb{T}_*$ in this order. In what follows, we make 8 perturbations (of size, say, $\epsilon$): one perturbation per scattering map. Each interval $I_k$ corresponds to 2 scattering maps. We give conditions on the term of order $\epsilon$ in the expansion of the scattering maps. To avoid monodromy, these conditions apply only to $\varphi$ in the corresponding interval $I_k$. The rest of $\mathbb{T}_*$ is dedicated to guaranteeing a well-defined perturbation.}
\label{figure_intervals}
\end{figure}

For $y_C \in \mathcal{L}$, let $C = \mathrm{graph} (y_C)$. Then $C$ is an $L$-Lipschitz $s_j$-invariant curve for each $j$. Each arc $I_k$ defines a subset of the curve $C$: $\hat{I}_k = \{ (\varphi, y_C(\varphi)) : \varphi \in I_k \}$. Since $C$ is $s_j$-invariant, there is an arc $\bar{I}^j_k \subsetneq \mathbb{T}_*$ such that $s_j (\hat{I}_k) = \{ (\varphi, y_C(\varphi)) : \varphi \in \bar{I}^j_k \}$. Clearly $l(\bar{I}^j_k) < \tau_*$. Let
\begin{equation} 
E = \max_{y_C \in \mathcal{L}}\,  \max_{j,k} \, l(\bar{I}^j_k).
\end{equation}
Then $E<\tau_*$, so
\begin{equation} \label{eq_nbhdsize}
\mu = \frac{\tau_* - E}{R} > 0.
\end{equation}
By the above argument, there are finitely many $y_1, \dots, y_q \in \mathcal{L}$ such that every $L$-Lipschitz common invariant curve of $s_1, \ldots, s_8$ in $D$ lies in one of the cylinders $D_p$ as defined in \eqref{eq_finitecoverofl} corresponding to this value of $\mu$.

Consider one of these cylinders $D_p$ around $y_p \in \mathcal{L}$, and let $C_p = \mathrm{graph} (y_p)$. By \eqref{eq_scatteringderivativesbound}, $s_j (D_p \cap \{\varphi \in I_k \})$ is contained in the $(R \mu)$-neighbourhood of $s_j (C_p \cap \{\varphi \in I_k \})$. This curve is a subset of $C_p$, since $C_p$ is $s_j$-invariant, and corresponds to an interval $\bar{I}^j_k$ where $l(\bar{I}^j_k) \leq E$. Therefore \eqref{eq_nbhdsize} implies that there is an arc $\hat{I}^{(p)}_{jk} \subsetneq \mathbb{T}_*$ such that
\begin{equation}
s_j (D_p \cap \{ \varphi \in I_k \}) \subset \{ (\varphi, y) \in D : |y - y_p (\varphi)| < R \mu, \, \varphi \in \hat{I}^{(p)}_{jk} \}.
\end{equation}
The fact that $\hat{I}^{(p)}_{jk}$ is a proper sub-arc of $\mathbb{T}_*$ will allow us to define a perturbation however we like on $\hat{I}^{(p)}_{jk}$ and extend it to the rest of $\mathbb{T}_*$ without creating monodromy.

We will now construct a two-parameter family of functions $Q_{\epsilon}$ where $\epsilon = (\epsilon_1, \epsilon_2)$ and $Q_0 = Q$. The modified scattering maps depend smoothly on $\epsilon$. Therefore, there are smooth functions $\Psi_j$, $Y_j$ for each $j = 1, \dots, 8$ such that $s_j (\varphi, y) = (\Psi_j (\varphi, y, \epsilon), Y_j (\varphi, y, \epsilon))$.

Consider the perturbed function
\begin{equation} \label{eq_eightperturbations}
Q_{\epsilon}(x) = Q(x) + \epsilon_1 (\psi_1 (x) + \cdots + \psi_4 (x)) + \epsilon_2 (\psi_5 (x) + \cdots + \psi_8 (x))
\end{equation}
where for now we assume that $\mathrm{supp}( \psi_j) \cap \Gamma \subset V_j$. Later, we will approximate $Q_{\epsilon}$ by a real-analytic family. Let $\Gamma_{\epsilon} = \{x \in \mathbb{R}^d : Q_{\epsilon}(x) = 0 \}$. It follows that 
\begin{equation} \label{eq_perturbationindependence}
 \frac{\partial \Psi_{1,2,3,4}}{\partial \epsilon_2} = \frac{\partial Y_{1,2,3,4}}{\partial \epsilon_2} = \frac{\partial \Psi_{5,6,7,8}}{\partial \epsilon_1} = \frac{\partial Y_{5,6,7,8}}{\partial \epsilon_1} = 0.
\end{equation}
The following lemma provides conditions on the $\epsilon$-derivatives of the modified scattering maps which, if satisfied, ensure the absence of $L$-Lipschitz common invariant curves of $s_1, \ldots, s_8$ in $D_p$. The lemma plays the same role as Lemma 5 in \cite{gelfreich2017arnold}. However we have had to improve the estimates, as the assumptions of that lemma do not directly apply in our case.
\begin{lemma} \label{lemma_perturbationconditions}
Suppose the scattering maps $s_1, \ldots, s_8$ corresponding to the family $\Gamma_{\epsilon}$ satisfy:
\begin{enumerate}[(i)]
\item
For $j = 1,2$, and whenever $\Psi_j (\varphi, y, 0) \in I_j$ we have
\begin{equation} \label{eq_lemmacondition1}
\frac{\partial Y_j}{\partial \epsilon_1} (\varphi, y, 0) > 2L \left| \frac{ \partial \Psi_j}{\partial \epsilon_1} (\varphi, y, 0) \right|, 
\end{equation}
and
\begin{equation} \label{eq_lemmacondition2}
\frac{\partial Y_{j+4}}{\partial \epsilon_2} (\varphi, y, 0) > 2L \left| \frac{ \partial \Psi_{j+4}}{\partial \epsilon_2} (\varphi, y, 0) \right|.
\end{equation}
\item
For $j = 3,4$, and whenever $\Psi_j (\varphi, y, 0) \in I_j$ we have
\begin{equation} \label{eq_lemmacondition3}
\frac{\partial Y_j}{\partial \epsilon_1} (\varphi, y, 0) <- 2L \left| \frac{ \partial \Psi_j}{\partial \epsilon_1} (\varphi, y, 0) \right|, 
\end{equation}
and
\begin{equation} \label{eq_lemmacondition4}
\frac{\partial Y_{j+4}}{\partial \epsilon_2} (\varphi, y, 0) <- 2L \left| \frac{ \partial \Psi_{j+4}}{\partial \epsilon_2} (\varphi, y, 0) \right|.
\end{equation}
\item
Moreover for $j=1, 2, 3, 4$ we have
\begin{equation} \label{eq_lemmacondition5}
\left| \frac{\partial \Psi_j}{\partial \epsilon_1} (\varphi, y, 0) \right| > \tau_* (1 + L^{-1})R + \tau_*
\end{equation}
whenever $\Psi_j (\varphi, y, 0) \in I_j$, and
\begin{equation} \label{eq_lemmacondition6}
\left| \frac{\partial \Psi_{j+4}}{\partial \epsilon_2} (\varphi, y, 0) \right| > \tau_* (1 + L^{-1})R  + \tau_*
\end{equation}
whenever $\Psi_{j+4} (\varphi, y, 0) \in I_j$.
\end{enumerate}
Then the set of parameters $\epsilon = (\epsilon_1, \epsilon_2)$ for which the modified scattering maps $s_1, \ldots, s_8$ have an $L$-Lipschitz common invariant curve in $D_p$ has Lebesgue measure 0, and so there are arbitrarily small values of $\epsilon$ for which the modified scattering maps have no $L$-Lipschitz common invariant curve in $D_p$.
\end{lemma}
\begin{proof}
On the space of parameters we use the norm $\| \epsilon \| = \max \{ |\epsilon_1|, |\epsilon_2| \}$. Take two values $\epsilon^*$ and $\epsilon^{**}$ of $\epsilon$, and suppose the maps $s_1, \ldots, s_8$ have an $L$-Lipschitz common invariant curve $\mathcal{L}^* \subset D_p$ at $\epsilon^*$, and $\mathcal{L}^{**} \subset D_p$ at $\epsilon^{**}$. These curves can be written as $\mathcal{L}^* = \mathrm{graph} (y^*)$ and $\mathcal{L}^{**} = \mathrm{graph} (y^{**})$ where $y^*, \, y^{**} : \mathbb{T}_* \to \mathbb{R}$. We claim that 
\begin{equation} \label{eq_lemmaclaim1}
\| \epsilon^* - \epsilon^{**} \| \leq \tau^{-1}_* |y^*(0) - y^{**}(0)|.
\end{equation}
The lemma follows from \eqref{eq_lemmaclaim1}. Indeed, suppose it is true, and let $\mathcal{E} \subset \mathbb{R}^2$ denote the set of all $\epsilon$ for which the scattering maps have an $L$-Lipschitz common invariant curve in $D_p$. Let $\mathcal{Y}$ denote the set of all points $y_0 \in \mathbb{R}$ for which there is an $L$-Lipschitz common invariant curve $\mathcal{L}^{\prime} = \mathrm{graph} (y^{\prime}) \subset D_p$ of $s_1, \ldots, s_8$ at some value of $\epsilon$ with $y^{\prime}(0) = y_0$. By \eqref{eq_lemmaclaim1}, for each $y_0 \in \mathcal{Y}$ there is exactly one $\epsilon \in \mathcal{E}$ at which the modified scattering maps have an $L$-Lipschitz common invariant curve in $D_p$ that intersects the $y$-axis at $y_0$. It follows that there is a well-defined bijective map from $\mathcal{Y}$ to $\mathcal{E}$ sending $y_0$ to $\epsilon$. Moreover, \eqref{eq_lemmaclaim1} also implies that this map is $(\tau^{-1}_*)$-Lipschitz. Since Lipschitz maps do not increase Hausdorff dimension, it follows that $\mathcal{E}$ is a 1-dimensional set in $\mathbb{R}^2$, and so it has Lebesgue measure 0. 

We must now establish estimate \eqref{eq_lemmaclaim1}. Suppose, for a contradiction, that it is not true. Without loss of generality, we may assume that
\begin{equation} \label{eq_rearrangingassumptions}
y^*(0) \geq y^{**}(0) \quad \mathrm{and} \quad | \epsilon^*_2 - \epsilon^{**}_2 | \leq | \epsilon^*_1 - \epsilon^{**}_1| = \epsilon^*_1 - \epsilon^{**}_1.
\end{equation}
It follows that $\Delta \epsilon \coloneqq \| \epsilon^* - \epsilon^{**} \| = \epsilon^*_1 - \epsilon^{**}_1 >0$. The supposition that \eqref{eq_lemmaclaim1} is not true implies that
\begin{equation} \label{eq_contradictionhypothesis}
0 \leq y^*(0) - y^{**}(0) < \tau_* \Delta \epsilon.
\end{equation}
We have that $\varphi = 0$ is in at least one of $I_3$ or $I_4$ since $\mathbb{T}_* = I_3 \cup I_4$. Assume $0 \in I_3$. We write
\begin{equation}
\begin{cases}
	(\bar{\varphi}^*, \bar{y}^*) = s_3(0, y^*(0), \epsilon^*) = (\Psi_3(0, y^*(0), \epsilon^*), Y_3(0, y^*(0), 			\epsilon^*))\\
	(\bar{\varphi}^{**}, \bar{y}^{**}) = s_3(0, y^{**}(0), \epsilon^{**}) = (\Psi_3(0, y^{**}(0), \epsilon^{**}), 			Y_3(0, y^{**}(0), \epsilon^{**})).
\end{cases}
\end{equation}
Since $\mathcal{L}^*$, $\mathcal{L}^{**}$ are invariant under $s_1, \ldots, s_8$ at $\epsilon^*$, $\epsilon^{**}$ respectively, we have $\bar{y}^* = y^* (\bar{\varphi}^*)$ and $\bar{y}^{**} = y^{**} (\bar{\varphi}^{**})$. It follows from the mean value theorem and \eqref{eq_scatteringderivativesbound}, \eqref{eq_perturbationindependence}, \eqref{eq_lemmacondition3}, \eqref{eq_rearrangingassumptions}, and \eqref{eq_contradictionhypothesis} that there are $\bar{\epsilon}_1, \tilde{\epsilon}_1 \in \left(\epsilon_1^{**}, \epsilon_1^* \right)$ such that
\begin{equation}
| \bar{\varphi}^* - \bar{\varphi}^{**}| < \left| \frac{\partial \Psi_3}{\partial \epsilon_1} \left( 0, y^*(0), (\bar{\epsilon}_1, \epsilon_2^*) \right) \right| \Delta \epsilon + R \tau_* \Delta \epsilon, 
\end{equation}
and 
\begin{equation}
\bar{y}^* - \bar{y}^{**} < -2L \left| \frac{\partial \Psi_3}{\partial \epsilon_1} \left( 0, y^*(0), (\tilde{\epsilon}_1, \epsilon_2^*) \right) \right| \Delta \epsilon + R \tau_* \Delta \epsilon.
\end{equation}
These inequalities together with the mean value theorem again, the $L$-Lipschitz property of $y^{**}$, and \eqref{eq_lemmacondition5} imply that
\begin{align}
y^* (\bar{\varphi}^*) - y^{**}(\bar{\varphi}^*) <{}& \Delta \epsilon \left[ R \tau_* (L+1) - L \left| \frac{\partial \Psi_3}{\partial \epsilon_1} \left( 0, y^*(0), 0 \right)\right| + O(\epsilon^*) \right] \\
<{}& \Delta \epsilon \left[ - \tau_* L + O(\epsilon^*) \right] < 0. 
\end{align}
We have assumed that  $y^*(0) \geq y^{**}(0)$ and we have just shown that there is $\varphi = \bar{\varphi}^* \in \mathbb{T}_*$ such that $y^*(\varphi) < y^{**}(\varphi)$. It follows that
\begin{equation} \label{eq_intersectingcurves}
\mathcal{L}^* \cap \mathcal{L}^{**} \neq \emptyset.
\end{equation}

Let us define what it means for an arc to be positive or negative. Let $I \subset \mathbb{T}_*$ be an arc such that $y^*(\varphi) = y^{**}(\varphi)$ for $\varphi \in \partial I$. We say that $I$ is:
\begin{itemize}
\item
\emph{Positive} if $y^* (\varphi) > y^{**}(\varphi)$ for all $\varphi \in \mathrm{Int}(I)$; and
\item
\emph{Negative} if $y^* (\varphi) < y^{**}(\varphi)$ for all $\varphi \in \mathrm{Int}(I)$.
\end{itemize}
We consider a single point in $\mathcal{L}^* \cap \mathcal{L}^{**}$ to be both a positive and negative arc. It follows that the only arcs in $\mathbb{T}_*$ that are both positive and negative are points in $\mathcal{L}^* \cap \mathcal{L}^{**}$. We will show that \eqref{eq_contradictionhypothesis} implies that we can construct a positive arc and a negative arc in $\mathbb{T}_*$ which agree along some nontrivial arc. This is absurd, and will complete the proof by contradiction.

We have already shown via \eqref{eq_intersectingcurves} that there is at least one positive arc and one negative arc. Let $J \subset \mathbb{T}_*$ be a positive arc, and let
\begin{equation}
\mathcal{L}^*_J = \{ (\varphi, y^*(\varphi)) : \varphi \in J \}, \quad \mathcal{L}^{**}_J = \{ (\varphi, y^{**}(\varphi)) : \varphi \in J \},
\end{equation}
and
\begin{equation}
\mathcal{D}_J = \{ (\varphi, y) : \varphi \in J, \, y^{**}(\varphi) \leq y \leq y^*(\varphi) \}.
\end{equation}
We will show that if $J \subseteq I_j$ for $j \in \{ 1,2 \}$ then there is a positive arc $J^{\prime} \subset \mathbb{T}_*$ such that $s_j (\mathcal{L}^*_J, \epsilon^*) \subset \mathcal{L}^*_{J^{\prime}}$ and
\begin{equation} \label{eq_lemmaclaim2}
l (J^{\prime}) > \delta \Delta \epsilon, \quad \mathrm{Area}(\mathcal{D}_{J^{\prime}}) > \mathrm{Area} (\mathcal{D}_J)
\end{equation}
where 
\begin{equation} \label{eq_deltaconstantdef}
\delta = \tau_* (1+L^{-1})R > 0.
\end{equation}
Write $s_j^*(\cdot, \cdot) = s_j(\cdot, \cdot, \epsilon^*)$, and $s_j^{**}(\cdot, \cdot) = s_j(\cdot, \cdot, \epsilon^{**})$. Let $\varphi \in J$, and let $P = (\varphi, y^*(\varphi)) \in \mathcal{L}^*_J$ denote the corresponding point in $\mathcal{L}^*_J$. Write
\begin{equation}
\begin{cases} \label{eq_cylinderpoints}
	P^* = (\varphi^*, y^*(\varphi^*)) = s_j^*(P) = (\Psi_j (\varphi, y^*(\varphi), \epsilon^*), Y_j(\varphi, 			y^*(\varphi), \epsilon^*)) \\
	P^{\prime} = (\varphi^{\prime}, y^{\prime}) = s_j^{**}(P) = (\Psi_j(\varphi, y^*(\varphi), \epsilon^{**}), 			Y_j (\varphi, y^*(\varphi), \epsilon^{**}).
\end{cases}
\end{equation}
Clearly $P^* \in \mathcal{L}^*$ and $P^{\prime} \in s_j^{**}(\mathcal{L}^*_J)$. Since $J$ is positive we have $y^{**}(\varphi) \leq y^*(\varphi)$, and so $P$ is either on or above $\mathcal{L}^{**}$. Moreover since $\mathcal{L}^{**}$ is $s_j^{**}$-invariant, $P^{\prime}$ must also lie on or above $\mathcal{L}^{**}$, and so
\begin{equation} \label{eq_cylinderpointsposition}
y^{\prime} \geq y^{**}(\varphi^{\prime}).
\end{equation}
The mean value theorem together with \eqref{eq_perturbationindependence}, \eqref{eq_lemmacondition1}, and expressions \eqref{eq_cylinderpoints} implies that there are $\hat{\epsilon}_1, \epsilon_1' \in (\epsilon_1^{**}, \epsilon_1^*)$ such that
\begin{equation}
| \varphi^* - \varphi^{\prime}| \leq \left| \frac{\partial \Psi_j}{\partial \epsilon_1} \left( \varphi, y^*(\varphi), (\hat{\epsilon}_1, \epsilon_2^*) \right) \right| \Delta \epsilon, 
\end{equation}
and
\begin{equation}
y^* (\varphi^*) - y^{\prime} > 2L \left| \frac{\partial \Psi_j}{\partial \epsilon_1} \left( \varphi, y^*(\varphi), (\epsilon_1', \epsilon_2^*) \right) \right| \Delta \epsilon.
\end{equation}
These inequalities combined with \eqref{eq_cylinderpointsposition} and the $L$-Lipschitz property of $y^{**}$ give
\begin{equation} \label{eq_pointsimageposition}
\begin{dcases}
y^*(\varphi^*) - y^{**}(\varphi^*) > L \, \Delta \epsilon \left[ \left| \frac{\partial \Psi_j}{\partial \epsilon_1} \left(\varphi, y^*(\varphi),0 \right) \right| + O (\epsilon^*) \right] >0, \\
y^*(\varphi^{\prime}) - y^{\prime} > L \, \Delta \epsilon \left[ \left| \frac{\partial \Psi_j}{\partial \epsilon_1} \left(\varphi, y^*(\varphi),0 \right) \right| + O (\epsilon^*) \right] >0. 
\end{dcases}
\end{equation}
The first inequality of \eqref{eq_pointsimageposition} shows that there is a positive arc $J^{\prime} \subset \mathbb{T}_*$ such that $s_j^*(\mathcal{L}^*_J) \subset \mathcal{L}^*_{J^{\prime}}$. Let $\tilde{y}^* = y^* - y^{**}$. Then $\tilde{y}^*$ is 0 on the endpoints of $J^{\prime}$ and strictly positive on $\mathrm{Int}(J^{\prime})$. Moreover $\tilde{y}^*$ is $2L$-Lipschitz since $y^*$ and $y^{**}$ are $L$-Lipschitz. Let $\tilde{s}_j : \mathbb{T}_* \to \mathbb{T}_*$ denote the map $\tilde{s}_j (\varphi) = \Psi_j(\varphi, y^*(\varphi), \epsilon^*)$. Then $\tilde{s}_j (J)$ is compactly embedded in $J^{\prime}$. Let $\varphi_1,\varphi_4$ denote the endpoints of $J^{\prime}$, and $\varphi_2, \varphi_3$ the endpoints of $\tilde{s}_j(J)$. It is possible that $\varphi_2 = \varphi_3$. Without loss of generality, we may assume (by translating our $\varphi$-coordinate on $\mathbb{T}_*$ if necessary) that $\varphi_1 < \varphi_2 \leq \varphi_3 < \varphi_4$. Notice that the first inequality of \eqref{eq_pointsimageposition} holds at $\varphi_2, \varphi_3$ since it holds at all points of $\tilde{s}_j (J)$. These facts combined with \eqref{eq_lemmacondition5} and \eqref{eq_deltaconstantdef} imply that
\begin{align}
l(J^{\prime}) ={}& \varphi_4 - \varphi_1 \geq \varphi_4 - \varphi_3 + \varphi_2 - \varphi_1 \\
={}& | \varphi_4 - \varphi_3 | + | \varphi_2 - \varphi_1 | \\
\geq{}& (2L)^{-1} \left[ | \tilde{y}^* (\varphi_4) - \tilde{y}^* (\varphi_3) | + | \tilde{y}^* (\varphi_2) - \tilde{y}^* (\varphi_1) | \right] \\
={}& (2L)^{-1} \left( \tilde{y}^* (\varphi_2) + \tilde{y}^* (\varphi_3) \right) \\
>{}& \Delta \epsilon \left[ \inf_{\varphi \in J} \left| \frac{\partial \Psi_j}{\partial \epsilon_1} \left(\varphi, y^*(\varphi),0 \right) \right| + O(\epsilon^*) \right] \\
>{}& \Delta \epsilon \left[ \delta + \tau_* + O(\epsilon^*) \right] > \delta \Delta \epsilon
\end{align}
which is the first inequality of \eqref{eq_lemmaclaim2}.

Due to \eqref{eq_cylinderpointsposition} and the second inequality of \eqref{eq_pointsimageposition}, the curve $\mathcal{L}^{\prime} = s_j^{**}(\mathcal{L}^*_J)$ lies above $\mathcal{L}^{**}_{J^{\prime}}$, and strictly below $\mathcal{L}^*_{J^{\prime}}$. Since $s^{**}_j (\mathcal{D}_J)$ is the region bounded by $\mathcal{L}^{\prime}$ and $s^{**}_j(\mathcal{L}^{**}_J)$ it follows that
\begin{equation}
\mathrm{Area}(\mathcal{D}_{J^{\prime}}) > \mathrm{Area}(s_j^{**}(\mathcal{D}_J)) = \mathrm{Area}(\mathcal{D}_J)
\end{equation}
where the last equality is due to the symplecticity of $s_j^{**}$. This is the second inequality of \eqref{eq_lemmaclaim2}.

Now, let $J \subseteq I_j$ be a positive arc where $j \in \{1,2\}$, and construct a sequence of positive arcs $J_s$ as follows: $J_0=J$, and $\tilde{s}_{j_s} (J_s) \subset J_{s+1}$ where $j_s = 1$ if $J_s \subseteq I_1$ and $j_s = 2$ if $J_s \subseteq I_2$. The sequence terminates if neither $J_s \not\subset I_1$ nor $J_s \not\subset I_2$. Due to the second inequality of \eqref{eq_lemmaclaim2}, $\mathrm{Area}(\mathcal{D}_{J_s})$ is strictly increasing as $s$ increases, and so $J_{s_1} \neq J_{s_2}$ if $s_1 \neq s_2$. Since all the arcs $J_s$ are positive, it follows that they are mutually disjoint. Therefore, \eqref{eq_shorttorus} and the first inequality of \eqref{eq_lemmaclaim2} guarantee that there cannot be more than $\frac{\tau_*}{\delta \Delta \epsilon}$ such arcs, and so the sequence must terminate at some positive arc $J_+$. Since the positive arc $J_+$ is contained neither in $I_1$, nor in $I_2$, we have $\emptyset \neq J_+ \cap \left(\mathbb{T}_* \setminus I_2 \right) = J_+ \cap \left(I_1 \setminus I_2 \right) = J_+ \cap I_{12}$, and $\emptyset \neq J_+ \cap \left(\mathbb{T}_* \setminus I_1 \right) = J_+ \cap \left(I_2 \setminus I_1 \right)= J_+ \cap I_{21}$ (see Figure \ref{figure_intervals}).

Applying similar logic to negative arcs in the intervals $I_3$ and $I_4$, we construct a negative arc $J_-$ such that $J_- \cap I_{34} \neq \emptyset$ and $J_- \cap I_{43} \neq \emptyset$. But since $I_{12}, I_{34}, I_{21}, I_{43}$ are located in this order on $\mathbb{T}_*$, this means that $\mathrm{Int}(J_+) \cap \mathrm{Int}(J_-) \neq \emptyset$, which is impossible since one is positive and the other is negative. This contradiction implies \eqref{eq_lemmaclaim1}, and completes the proof of the lemma.
\end{proof}

Now, suppose we make a perturbation as in \eqref{eq_eightperturbations} where each $\psi_j$ is supported near $\pi \circ f \left(\tilde{B}_{*,j} \right) \subset V_j$ (where $V_j$ are the neighbourhoods described at the beginning of the section). We will show now that we can choose $\psi_1, \ldots, \psi_8$ so that the modified scattering maps satisfy the assumptions of Lemma \ref{lemma_perturbationconditions}. 

Suppose $j \in \{1,2,3,4\}$, and $(\varphi, y) \in D_p$ such that $\Psi_j (\varphi, y, 0) \in I_j$ where $s_j=(\Psi_j, Y_j)$ are the modified scattering maps. Let 
\begin{equation}
(\bar{x}, \bar{w}, \bar{z}) = f \circ \left( \pi^u_j \right)^{-1} \circ f^{n(\varphi,y)} (\varphi, y) = f \circ \left( \pi^s_j \right)^{-1} \circ s_j (\varphi, y) \in f \left( \tilde{B}_{*,j} \right).
\end{equation}
By \eqref{eq_scatteringpertformula1} and \eqref{eq_scatteringpertformula2} we have
\begin{equation}
\frac{\partial}{\partial \epsilon_1} \Psi_j (\varphi, y, 0) = - \tau_*  \psi_j (\bar{x}) Ky^{-2} \kappa (\bar{x}, \bar{w})^{- \frac{1}{3}} + O(\tau_*^2)
\end{equation}
and
\begin{equation}
\frac{\partial}{\partial \epsilon_1} Y_j (\varphi, y, 0) = - \tau_*  K^2 y^{-2} \kappa (\bar{x}, \bar{w})^{- \frac{5}{3}} \langle \nabla \psi_j (\bar{x}), \bar{w} \rangle  + O(\tau_*^2).
\end{equation}
Assume $\psi_j > 0$. Using these expressions, we see that conditions \eqref{eq_lemmacondition1} and \eqref{eq_lemmacondition3} are of the form:
\begin{equation} \label{eq_newconditions1}
\langle \nabla \psi_j (\bar{x}), \bar{w} \rangle < -2L K^{-1} \kappa (\bar{x}, \bar{w})^{\frac{4}{3}} \psi_j (\bar{x}) + O(\tau_*)
\end{equation}
for $j=1,2$ whenever $\Psi_j (\varphi, y, 0) \in I_j$; and
\begin{equation} \label{eq_newconditions2}
\langle \nabla \psi_j (\bar{x}), \bar{w} \rangle > 2L K^{-1} \kappa (\bar{x}, \bar{w})^{\frac{4}{3}} \psi_j (\bar{x}) + O(\tau_*)
\end{equation}
for $j=3,4$ whenever $\Psi_j (\varphi, y, 0) \in I_j$. Moreover condition \eqref{eq_lemmacondition5} is 
\begin{equation} \label{eq_newconditions3}
\psi_j (\bar{x}) >  K^{-1} y^2 \kappa (\bar{x}, \bar{w})^{\frac{1}{3}} \left[ (1 + L^{-1})R +1 \right] + O(\tau_*).
\end{equation}
Fix some point in $ f \left( \tilde{B}_{*,j} \right)$ and denote by $w^*$ its $w$-component. Notice that for all $(\bar{x}, \bar{w}, \bar{z}) \in f \left( \tilde{B}_{*,j} \right)$ we have 
\begin{equation}
\langle w^*, \bar{w} \rangle = 1 + O(\tau_*)
\end{equation}
where the terms of order $\tau_*$ are uniformly bounded. For constants $C_{j,1}>0$ and $C_{j,2} \in \mathbb{R}$ which are yet to be determined, let 
\begin{equation}
\psi_j (\bar{x}) = C_{j,1} \exp \left( C_{j,2} \sum_{k=1}^d w^*_k \bar{x}_k \right)
\end{equation}
whenever $\bar{x} \in \pi \circ f \left( \tilde{B}_{*,j} \right)$ corresponds via $f \circ (\pi^u_j)^{-1} \circ f^n$ to some $(\varphi, y) \in D$ satisfying $\Psi_j (\varphi, y, 0) \in I_j$. Moreover $\psi_j$ is 0 outside $V_j$, and $\psi_j$ is $C^{\infty}$-smooth. For $j=1,2$ let
\begin{equation}
C_{j,2} =-1 + \min_{(x, w) \in T^1 \Gamma} \left[ - 2 L K^{-1}  \kappa (x, w)^{\frac{4}{3}} \right],
\end{equation}
where $T^1 \Gamma$ is the subset of the tangent bundle consisting of vectors of norm 1. Since $\psi_j > 0$, we may divide both sides of \eqref{eq_newconditions1} by $\psi_j (\bar{x})$ to see that it is equivalent to
\begin{equation}
-2 L K^{-1} \kappa (\bar{x}, \bar{w})^{\frac{4}{3}} > \langle \nabla \ln \psi_j (\bar{x}), \bar{w} \rangle + O(\tau_*) = C_{j,2} + O(\tau_*) \label{eq_isconstantbigenough}
\end{equation}
which is true for sufficiently small $\tau_*$. Since the exponential term is nonzero, we may choose $C_{j,1} >0$ large enough so that \eqref{eq_newconditions3} is satisfied. For $j=3,4$, let
\begin{equation}
C_{j,2} = 1 + \max_{(x, w) \in T^1 \Gamma} \left[ 2 L K^{-1}  \kappa (x, w)^{\frac{4}{3}} \right].
\end{equation}
Then a similar computation shows that $\psi_j$ satisfies \eqref{eq_newconditions2} and \eqref{eq_newconditions3} for sufficiently small $\tau_*$. Similar computations show that similar functions satisfy conditions \eqref{eq_lemmacondition2}, \eqref{eq_lemmacondition4}, and, \eqref{eq_lemmacondition6}. 

Lemma \ref{lemma_perturbationconditions} thus implies that the modified scattering maps $s_1, \ldots, s_8$ corresponding to this perturbed hypersurface $\Gamma_{\epsilon}$ have no $L$-Lipschitz common invariant curves in the cylinder $D_p$ for arbitrarily small values of $\epsilon$. Approximating the perturbed function $Q_{\epsilon}$ sufficiently well by a real-analytic family $\tilde{Q}_{\epsilon} \subset \mathcal{V}$, we can guarantee that for arbitrarily small values of $\epsilon$, the conditions of Lemma \ref{lemma_perturbationconditions} are still satisfied for the analytic family. Now, apply this process to clear the $L$-Lipschitz common invariant curves in $D_1$. Since the property of the scattering maps having no invariant curves in $D_1$ is open, we may repeat the process with sufficiently small values of $\epsilon$ to kill the invariant curves in $D_2$ without creating new invariant curves in $D_1$. Repeating this process $q$ times, we can simultaneously clear the subcylinders $D_1, \ldots, D_q$ of invariant curves. Since all $L$-Lipschitz common invariant curves of $s_1, \ldots, s_8$ in $D$ were contained in these subcylinders, we have therefore cleared the cylinder $D$ of common invariant curves of the modified scattering maps. This in turn implies that the IFS on the glued fundamental domain $D$ itself has no invariant essential curves for some $\tilde{Q}$ arbitrarily close to $Q$ in $\mathcal{V}$.

We have shown that by making arbitrarily small analytic perturbations, we can destroy the invariant essential curves on the fundamental domain $D$ for all sufficiently small $\tau_*$. This implies that the property of the IFS having no invariant essential curves on $D$ is dense in $\mathcal{V}$ for all sufficiently small $\tau_*$. As mentioned earlier, it is also open. Choosing the sequence of cylinders $\{ A_n \}_{n \in \mathbb{N}}$ so that the corresponding sequence $\{ \tau_n \}_{n \in \mathbb{N}}$ of constants are all sufficiently small, we therefore obtain for each $n \in \mathbb{N}$ an open dense set $\tilde{\mathcal{V}}_n \subset \mathcal{V}$ with the property that for all $Q \in \tilde{\mathcal{V}}_n$, the corresponding IFS $\{ \Phi, s_1, \dots, s_8 \}$ has no invariant essential curves on the cylinder $D_n$. It follows that 
\begin{equation}
\tilde{\mathcal{V}} = \bigcap_{n \in \mathbb{N}} \tilde{\mathcal{V}}_n
\end{equation}
is the residual set we are looking for.

\end{proof}

\appendix
\section*{Appendix. Exact Symplecticity of the Modified Scattering Maps}

In this appendix, we prove Lemma \ref{lemma_modscatisexactsymplectic}, which for convenience is restated here.

\begin{lemma*}
If we consider $D$ as a cylinder by identifying points on the line $\{\varphi=0\}$ with their image under $f$, and use the coordinates $(\varphi,y) \in [0, \tau_*) \times \left[ a,b \right]$, then the modified scattering map $\tilde{s} : D \to \tilde{s}(D) \subset D'$ is an exact symplectic $C^r$-diffeomorphism onto its image, and its image contains an essential curve in $D'$.
\end{lemma*}

\begin{proof}
Let $\Lambda = \mathbb{T} \times \left[ E_-, E_+ \right]$, where $\mathbb{T} = \mathbb{R} /  \mathbb{Z} $ and where $E_{\pm} \in \mathbb{R}$ are such that $E_- < E_+$. Recall the mapping torus of the billiard map $f$ on $M$ is obtained by taking the direct product of $M$ with the closed unit interval $[0,1]$ and identifying points $(z,1) \in M \times [0,1]$ with points $(f(z),0) \in M \times [0,1]$. The resulting set is a fibre bundle, with base space $\mathbb{T}$, and fibres $M$. Denote by $\widehat{M}$ the product of the mapping torus of $f$ on $M$ with the interval $\left[ E_-, E_+ \right]$, so $\widehat{M}$ is itself a fibre bundle, with base space $\Lambda$, and fibres $M$. Note that we do not claim that $\widehat{M}$ is symplectomorphic to a direct product of $M$ and $\Lambda$. Notice, however, that it is an exact symplectic manifold. Indeed, recall from Section \ref{subsec_billiardmapdef} that the billiard map is exact symplectic, in the sense that $f^* \lambda - \lambda = - d \tau$, where $\lambda$ is the Liouville 1-form. Let $\zeta (t)$ be a smooth function that is identically 0 near $t=0$ and identically 1 near $t=1$, and define $\hat{\lambda} = \lambda - E \, dt - d \left( \zeta (t) \, \tau \right)$. Then $\hat{\lambda}$ is a Liouville 1-form on $\widehat{M}$, and
\begin{equation}
\hat{\omega} = d \hat{\lambda} =  \omega + dt \wedge dE
\end{equation}
is a symplectic form on $\widehat{M}$, where $\omega$ is the symplectic form on $M$. We can thus define a Hamiltonian function $K : \widehat{M} \to \mathbb{R}$ by $K=E$. This Hamiltonian has the following properties:
\begin{itemize}
\item
The Hamiltonian vector field of $K$ with respect to the symplectic form $\hat{\omega}$ is such that
\begin{equation}
\dot{t} = 1, \quad \dot{E} = 0
\end{equation}
where the coordinates $(t,E)$ on $\Lambda$ represent time and energy respectively. 
\item
If we fix $E_0 \in [E_-,E_+]$ and define $\widehat{M}_0$ to be the fibre at $t=0, \, E=E_0$, then the return map $\hat{f} : \widehat{M}_0 \to \widehat{M}_0$ coincides with the billiard map $f$ on $M$, under the obvious identification of $\widehat{M}_0$ with $M$.  
\end{itemize}
Moreover, the set $\widehat{A}$ -- defined as the fibre bundle with base space $\Lambda$ and fibres $A_*$ -- is a normally hyperbolic invariant manifold in $\widehat{M}$ for the Hamiltonian flow $\phi^t_K$ of $K$, and the set $\widehat{B}$ -- defined as the fibre bundle with base space $\Lambda$ and fibres being the homoclinic manifold $B_*$ of $f$ in $M$ -- is a homoclinic manifold in the transverse homoclinic intersection of $W^{s,u} \left( \widehat{A} \right)$. Denote by $\hat{\pi}^{s,u}: \widehat{B} \to \widehat{A}$ the holonomy maps corresponding to the Hamiltonian flow $\phi^t_K$. 

Let $\Delta = \left( \pi^u_{B_*} \right)^{-1} (D)$ be a preimage of $D$ under the unstable holonomy map. For positive integers $N_{\pm}$ that we will fix later, let 
\begin{equation}
\Delta_+ = f^{N_+} \left( \Delta \right), \quad \Delta_- = f^{- N_-} \left( \Delta \right), 
\end{equation}
and define
\begin{equation}
\pi_+ = \left. \pi^s \right|_{\Delta_+}, \quad \pi_- = \left. \pi^u \right|_{\Delta_-}.
\end{equation}
Then the modified scattering map is
\begin{align}
\tilde{s} ={}& f^{n \circ s_{B_*}} \circ s_{B_*} = f^{n \circ s_{B_*}} \circ \pi^s_{B_*} \circ \left( \pi^u_{B_*} \right)^{-1} \\
={}& f^{n \circ s_{B_*}} \circ f^{- N_+} \circ \pi_+ \circ f^{N_+ + N_-} \circ \pi_-^{-1} \circ f^{-N_-}. \label{eq_modscatmapexpanded1}
\end{align}
In what follows we rephrase \eqref{eq_modscatmapexpanded1} using Poincar\'e maps of the Hamiltonian flow $\phi^t_K$, and we use the fact that these return maps are exact symplectic diffeomorphisms of topological cylinders. 

Denote by $C_0$ the curve $\{ \varphi =0 \} \subset A_*$ in $M$, so that the fundamental domain $D$ is bounded by $C_0$ and $f(C_0)$. It follows that the set $\Delta$ is a fundamental domain of $f$ in $B_*$, bounded by
\begin{equation}
\widetilde{C}_0 = \left( \pi^u_{B_*} \right)^{-1} \left( C_0 \right)
\end{equation}
and $f \left(\widetilde{C}_0 \right)$. Moreover the sets $\Delta_{\pm}$ are fundamental domains of $f$, in the connected component of $W^s(A_*) \cap W^u(A_*)$ containing $B_*$, bounded by
\begin{equation}
C_0^{\pm} = f^{\pm N_{\pm}} \left( \widetilde{C}_0 \right)
\end{equation}
and $f \left( C_0^{\pm} \right)$ respectively.

Fix some $E_0 \in [E_-, E_+]$, and define the sets:
\begin{itemize}
\item
$\widehat{D} \subset \widehat{A}$ is the set with base points $\{ (t,E) \in \Lambda : E=E_0 \}$ with the points of $C_0$ in each fibre;
\item
$\widehat{\Delta}_{\pm} \subset W^s \left( \widehat{A} \right) \cap W^u \left( \widehat{A} \right)$ is the set with base points $\{ (t,E) \in \Lambda : E=E_0 \}$ with the points of $C_0^{\pm}$, respectively, in each fibre;
\item
$\widehat{D}_+= \hat{\pi}^s \left( \widehat{\Delta}_+ \right) \subset \widehat{A}, \quad \widehat{D}_-= \hat{\pi}^u \left( \widehat{\Delta}_- \right) \subset \widehat{A}.$
\end{itemize}
Write
\begin{equation}
\hat{\pi}_+ = \left. \hat{\pi}^s \right|_{\widehat{\Delta}_+}, \quad \hat{\pi}_- = \left. \hat{\pi}^u \right|_{\widehat{\Delta}_-}
\end{equation}
and define the following return maps:
\begin{itemize}
\item
$T_0 :  \{ E=E_0 \} \to \widehat{M}_0$ is the return map of $\phi^t_K$ forward in time to $\{ t=0, \, E=E_0 \}$: i.e. if $z_0 \in \{ E=E_0 \}$ then there is a unique $t_0 \in [0,1)$ such that $T_0 (z_0) = \phi^{t_0}_K (z_0) \in \widehat{M}_0$;
\item
$T_1: \widehat{\Delta}_- \to \widehat{\Delta}_+$ is the return map of $\phi^t_K$ to $\widehat{\Delta}_+$.  
\item
$T_+ : \widehat{A} \cap \{ E=E_0\} \to \widehat{D}$ is the return map of $\phi^t_K$ to $\widehat{D}$.
\end{itemize}
Finally, define the map $\hat{s} : \widehat{D} \to \hat{s} \left(\widehat{D} \right)$ by 
\begin{equation}\label{eq_modscatmapforflowexpanded1}
\hat{s} = T_+ \circ \phi^{- N_+}_K \circ \hat{\pi}_+ \circ T_1 \circ \hat{\pi}_-^{-1} \circ \phi_K^{-N_-}.
\end{equation}
Observe that $\widehat{D}$, $\widehat{D}_{\pm}$ and $\widehat{\Delta}_{\pm}$ are topological cylinders in $\widehat{M}$. We will show:
\begin{enumerate}[(i)]
\item \label{item_exactsymplecticityfact1}
$\hat{s} : \widehat{D} \to \hat{s} \left(\widehat{D} \right)$ is an exact symplectic diffeomorphism onto its image $\hat{s} \left(\widehat{D} \right)$, which contains an essential curve; 
\item \label{item_exactsymplecticityfact2}
If we glue the boundaries of $D$ by identifying points $x$ on one boundary with points $f(x)$ on the other then
\begin{equation}
\left. T_0 \right|_{\widehat{D}} : \widehat{D} \to \{ (z, (t,E)): z \in D, \, t=0, \, E=E_0 \} \subset \widehat{M}_0
\end{equation}
is an exact symplectic diffeomorphism. 
\item \label{item_exactsymplecticityfact3}
The modified scattering map $\tilde{s}$ can be written $\tilde{s} = T_0 \circ \hat{s} \circ \left( \left. T_0 \right|_{\widehat{D}} \right)^{-1}$.
\end{enumerate}
These three facts together complete the proof of the lemma. 

In order to prove \eqref{item_exactsymplecticityfact1}, notice first that the time-shifts $\phi_K^{- N_{\pm}}$ and the return maps $T_1, T_+$ are exact symplectic diffeomorphisms between cylinders since $\phi^t_K$ is a Hamiltonian flow. We must show that $\hat{\pi}_{\pm}$ are exact symplectic diffeomorphisms between cylinders. We pass to Fenichel coordinates $((q,p), (t, E), \alpha, \beta)$ in a neighbourhood $\U$ of $\widehat{A}$ (see the proof of Lemma 9 for a more complete description of Fenichel coordinates). Here $((q,p), (t,E))$ are coordinates on $\widehat{A}$ where $(q,p)$ are coordinates on $A_*$ and $(t,E) \in \Lambda$, and $\alpha, \, \beta$ are coordinates in the hyperbolic directions. In particular, if $z \in \widehat{A}$ then the corresponding leaves of the strong stable and strong unstable foliations are given by
\begin{align}
W^s(z) ={}& \left\{ ((q,p), (t, E), \alpha, \beta) : ((q,p), (t, E)) = z, \, \alpha = 0 \right\}, \\
W^u(z) ={}& \left\{ ((q,p), (t, E), \alpha, \beta) : ((q,p), (t, E)) = z, \, \beta = 0 \right\}.
\end{align}
The local stable manifold $W^s_{loc} \left( \widehat{A} \right)$ (respectively local unstable manifold $W^u_{loc} \left( \widehat{A} \right)$) is the set of points whose forward orbits (resp. barckward orbits) tend toward $\widehat{A}$ without ever leaving the neighbourhood $\U$ of $\widehat{A}$. Notice that the restriction
\begin{equation}
\left. \hat{\pi}^{s,u} \right|_{W^{s,u}_{loc}(\widehat{A})} : W^{s,u}_{loc} \left(\widehat{A} \right) \to \widehat{A}
\end{equation} 
of the stable and unstable holonomy maps to the local stable and unstable manifolds, in Fenichel coordinates, are simply projections onto the $((q,p),(t,E))$ variables. Now choose $N_{\pm}$ sufficiently large so that
\begin{equation}
\widehat{\Delta}_+ \subset W^s_{loc} \left( \widehat{A} \right), \quad \widehat{\Delta}_- \subset W^u_{loc} \left( \widehat{A} \right). 
\end{equation}
It follows that the maps $\hat{\pi}_{\pm}$ are just the identity mapping in Fenichel coordinates, and therefore are exact symplectic diffeomorphisms between cylinders. We have thus proved \eqref{item_exactsymplecticityfact1}. 

The proof of \eqref{item_exactsymplecticityfact2} is immediate from the facts that the return map of a Hamiltonian flow is exact symplectic, that $D$ is a fundamental domain of $f$, and that the return map $\hat{f} : \widehat{M}_0 \to \widehat{M}_0$ coincides with the billiard map. 

Finally, let us prove \eqref{item_exactsymplecticityfact3}. In this part of the proof, we consider the billiard map $f$ and holonomy maps $\pi_{\pm}$ as maps on (subsets of) $\widehat{M}_0$ in the natural way without changing notation. From the definitions we have
\begin{equation}
\left. \phi^m_K \right|_{\widehat{M}_0} = f^m
\end{equation}
for each $m \in \mathbb{Z}$. We prove
\begin{equation} \label{eq_mseseqns1}
\left. f^{- N_-} \right|_D = T_0 \circ \phi^{-N_-}_K \circ \left( \left. T_0 \right|_{\widehat{D}} \right)^{-1}
\end{equation}
\begin{equation}\label{eq_mseseqns2}
\pi_{\pm} = T_0 \circ \hat{\pi}_{\pm} \circ \left( \left. T_0 \right|_{\widehat{\Delta}_{\pm}} \right)^{-1}
\end{equation}
\begin{equation}\label{eq_mseseqns3}
\left. f^{N_+ + N_-} \right|_{\Delta_-} = T_0 \circ T_1 \circ \left( \left. T_0 \right|_{\widehat{\Delta}_-} \right)^{-1}
\end{equation}
\begin{equation}\label{eq_mseseqns4}
\left. f^{-N_+} \right|_{D_+} = T_0 \circ \phi^{-N_+}_K \circ \left( \left. T_0 \right|_{\widehat{D}_+} \right)^{-1}
\end{equation}
\begin{equation}\label{eq_mseseqns5}
\left. f^{n} \right|_{s_{B_*}(D)} = T_0 \circ T_+ \circ \left( \left. T_0 \right|_{\phi^{-N_+}_K \left(\widehat{D}_+ \right)} \right)^{-1}.
\end{equation}
Combining \eqref{eq_mseseqns1}, \eqref{eq_mseseqns2}, \eqref{eq_mseseqns3}, \eqref{eq_mseseqns4}, \eqref{eq_mseseqns5} with \eqref{eq_modscatmapexpanded1} and \eqref{eq_modscatmapforflowexpanded1} yields \eqref{item_exactsymplecticityfact3} and completes the proof of the lemma. 

Let us prove \eqref{eq_mseseqns1}. Let $z_0 \in D$. Then there is a unique $t_0 \in \mathbb{T}$ such that 
\begin{equation}
\phi^{-t_0}_K \left( z_0, 0, E_0 \right) = \left( \left. T_0 \right|_{\widehat{D}} \right)^{-1} \left( z_0, 0, E_0 \right) \in \widehat{D}. 
\end{equation}
Then, since $N_- \in \mathbb{Z}$, we have
\begin{align}
T_0 \circ \phi^{-N_-}_K \circ \left( \left. T_0 \right|_{\widehat{D}} \right)^{-1} (z_0, 0, E_0) ={}& T_0 \circ \phi^{-N_-  - t_0}_K (z_0, 0, E_0) \\
={}& \phi^{-N_-  }_K (z_0, 0, E_0) = f^{-N_-  } (z_0, 0, E_0).
\end{align}

Next, we prove \eqref{eq_mseseqns2}. We prove the formula for $\pi_+$ as the formula for $\pi_-$ is analogous. Let $z_1 \in \Delta_+$. Then there is a unique $t_1 \in \mathbb{T}$ such that
\begin{equation}
\phi^{-t_1}_K \left( z_1, 0, E_0 \right) = \left( \left. T_0 \right|_{\widehat{\Delta}_{\pm}} \right)^{-1} \left( z_1, 0, E_0 \right) \in \widehat{\Delta}_+.
\end{equation}
Since $\Delta_+ \subset W^s_{loc} (A_*)$, $z_1$ is of the form
\begin{equation}
z_1 = \left( (q_1, p_1), (0, E_0), 0, \beta_1 \right)
\end{equation}
in Fenichel coordinates, where $(q,p)$ denote coordinates on the cylinder $A_*$. Therefore $\pi_+ (z_1, 0, E_0) = ((q_1,p_1),(0,E_0))$. Now, 
\begin{align}
\hat{\pi}_+ \circ \left( \left. T_0 \right|_{\widehat{\Delta}_+} \right)^{-1} (z_1, 0, E_0) ={}& \hat{\pi}^s_{\widehat{\Delta}_+} \circ \phi^{-t_1}_K \left( (q_1,p_1), (0, E_0), 0, \beta_1 \right) \\
={}& \phi^{-t_1}_K \circ \hat{\pi}^s_{\phi^{t_1}_K \left(\widehat{\Delta}_+ \right)} \left( (q_1,p_1), (0, E_0), 0, \beta_1 \right) \\
={}& \phi^{-t_1}_K ((q_1, p_1), (0, E_0)).
\end{align}
It thus follows that
\begin{align}
T_0 \circ \hat{\pi}_{+} \circ \left( \left. T_0 \right|_{\widehat{\Delta}_{\pm}} \right)^{-1} (z_1, 0, E_0) ={}& T_0 \circ \phi^{-t_1}_K ((q_1, p_1), (0, E_0)) \\
 ={}& ((q_1, p_1), (0, E_0)) = \pi_+ (z_1, 0, E_0)
\end{align}
which is \eqref{eq_mseseqns2}.

Notice that \eqref{eq_mseseqns3} and \eqref{eq_mseseqns4} are analogous to \eqref{eq_mseseqns1}, so it remains to prove \eqref{eq_mseseqns5}. Let $z_2 \in s_{B_*} (D)$, and let $n_2 = n(z_2) \in \mathbb{N}$. Define
\begin{equation}
\widetilde{D} = \phi^{-n_2}_K \left( \widehat{D} \right)
\end{equation}
and observe that 
\begin{equation}\label{eq_intersectionoftwofundomains}
(z_2,0,E_0) \in T_0 \left(\widetilde{D} \right) \cap T_0 \left( \phi^{-N_-}_K \left( \widehat{D}_+ \right) \right).
\end{equation}
Therefore, since $T_+$ is a return map, it can be seen that
\begin{equation}\label{eq_tpluscomposed}
T_+ = T_+ \circ  \left( \left. T_0 \right|_{\widetilde{D}} \right)^{-1} \circ \left. T_0 \right|_{\phi^{-N_-}_K \left( \widehat{D}_+ \right)}
\end{equation}
at points of the type \eqref{eq_intersectionoftwofundomains}. Now, there is a unique $t_2 \in \mathbb{T}$ such that
\begin{equation}
\phi^{-t_2}_K (z_2, 0, E_0) = \left( \left. T_0 \right|_{\widetilde{D}} \right)^{-1} (z_2, 0 , E_0) \in \widetilde{D}. 
\end{equation}
Using \eqref{eq_tpluscomposed}, it follows that
\begin{align}
T_0 \circ T_+ \circ \left( \left. T_0 \right|_{\phi^{-N_+}_K \left(\widehat{D}_+ \right)} \right)^{-1} (z_2, 0 , E_0) ={}& T_0 \circ T_+ \circ \left( \left. T_0 \right|_{\widetilde{D}} \right)^{-1} (z_2, 0 , E_0) \\
={}& T_0 \circ T_+ \circ  \phi^{-t_2}_K (z_2, 0 , E_0) \\
={}& T_0 \circ  \phi^{n_2-t_2}_K (z_2, 0 , E_0) \\
={}& \phi^{n_2}_K (z_2, 0 , E_0) = f^{n_2} (z_2, 0, E_0).
\end{align}

\end{proof}

\bibliographystyle{abbrv}
\bibliography{arnold_diffusion_refs} 
\end{document}